\documentclass[12pt]{iopart}

\usepackage{subcaption}

\usepackage{xcolor}

\usepackage{amsfonts}
\usepackage{amsthm}

\usepackage{amssymb}

\usepackage{cleveref}

\newcommand{\bp}{{\bowtie}}

\newcommand{\implies}{{\Longrightarrow}}

\newcommand{\edge}{\rightarrow}

\newcommand{\dpath}{\rightsquigarrow}

\newcommand{\xset}{\mathbb{X}}

\newcommand{\yset}{\mathbb{Y}}

\newcommand{\permut}{\sigma}

\newcommand{\type}[1]{\mathcal{T}_\mathcal{#1}}
\newcommand{\tset}{T}

\theoremstyle{definition}

\newtheorem{defi}{Definition}[section]
\crefname{defi}{Definition}{Definitions}
\newtheorem{exmp}[defi]{Example}
\crefname{exmp}{Example}{Examples}
\theoremstyle{remark}

\newtheorem{remark}[defi]{Remark}
\crefname{remark}{Remark}{Remarks}
\theoremstyle{plain}

\newtheorem{theorem}[defi]{Theorem}
\crefname{theorem}{Theorem}{Theorems}
\newtheorem{lemma}[defi]{Lemma}
\crefname{lemma}{Lemma}{Lemmas}
\newtheorem{corollary}[defi]{Corollary}
\crefname{corollary}{Corollary}{Corollaries}
\usepackage{tikz}
\usetikzlibrary{shapes}
\usetikzlibrary{calc}
\usetikzlibrary{fit}

\newcommand\DoubleLine[7][4pt]{%
	\path(#2)--(#3)coordinate[at start](h1)coordinate[at end](h2);
	\draw[#4]($(h1)!#1!90:(h2)$)-- node [auto=left] {#5} ($(h2)!#1!-90:(h1)$); 
	\draw[#6]($(h1)!#1!-90:(h2)$)-- node [auto=right] {#7} ($(h2)!#1!90:(h1)$);
}

\makeatletter
\DeclareRobustCommand{\text}{%
	\ifmmode\expandafter\text@\else\expandafter\mbox\fi}
\let\nfss@text\text
\def\text@#1{{\mathchoice
		{\textdef@\displaystyle\f@size{#1}}%
		{\textdef@\textstyle\f@size{#1}}%
		{\textdef@\textstyle\sf@size{#1}}%
		{\textdef@\textstyle \ssf@size{#1}}%
		\check@mathfonts
	}%
}
\def\textdef@#1#2#3{\hbox{{%
			\everymath{#1}%
			\let\f@size#2\selectfont
			#3}}}
\makeatother

\begin{document}
	\title[An in-reachability based classification of invariant synchrony patterns]{An in-reachability based classification of invariant synchrony patterns in weighted coupled cell networks}
	
	\author{P M Sequeira$^1$, J P Hespanha$^2$ and	A P Aguiar$^1$}
	
	\address{$^1$ SYSTEC-ARISE \& Department of Electrical and Computer Engineering, Faculdade de Engenharia, Universidade do Porto, 
		Rua Dr. Roberto Frias, 4200-465 Porto, Portugal}
	\address{$^2$ Department of Electrical and Computer Engineering, 
		University of California, Santa Barbara, CA, USA}

	\eads{\mailto{pedro.sequeira@fe.up.pt}, \mailto{hespanha@ece.ucsb.edu}, \mailto{pedro.aguiar@fe.up.pt}}

	\begin{abstract}
		This paper presents an in-reachability based classification of invariant synchrony patterns in Coupled Cell Networks (CCNs). These patterns are encoded through partitions on the set of cells, whose subsets of synchronized cells are called colors.\\ 
		We study the influence of the structure of the network in the qualitative behavior of invariant synchrony sets, in particular, with respect to the different types of (cumulative) in-neighborhoods and the in-reachability sets.
		This motivates the proposed approach to classify the partitions into the categories of strong, rooted and weak, according to how their colors are related with respect to the connectivity structure of the network.\\
		Furthermore, we show how this classification system acts under the partition join ($ \vee $) operation, which gives us the synchrony pattern that corresponds to the intersection of synchrony sets.
	\end{abstract}
	\ams{34A34, 34C45, 05C20, 05C40}
	\noindent{\it Keywords\/}: coupled cell networks, synchrony, connectivity\\
	\submitto{\NL at 17 of February of 2023. Revised at 20 of October of 2023.}
	\maketitle
\section{Introduction}
\textit{Networks} are structures that describe systems with multiple components, called \textit{cells}. These cells can be connected through \textit{edges}, which encode how one cell affects another. In general, these edges can be directed or undirected and they can have weights in order to parameterize their interaction.\\
Networks are ubiquitous structures, both in the natural world and in engineering applications. Some examples are for instance the brain, the internet, the electric grid and electronic circuits in general, food webs and the spread of a virus in a pandemic.\\
In order to study these types of systems, the theory of \textit{coupled cell networks} (CCN) was first formalized in \cite{stewart2003symmetry,golubitsky2005patterns,golubitsky2006nonlinear}. In \cite{sequeira2021commutative}, this was generalized for networks with weighted connections and with arbitrary edges and edge types, and it is the formalism that we adopt in this work.\\ 
In theory of CCNs the concept of \textit{admissible function} is defined such that a function $ f $ is admissible in a network if it satisfies certain minimal properties that allow it to be a valid modeling of some (discrete-time/continuous-time) dynamical system $ \mathbf{x}^{+}/\dot{\mathbf{x}}=f(\mathbf{x}) $ on that network. \\
In this work we study general equality-based invariant synchrony patterns, which are represented through partitions on the set of cells of a network. 
Much work has been done regarding balanced partitions, which represent patterns of synchrony that are invariant under any admissible function on the network of interest. Although balanced partitions represent a very important subclass of invariant synchrony patterns with strong properties, it is possible for other invariant patterns to be present in a network.\\
Consider for instance the subset of admissible functions such that a cell becomes insensitive to cells that are on the same state. Such a system is, consequently, always insensitive to self-loops. 
This happens, for instance, in the Kuramoto model \cite{arenas2008synchronization,dorfler2014synchronization,rodrigues2016kuramoto}. This property leads to the study of exo-balanced partitions \cite{aguiar2018synchronization,neuberger2020invariant,aguiar2021synchrony}, which is a larger class of partitions than the balanced ones.\\
For this reason, we consider arbitrary subsets of admissible functions $ F$ and show that the set of partitions $ L_{F} $ that describe synchrony patterns that are invariant under $ F $ always form lattices. Furthermore, we show that these lattices have similar properties to the lattices of balanced partitions $ \Lambda_{\mathcal{G}} $ \cite{stewart2007lattice}. In particular, these lattices share the same join operation $ \vee $ and have a $ cir_F $ function associated with them.
For a computational approach to the study of lattices of invariant synchrony through the eigenvalues of network adjacency matrices refer to \cite{aguiar2014lattice,moreira2015special,kamei2021reduced}.
\\
The coarsest invariant refinement ($ cir $), was first developed in \cite{aldis2008polynomial} as a polynomial-time algorithm that finds the maximal element of the lattice balanced partitions. In \cite{neuberger2020invariant} it was noted that this algorithm does more than just finding the maximal balanced partition. In fact, given any input partition, it outputs the greatest balanced partition that is finer ($ \leq $) than the input one. Therefore the maximal balanced partition is given by $ cir(\type{})$, where $ \type{} $ is the partition of cell types.\\
In this work, we show that the concept of $ cir $, as a function, is not specific to balanced partitions and that every $ F $-invariant lattice $ L_F $ has an associated $ cir_F $ function.\\
In summary, we show that the fundamental properties of the lattices of synchrony $ L_F $ are having the trivial partition $ \bot $ as their minimal element and their join operation match the partition join $ \vee $, which corresponds to the intersection of synchrony sets. The lattice of balanced partitions $ \Lambda_{\mathcal{G}} $ is seen as a particular case of a lattice of synchrony and consequently has this structure. Furthermore, for networks $ \mathcal{G} $ and $\mathcal{Q} $ related by the quotient $ \mathcal{Q} = \mathcal{G}/\bp $ with $ \bp \in \Lambda_{\mathcal{G}} $, we can similarly relate their associated subsets of admissible functions $ F_{\mathcal{Q}} = F_{\mathcal{G}}/\bp $. Their responding lattices $ L_{F_{\mathcal{G}}}, L_{F_{\mathcal{Q}}} $ are also related by $ L_{F_{\mathcal{Q}}} = L_{F_{\mathcal{G}}}/ \bp $ in a completely analogous manner to the known result $ \Lambda_{\mathcal{Q}} = \Lambda_{\mathcal{G}}/\bp $.\\
The main goal of this work is to study the influence that the connectivity structure of a network has on its synchrony patterns. In particular, we focus on the different types of (cumulative) in-neighborhoods and the in-reachability sets. For that purpose, we explore how the connectivity structure of a CCN affects an admissible dynamical system in a network and show that differences in this structure can lead to qualitatively different behaviors of general equality-based invariant synchrony patterns.
This motivates the classification of partitions into the categories of strong, rooted and weak, according to their relation to the \textit{strongly connected components} (SCC) of the network.\\
We study how this classification scheme acts under the join $ (\vee) $ operation and the relation between the classes of partitions $ \mathcal{A} $ and $ \mathcal{A}/\bp $ on $ \mathcal{G} $ and $ \mathcal{Q} $, respectively. 
Under certain assumptions these results are greatly simplified. The tamest such assumptions that we found were the concepts of $ \mathcal{R}^{-} $-matching and $ \mathcal{R}^{-} $-invariance, which also relate a partition to the in-reachability sets of their associated networks.
\\
In \cref{sec:new_formalism} we summarize the formalism for general weighted CCNs.\\
In \cref{sec:invariant_synchrony} we provide the necessary background regarding general equality-based invariant synchrony patterns in CCNs.\\
In \cref{sec:net_connect} we clarify how the connectivity structure of a network affects its dynamics. This motivates the study of the network according to its in-reachability sets.\\
In \cref{sec:part_classif}, motivated by the previous observations regarding the role of connectivity, we define a classification of partitions of invariant synchrony into strong, rooted and weak types.
\section{Weighted multi-edge formalism} \label{sec:new_formalism}
We start by describing the formalism for general networks with arbitrary, weighted connections.
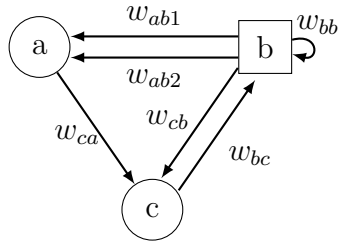
\begin{figure}[h]
	\centering
	\begin{tikzpicture}[
node1/.style = {circle,minimum size=23,draw},
node2/.style = {circle,minimum size=23,draw,fill=white!75!black},
node3/.style = {circle,minimum size=23,draw,fill=white!50!black},
noderect/.style = {rectangle,minimum size=20,draw},
edge1/.style = {>=latex,thick},
edge2/.style = {>=latex,thick,blue},
edge3/.style = {>=latex,thick,red}
]
\node[node1] at (0,{2.5*sqrt(3)/2})(n1){a};
\node[noderect] at (3,{2.5*sqrt(3)/2})(n2){b};
\node[node1] at (1.5,0)(n3){c};

\DoubleLine{n1}{n2}{<-,edge1}{$ w_{ab1} $}{<-,edge1}{$ w_{ab2} $}
\DoubleLine{n2}{n3}{<-,edge1}{$ w_{bc} $}{->,edge1}{}

\node (weightcb) at ($(n2)!0.4!(n3) + (-0.7,-0.1)$) {$ w_{cb} $};

\draw [->,edge1](n1) -- (n3);

\node (weightca) at ($(n1)!0.4!(n3) + (-0.1,-0.3)$) {$ w_{ca} $};

%\DoubleLine{n1}{n3}{<-,edge1}{}{->,edge1}{}

%\draw [->,edge1](n1) -- (n2);

%\draw [->,edge1] (n1) edge[loop left,looseness=5] (n1);
\draw [->,edge1] (n2) edge[loop right,looseness=5] (n2);

\node (nbnb) at ($ (n2) + (0.7,0.3)$) {$ w_{bb} $};

%\draw [->,edge1] (n3) edge[loop left,looseness=5] (n3);

%\node (minus123) at ($(n1)!0.4!(n2) + (0,0.3)$) {(-)};

%\DoubleLine{n1}{n2}{<-,edge1}{}{->,edge1}{}
%\DoubleLine{n1}{n3}{<-,edge1}{}{->,edge1}{}
%\DoubleLine{n2}{n3}{<-,edge1}{}{->,edge1}{}

\end{tikzpicture} 
	\caption{A weighted coupled cell network.}
	\label{fig:general_net_showcase}
\end{figure}
The schematic in \cref{fig:general_net_showcase} illustrates the object of study. It represents a network containing three cells $ \mathcal{C} = \{a,b,c\}$. Furthermore, cells $ a $ and $ c $ are both represented with a circle, meaning that they are objects of the same type, while cell $ b $ is of a different type. We usually index the set of cell types with integers, such as $ \tset = \{1,2\} $ (e.g., identify ``circle'' with $ 1 $ and ``square'' with $ 2 $). Our goal is to study the global dynamical system associated with a network in which its cells are also dynamical systems in their own right. This means that for each cell type $ i\in \tset $ we have an associated state set $ \xset_{i} $ and output set $ \yset_i $, which we use to build the domains and co-domains of the associated dynamical systems. A cell can influence the dynamical evolution of other cells in the network, which we represent by a directed edge, where the receiving cell is the affected one. Each interaction can be arbitrarily parameterized by associating a weight/label on the corresponding directed edge. It is possible for a cell to affect another in multiple ways, as seen in \cref{fig:general_net_showcase}, where cell $ b $ affects cell $ a $ through an interaction parameterized by $ w_{ab1} $  and another one parameterized by $ w_{ab2} $. Using the algebraic structure of the commutative monoid described in the following subsection, we can represent this by a single interaction parameterized by $  w_{ab1}\|w_{ab2} $, which is a common notation used in electrical circuit theory.
\subsection{Commutative monoids}
\label{subsec:commutative_monoids}
A commutative monoid is a set equipped with a binary operation (usually denoted $ + $) such that it is commutative and associative. Furthermore, it has one identity element (usually denoted $ 0 $). This is the simplest algebraic structure that can be used to describe arbitrary finite parallels of edges. Note that associativity and commutativity, together, are equivalent to the invariance to permutations property.\\
In this work, the ``sum'' operation is denoted by $ \| $, with the meaning of ``adding in parallel''. In this context, the zero element of a monoid should be interpreted as ``no edge''.
We do not require the existence of inverse elements. That is, given an edge, there does not need to exist another one such that the two in parallel act as ``no edge". This is the reason for the use of monoids instead of the algebraic structure of groups.
\subsection{Multi-indexes}
\label{subsec:multi_indexes}
A multi-index is an ordered $ n $-tuple of non-negative integers (indexes). That is, an element of $ \mathbb{N}_0^n $. In particular, $ \mathbf{0}_n $ represents the tuple of $ n $ zeros. 
We denote the multi-indexes with the same notation we use for vectors, using bold, as in $ \mathbf{k} = \left[k_1, \ldots. k_n\right]^{\top}$. 
We often specify the tupleness $ n $ of a multi-index $ \mathbf{k} $ indirectly, by using $ \mathbf{k}\geq \mathbf{0}_n  $ in order to denote $ \mathbf{k}\in\mathbb{N}_0^n $.
\subsection{CCN formalism} \label{subsec:CCN_formalism}
According to \cite{sequeira2021commutative}, a general weighted coupled cell network is given by the following definition.
\begin{defi}[Definition 2.1. in \cite{sequeira2021commutative}]]\label{defi:CCN}
	A network $ \mathcal{G} $ consists of a set of cells $ \mathcal{C}_{\mathcal{G}} $, where each cell has a type, given by a set $ \tset $ according to $ \type{G}\colon\ \mathcal{C}_{\mathcal{G}}\to\tset $ and has an $ \vert\mathcal{C}_{\mathcal{G}}\vert\times\vert\mathcal{C}_{\mathcal{G}}\vert $ in-adjacency matrix $ M_{\mathcal{G}} $. The entries of $ M_{\mathcal{G}} $ are elements of a family of commutative monoids $ \{\mathcal{M}_{ij}\}_{i,j\in\tset} $ such that $ \left[M_{\mathcal{G}}\right]_{cd} = m_{cd}\in\mathcal{M}_{ij} $, for any cells $ c,d \in \mathcal{C}_{\mathcal{G}} $ with types $ i =\type{G}(c) $, $ j =\type{G}(d) $. 
	\hfill$ \square $
\end{defi}
We often have the set of cell types be of the form $ \tset=\{1, \ldots, \vert\tset\vert\} $ so that it is simple to index.
For each commutative monoid $ \mathcal{M}_{ij} $ we denote its ``zero'' element as $ 0_{ij} $. 
\begin{remark}
	The subscripts $ _\mathcal{G} $ are omitted when the network of interest is clear from context. 
	\hfill$ \square $
\end{remark}
\subsection{Admissibility}\label{subsec:Gadmissibility}
A function $ f\colon\xset\to\yset $ is said to be \textit{admissible} on a network if it respects the minimal properties that we expect from it in order to be a plausible modeling of the dynamics $ \dot{\mathbf{x}}/\mathbf{x}^{+} = f(\mathbf{x}) $ or some measurement function $ \mathbf{y} = f(\mathbf{x}) $ on the network. In particular, it has to describe some \textit{first-order property}. That is, it models something that, when evaluated at cell, depends on the state of that cell and its in-neighbors. This does not mean that everything on a network has to (or can) be defined by such a function. For instance, the second derivative or the two-step evolution of a dynamical system on a network are not of this form. Those functions are second-order in the sense that, when evaluated on a cell, they depend on the states of that cell, together with the states of the cells in its first and second in-neighborhoods (neighbor of neighbor). Such second-order functions are, however, fully defined from the original first-order functions.\\
We construct admissible functions by using mathematical objects called \textit{oracle components}, first introduced in \cite{sequeira2021commutative} and then simplified in \cite{sequeira2022decomposition}.
An oracle component is a mathematical object that describes how cells of a given type respond to arbitrary finite in-neighborhoods. It completely separates the modeling of the behavior of cells from the particular network that contains the cells of interest.\\
Consider the simple network of \fref{fig:edge_merginga}, (which could be part of a larger network) consisting of cell $ c $ and its in-neighborhood. 
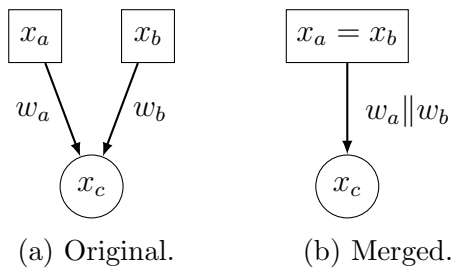
\begin{figure}[h]
	\centering
	\begin{subfigure}[t]{0.23\textwidth}
		\centering
		\begin{tikzpicture}[
node1/.style = {circle,minimum size=23,draw},
node2/.style = {circle,minimum size=23,draw,fill=white!75!black},
node3/.style = {circle,minimum size=23,draw,fill=white!50!black},
noderect/.style = {rectangle,minimum size=20,draw},
edge1/.style = {>=latex,thick},
edgedash/.style = {>=latex,thick,dashed},
edge2/.style = {>=latex,thick,blue},
edge3/.style = {>=latex,thick,red}
]
\node[noderect] at (-0.75,2)(n1){$ x_a $};
\node[noderect] at (+0.75,2)(n2){$ x_b $};
\node[node1] at (0,0)(n3){$ x_c $};

\draw [->,edge1](n1) -- (n3);
\draw [->,edge1](n2) -- (n3);

\node (w13) at ($(n1)!0.5!(n3) + (-0.4,-0.0)$) {$ w_a $};
\node (w23) at ($(n2)!0.5!(n3) + (+0.4,-0.0)$) {$ w_b $};

%\DoubleLine{n1}{n2}{<-,edge1}{}{->,edge1}{}
%\DoubleLine{n1}{n3}{<-,edge1}{}{->,edge1}{}
%\DoubleLine{n2}{n3}{<-,edge1}{}{->,edge1}{}

\end{tikzpicture} 
		\caption{Original.}
		\label{fig:edge_merginga}
	\end{subfigure}
	\begin{subfigure}[t]{0.23\textwidth}
		\centering
		\begin{tikzpicture}[
node1/.style = {circle,minimum size=23,draw},
node2/.style = {circle,minimum size=23,draw,fill=white!75!black},
node3/.style = {circle,minimum size=23,draw,fill=white!50!black},
noderect/.style = {rectangle,minimum size=20,draw},
edge1/.style = {>=latex,thick},
edgedash/.style = {>=latex,thick,dashed},
edge2/.style = {>=latex,thick,blue},
edge3/.style = {>=latex,thick,red}
]

\node[noderect] at (0,2)(n3){$ x_a=x_b $};
\node[node1] at (0,0)(n12){$ x_c $};

\draw [->,edge1](n3) -- (n12);
\node (w13) at ($(n3)!0.5!(n12) + (+0.8,-0.0)$) {$ w_a \| w_b $};

%\DoubleLine{n1}{n2}{<-,edge1}{}{->,edge1}{}
%\DoubleLine{n1}{n3}{<-,edge1}{}{->,edge1}{}
%\DoubleLine{n2}{n3}{<-,edge1}{}{->,edge1}{}

\end{tikzpicture} 
		\caption{Merged.}
		\label{fig:edge_mergingb}
	\end{subfigure}
	\centering
	\caption{Edge merging.}
	\label{fig:edge_merging}
\end{figure}
We have cell types $ \tset =\{1,2\} $ which represent ``circle" and ``square" cells, respectively. In order to define functions on the cells we associate with them the state sets $ \xset_{1},\xset_{2} $ and the output sets $ \yset_{1},\yset_{2} $ according to their respective type.\\
We consider that the input received by a cell is independent of how we draw the network, that is, from the point of view of cell $ c $, there would be no difference if cell $ b $ was at the left of cell $ a $. Then for a function $ \hat{f}_1 $ acting on cells of type $ 1 $, we would expect that
\begin{eqnarray*}
\hat{f}_1\left(
x_c;
\left[\begin{array}{c} w_a\\ w_b \end{array} \right]
,
\left[\begin{array}{c} x_a\\ x_b \end{array} \right]
\right) 
=
\hat{f}_1\left(
x_c;
\left[\begin{array}{c} w_b\\ w_a \end{array} \right]
,
\left[\begin{array}{c} x_b\\ x_a \end{array} \right]
\right)
, 
\end{eqnarray*}
for $ x_c\in\xset_{1} $, $ x_a,x_b\in\xset_{2} $ and $ w_a,w_b\in \mathcal{M}_{12}$. 
Moreover, since cells $ a $ and $ b $ are of the same cell type (square) ($ \type{}(a) = \type{}(b) = 2 $), we expect that when they are in the same state ($ x_a=x_b =x_{ab}$), the total input received by cell $ c $ at that instant, is the same as if both edges originated from a single ``square" cell with that state, as in \fref{fig:edge_mergingb}. That is,
\begin{eqnarray*}
\hat{f}_1\left(
x_c;
\left[\begin{array}{c} w_a\\ w_b \end{array} \right]
,
\left[\begin{array}{c} x_{ab} \\ x_{ab} \end{array} \right]
\right) 
=
\hat{f}_1\left(
x_c;
w_a \| w_b, x_{ab}
\right). 
\end{eqnarray*}
Although this might look inconsistent since the domains look mismatched, the following definition formalizes it in a rigorous way. 
Finally, when $ \hat{f}_1 $ is evaluated at a cell it should depend only on the in-neighborhood of that cell. Therefore, if $ w_a = 0_{12} $, cell $ c $ should not be directly influenced by cell $ a $. That is,
\begin{eqnarray*}
\hat{f}_1\left(
x_c;
\left[\begin{array}{c} 0_{12}\\ w_b \end{array} \right]
,
\left[\begin{array}{c} x_a\\ x_b \end{array} \right]
\right) 
=
\hat{f}_1\left(x_c; w_b,
x_b\right). 
\end{eqnarray*}
These ideas are now formalized in the following definition.
\begin{defi}[Definition 2.2. in \cite{sequeira2022decomposition}] \label{defi:oracle}
	Consider the set of cell types $ \tset $, and some related sets $ \{\xset_j, \yset_j\}_{j\in\tset} $ together with a family of commutative monoids $ \{\mathcal{M}_{ij}\}_{j\in\tset} $, for a given fixed $ i\in\tset $. Take a function $ \hat{f}_i $ defined on
	\begin{eqnarray}
	\hat{f}_i\colon \xset_i\times\bigcup^\circ_{\mathbf{k} 
		\geq \mathbf{0}_{\vert\tset\vert}} \left( \mathcal{M}_i^{\mathbf{k}} \times \xset^{\mathbf{k}} \right)  \to \yset_i,
	\label{eq:oracle_domain_func}
	\end{eqnarray}
	where $ \bigcup\limits^\circ $ denotes the disjoint union and for multi-index $ \mathbf{k} $ we define $ \xset^{\mathbf{k}} := \xset_1^{k_1} \times \ldots \times \xset_{\vert\tset\vert}^{k_{\vert\tset\vert}} $ and $ \mathcal{M}_i^{\mathbf{k}} := \mathcal{M}_{i1}^{k_1} \times \ldots \times \mathcal{M}_{i\vert\tset\vert}^{k_{\vert\tset\vert}} $.\\
	The function $ \hat{f}_i $ is called an \textit{oracle component of type i}, if it has the following properties:
	\begin{enumerate}
		\item \label{thm:oracle_permut}
		If $ \permut $ is a \textit{permutation} matrix (of appropriate dimension), then
		\begin{eqnarray}
		\label{eq:oracle_permut}
		\hat{f}_i(x;\mathbf{w},\mathbf{x})
		=
		\hat{f}_i(x;\permut{\mathbf{w}},\permut{\mathbf{x}}),
		\end{eqnarray}
		where we assume, without loss of generality, that one can keep track of the cell types of each element of $ \permut{\mathbf{w}} $ and $ \permut{\mathbf{x}} $.
		\item
		\label{thm:oracle_merge}
		If the indexes $ j_1 $, $j_2 $ and $ j_{12} $ denote cells of type $ j \in\tset$, then
		\begin{eqnarray}
		\label{eq:oracle_merge}
		\hat{f}_i
		\left(x;
		\left[\begin{array}{c} w_{j_1} \| w_{j_2} \\ \mathbf{w} \end{array} \right]
		,
		\left[\begin{array}{c} x_{j_{12}} \\ \mathbf{x} \end{array} \right]
		\right)
		=
		\hat{f}_i
		\left(x; 
		\left[\begin{array}{c} w_{j_1} \\ w_{j_2} \\ \mathbf{w} \end{array} \right]
		,
		\left[\begin{array}{c} x_{j_{12}} \\ x_{j_{12}} \\ \mathbf{x} \end{array} \right]
		\right).
		\end{eqnarray}
		\item
		\label{thm:oracle_0_equal}
		If the index $ j $ denotes a cell of type $ j\in\tset $, then 
		\begin{eqnarray}
		\hat{f}_i
		\left(x;
		\left[\begin{array}{c} 0_{ij} \\ \mathbf{w} \end{array} \right]
		,
		\left[\begin{array}{c} x_{j}  \\ \mathbf{x} \end{array} \right]
		\right)
		=
		\hat{f}_i
		\left(x; 
		\mathbf{w},
		\mathbf{x}
		\right).
		\label{eq:oracle_0_equal}
		\end{eqnarray}
	\end{enumerate}
	\hfill$ \square $
\end{defi}
The disjoint union allows us to distinguish neighborhoods of different types. That is, even in the particular case of $ \xset_1 = \xset_2 $ and $ \mathcal{M}_{i1} = \mathcal{M}_{i2} $, we are able to differentiate the part of the domain associated with $ \mathcal{M}_{i1}^2\times\xset_1^2 $ from the one associated with $ \mathcal{M}_{i1}\times\mathcal{M}_{i2}\times \xset_1\times \xset_{2} $. A non-disjoint union, on the other hand, would merge these sets together.
\begin{remark}
	Note that the way the domain of $\hat{f}_i$ was defined allows us to deal with variable input set configurations. This is an equivalent, but cleaner, way of defining a family of functions, each on a different domain based on its specific input set, such that they are all connected by the self-consistency rules that we expect from them. This way, we can use a single function to describe what really matters to us, that is, describing how a cell is affected by its in-neighbors.
	\hfill$ \square $
\end{remark}
\begin{remark}
	As stated in \cref{thm:oracle_permut} of \cref{defi:oracle}, it is always assumed that given any weight $ w_c $ or state $ x_c $, we always know the cell type of the corresponding cell $ c $. Note that we can always do enough bookkeeping in order to ensure this. For instance, we can extend $ \hat{f}_i(x;\mathbf{w},\mathbf{x}) $ into $ \hat{f}_i(x;\mathbf{t},\mathbf{w},\mathbf{x}) $, where $ \mathbf{t} $ would be a vector that encodes the cell types associated with $ \mathbf{w},\mathbf{x} $. Then we would have $ \hat{f}_i(x;\mathbf{t},\mathbf{w},\mathbf{x})=\hat{f}_i(x;\permut{\mathbf{t}},\permut{\mathbf{w}},\permut{\mathbf{x}}) $ instead.\\
	This implicit bookkeeping means that we do not have to constrain $ \permut $ to preserve cell typing. That is, if we assume some canonical order of the cell types in the part of the domain $ \mathcal{M}_i^{\mathbf{k}} \times \xset^{\mathbf{k}} $ in \eref{eq:oracle_domain_func}, then we know the correct $ \mathbf{k}	\geq \mathbf{0}_{\vert\tset\vert} $ and can reorder the rows of $ \mathbf{w} $ and $ \mathbf{x} $ in $ \hat{f}_i(x;\mathbf{w},\mathbf{x}) $ appropriately.\\
	By considering invariance under general permutations, and not having to worry about preserving cell types or respecting some canonical ordering of cell types, we are always able to shift the cells of major interest to the top of the vectors, as in \eref{eq:oracle_merge},\eref{eq:oracle_0_equal}, regardless of the types of other cells.
	\hfill$ \square $
\end{remark}
We consider the function $ \mathcal{K} $ such that for a set of cells $ \mathbf{s} $, we have that $ \mathbf{k} = \mathcal{K}(\mathbf{s}) $ is the $ \vert \tset\vert $-tuple such that $ k_i $ is the number of cells in $ \mathbf{s} $ that are of type $ i\in\tset $. This allows us to pick the proper $ \mathbf{k} \geq \mathbf{0}_{\vert\tset\vert} $ in \eref{eq:oracle_domain_func} when we want to evaluate oracle components at a cell and its in-neighbors.\\
The \textit{oracle set} is the set of all $ \vert \tset\vert $-tuples of oracle components, such that each element of the tuple represents one of the types in $ \tset $. It is denoted as
\begin{eqnarray*}
\hat{\mathcal{F}}_{\tset} = \prod_{i\in \tset}\hat{\mathcal{F}}_{i}, 
\end{eqnarray*}
where $ \hat{\mathcal{F}}_{i} $ is the set of all oracle components of type $ i $. We are always implicitly assuming sets $ \{\xset_i, \yset_i\}_{i\in\tset} $ and commutative monoids $ \{\mathcal{M}_{ij}\}_{i,j\in\tset} $.
Modeling some aspect of a network that follows our assumptions is effectively choosing one of the elements of $ \hat{\mathcal{F}}_{\tset} $, which we call \textit{oracle functions}.
\begin{defi}[Definition 2.13. in \cite{sequeira2021commutative}]\label{defi:F_G_admissibility}
	Consider a network $ \mathcal{G} $ on a cell set $ \mathcal{C} $ with cell types in $ \tset $ according to the cell type partition $ \type{} $, and an in-adjacency matrix $ M $. Assume without loss of generality that the cells are ordered according to the cell types such that we can associate with the network a state $ \xset := \xset^{\mathbf{k}} $ and output $ \yset := \yset^{\mathbf{k}} $ sets, with $ \mathbf{k} = \mathcal{K}(\mathcal{C}) $.\\
	A function \mbox{$ f\colon\xset\to\yset $}, given as
	\begin{eqnarray*}
	f = (f_c)_{c\in\mathcal{C}},\qquad \text{with } f_c\colon\xset\to\yset_{i}, \qquad i = \type{}(c),
	\end{eqnarray*}
	is said to be \mbox{$ \mathcal{G} $\textit{-admissible}} if there is some oracle function \mbox{$ \hat{f}\in\hat{\mathcal{F}}_{\tset}$, $ \hat{f} =  (\hat{f}_i)_{i\in\tset}$} such that
	\begin{eqnarray}
	f_c(\mathbf{x}) = \hat{f}_{i}\left(x_c; \mathbf{m}_c^{\top}, \mathbf{x} \right),
	\label{eq:dynamics_dependence}
	\end{eqnarray}
	for $ \mathbf{x}\in\xset $, where $ x_c $ is the $ c^{th} $ coordinate of $ \mathbf{x} $ and $ \mathbf{m}_c $ is the $ c^{th} $ row of matrix $ M $. In this case we write $ f = \left.\hat{f}\right|_{\mathcal{G}} $.
	\hfill$ \square $
\end{defi}
The set of all \mbox{$ \mathcal{G} $-admissible} functions is denoted by $ \mathcal{F}_{\mathcal{G}} $. It can be thought of as the result of evaluating $ \hat{\mathcal{F}}_{\tset} $ at $ \mathcal{G} $, which can be written as $ \left.\hat{\mathcal{F}}_{\tset}\right|_{\mathcal{G}} $. The process of evaluating oracle functions at a network is not necessarily injective. There might be oracle functions $ \hat{f},\hat{g}\in\hat{\mathcal{F}}_{T} $ with $ \hat{f} \neq \hat{g} $ such that $ \left.\hat{f}\right|_{\mathcal{G}} = \left.\hat{g}\right|_{\mathcal{G}} $.\\
The next example makes explicit the relation between the connectivity graph of a network and how that constrains any possible admissible function that acts on it.
\begin{exmp}\label{exmp:admiss_funcs}	
	\Fref{fig:Gadmissibility} shows an example of a \textit{CCN} of three cells. We have cell types $ \tset = \{1,2\} $ which represent ``circle" and and ``square" cells, respectively.
	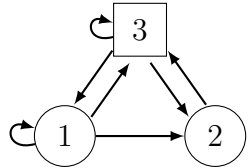
\begin{figure}[h]
		\centering
		\begin{tikzpicture}[
node1/.style = {circle,minimum size=23,draw},
node2/.style = {circle,minimum size=23,draw,fill=white!75!black},
node3/.style = {circle,minimum size=23,draw,fill=white!50!black},
noderect/.style = {rectangle,minimum size=20,draw},
edge1/.style = {>=latex,thick},
edge2/.style = {>=latex,thick,blue},
edge3/.style = {>=latex,thick,red}
]
\node[node1] at (0,0)(n1){1};
\node[node1] at (2,0)(n2){2};
\node[noderect] at (1,{sqrt(2)})(n3){3};

\DoubleLine{n1}{n3}{<-,edge1}{}{->,edge1}{}
\draw [->,edge1](n1) -- (n2);
\DoubleLine{n2}{n3}{<-,edge1}{}{->,edge1}{}

\draw [->,edge1] (n1) edge[loop left,looseness=5] (n1);
\draw [->,edge1] (n3) edge[loop left,looseness=5] (n3);

%\node (minus123) at ($(n1)!0.4!(n2) + (0,0.3)$) {(-)};

%\DoubleLine{n1}{n2}{<-,edge1}{}{->,edge1}{}
%\DoubleLine{n1}{n3}{<-,edge1}{}{->,edge1}{}
%\DoubleLine{n2}{n3}{<-,edge1}{}{->,edge1}{}

\end{tikzpicture} 
		\caption{Simple network with admissible functions that have the structure given by \eref{eq:func_match1},\eref{eq:func_match2},\eref{eq:func_match3}.}
		\label{fig:Gadmissibility}
	\end{figure}	
	This \textit{CCN} can be described by the in-adjacency matrix $ M $ 
	\begin{eqnarray} M = 
	\left[\begin{array}{ccc} 1 & 0 & 1 \\ 1 & 0 & 1 \\ 1 & 1 & 1 \end{array} \right]
	,
	\end{eqnarray}
	together with the cell type partition \mbox{$ \type{} = \{\{1,2\},\{3\}\} $.} This means that a suitable $ f\in\mathcal{F}_{\mathcal{G}} $ should have the following structure
	\begin{eqnarray}
	f_1(\mathbf{x}) &= \hat{f}_1(
	x_1;
	\left[\begin{array}{ccc} 1 & 0 & 1 \end{array} \right]^{\top}
	,
	\mathbf{x})\label{eq:func_match1}
	,\\
	f_2(\mathbf{x}) &= \hat{f}_1(
	x_2;
	\left[\begin{array}{ccc} 1 & 0 & 1 \end{array} \right]^{\top}
	,
	\mathbf{x})\label{eq:func_match2}
	,\\
	f_3(\mathbf{x}) &= \hat{f}_2(
	x_3;
	\left[\begin{array}{ccc} 1 & 1 & 1 \end{array} \right]^{\top}
	,
	\mathbf{x})\label{eq:func_match3},
	\end{eqnarray}
	for some $ \hat{f}\in\hat{\mathcal{F}}_{T}  $.
	\hfill$ \square $
\end{exmp}
\subsection{Comparison to the groupoid formalism}
This formalism differs from the one based on a groupoid of bijections in ways other than just extending it to allow arbitrary sets of edges in parallel with arbitrary weights.\\
A minor difference is that it does not assume any prior structure on the cell domains $ \xset_{i} $ and codomains $ \yset_{i} $ or any smoothness requirement in order for a function to be admissible. While this is quite usual and reasonable for continuous-time systems, it is not so for discrete-time systems. Such assumptions are application-dependent and therefore, should not be built in the general definitions. We can always particularize by adding such restrictions at a later point. Furthermore, we do not require the networks to obey the ``consistency condition'', which would force the different monoids $ \{\mathcal{M}_{ij}\}_{i,j\in\tset} $ to be pairwise disjoint. We only operate $ (\|) $ and compare elements within the same monoid $ \mathcal{M}_{ij} $, therefore, such condition is unnecessary.\\
The most important change is the definition of oracle components. The crucial difference from the previous formalism is that the notion of in-neighborhood equivalence is imposed on all (finite) in-neighborhoods of a cell, and not just the ones that might appear in our particular network of interest, according to its size. In this sense, it is more restrictive that the original formalism, however, we argue that this is a very reasonable restriction. In particular, from a physical point of view, the existence of function that describes how a cell behaves under arbitrary (finite) in-neighborhoods is not too much to ask for. No matter how large, complicated, or even nonsensical is the way that we connect a physical system, there is always an expression that describes its evolution, regardless of how dramatic the result is (e.g., the whole system crashing or burning). That is, the underlying oracle components always ``exists''.\\
The idea of specifying how a cell behaves under arbitrary in-neighborhoods is not new. This is something that in practice is done everywhere, through specific expressions according to the particular application. This is done, for instance, with expressions such as $ \dot{x}_i = f(x_i) + \sum_{j=1}^{N} g(x_i;w_{ij},x_{j})$, where $ N $ does not refer to a specific fixed integer, but instead, allows the expression to describe arbitrary in-neighborhoods.\\ 
Oracle components formalize the idea underlying such expressions as a mathematical object in its own right. Furthermore, we do not assume any particular structure such as in the aforementioned expression, but instead, only require them to satisfy the notion of in-neighborhood equivalence established in the original groupoid formalism. In that formalism, this notion is enforced by a set of equality constraints through a pullback map on a groupoid of bijections. Here, \cref{thm:oracle_permut,thm:oracle_merge,thm:oracle_0_equal} of \cref{defi:oracle} act as generators of that set of equalities, extended to arbitrary in-neighborhoods.\\
This approach has the important advantage that given an oracle function, the admissible function associated with every network (with finite neighborhoods) is automatically (and uniquely) defined. On the other hand, by constructing admissible functions directly while only enforcing cross-compatibility between the particular types of in-neighborhoods that we are interested in, does not gives us any guarantees that such functions could then be further extended for other types of in-neighborhoods. In such a case, we would have assigned an admissible function to a network while claiming that some other networks living in the same universe and evolving according to the same laws would simply not have an admissible function that specifies its dynamical evolution. The approach used here guarantees that such non-physical scenario can never arise.
\\
In summary, the formalism used here generalizes the groupoid formalism in multiple aspects, but restricts it in one particular detail. We require the existence of a self-consistent, underlying law, that fully defines the behavior of a cell under any circumstance. The construction of admissible functions is then completely analogous to the modeling of real physical systems. That is, we take the appropriate physical law (oracle function) and evaluate it on the physical system of interest (network).
\\
The structure of commutative monoids is quite general and we introduced it to the formalism motivated by real world examples often used in practice (e.g., in electrical circuit theory, resistors in parallel form a commutative monoid such that $ 30\|15 = 10 $). However, this work is completely agnostic to the particular choice of monoids so we keep all examples as simple possible in order to illustrate the corresponding concepts. In fact, in this work the main concern is with whenever an element of a monoid is zero or non-zero. Therefore we keep it simple by having in all examples the monoids be integers under the usual sum. The results being valid for arbitrary monoids does not increase the complexity of their proofs, so generality comes for free.
\section{Equality-based synchronism}\label{sec:invariant_synchrony}
In this section, we concern ourselves with patterns of synchronism defined by equalities between the states of cells. Such a set of equalities establishes an equivalence relation, which we encode through the use of partitions on the set of cells. For an overview of equivalence relations, partitions and the related concepts in \cref{subsec:partition_representation} refer to \cite{enderton1977elements}.\\
We generalize the known results in \cite{stewart2007lattice} regarding lattices of balanced partitions $ \Lambda_{\mathcal{G}} $ to the case of invariance under some subset of admissible functions $ F\subseteq\mathcal{F}_{\mathcal{G}}$. In particular, we show that the set of partitions $ L_{F} $ that are invariant under $ F $ always form lattices. Furthermore, these lattices share with $ \Lambda_{\mathcal{G}} $ the very particular properties of always containing the trivial partition $ (\bot) $ and being closed under the standard partition join $ ( \vee) $. For general lattice theory refer to \cite{davey_priestley_2002}.\\
We study with some detail the particular case of lattices $ L $ such that $ \vee_{L} = \vee $ and $ \bot_L = \bot $ and then show that all the lattices regarding invariant synchrony patterns are of this type. Furthermore, such lattices always have an associated $ cir_L $ function. 
\subsection{Partitions and their representations}\label{subsec:partition_representation}
A partition $ \mathcal{A} $ on a set of cells $ \mathcal{C} $ is a set of non-empty subsets of $ \mathcal{C} $ such that they are pairwise disjoint and their union is equal to $ \mathcal{C} $. We often refer to each element of a given partition (corresponding to a subset of cells) by the term \textit{color}. The number of colors in a partition is called its \textit{rank}.\\
We construct the \textit{quotient set} $ \mathcal{C}/\mathcal{A} $ by taking the elements of $ \mathcal{C} $ and merging them together according to $ \mathcal{A} $, such that each color of $ \mathcal{A} $ is associated with an element of $ \mathcal{C}/\mathcal{A} $. We can now think of $ \mathcal{A} $ as a function from $ \mathcal{C} $ to $ \mathcal{C}/\mathcal{A} $, which we illustrate in the following example.
\begin{exmp}
	Consider the set of cells $ \mathcal{C} = \{a,b,c,d,e\} $. Then $ \mathcal{A} = \{\{a,b\},\{c\},\{d,e\}\} $ is a partition on $ \mathcal{C} $ with three colors $(rank(\mathcal{A}) =3)$. We denote the quotient set as $ \mathcal{C}/\mathcal{A}=\{ab,c,de\} $, which contains three elements. Then $ \mathcal{A} $ acts as function in $ \mathcal{C}\to \mathcal{C}/\mathcal{A} $, and we write $ \mathcal{A}(a)=\mathcal{A}(b) = ab $, $ \mathcal{A}(c) = c$ and $ \mathcal{A}(d) =\mathcal{A}(e) = de $.
	\hfill$ \square $
\end{exmp}
\begin{remark}
	In the example above it might look more canonical to think of the elements of $ \mathcal{C}/\mathcal{A} $ as $ ab := \{a,b\}$, $ c := \{c\} $ and $  de := \{d,e\} $. That is, each of its elements is a color according to partition $ \mathcal{A} $. 
	However, using this notation, $ \mathcal{A} $ and $ \mathcal{C}/\mathcal{A} $ would look indistinguishable. We want to think of these objects as semantically different. While we think of a partition $ \mathcal{A} $ as a set of sets of elements (cells), we think of $ \mathcal{C}/\mathcal{A} $ as just a set of elements, (which are colors, and therefore end up being sets themselves). In order to make this clear we use this shorthand notation. This becomes more important when we compose partitions (e.g., we apply a partition on the set $ \mathcal{C}/\mathcal{A} $) and define the concept of partition quotients.
	\hfill$ \square $
\end{remark}
Interpreting partitions as functions allows us to say that two cells $ c,d\in\mathcal{C} $ are of the same color, according to $ \mathcal{A} $, if and only if $ \mathcal{A}(c) = \mathcal{A}(d) $. Furthermore, they are surjective functions by construction and each color is given by the preimage of each element of $ \mathcal{C}/\mathcal{A} $. Conversely, every surjective function establishes a partition on its domain through its level sets.\\
Given two partitions $ \mathcal{A} $, $ \mathcal{B} $ on a set of cells $ \mathcal{C} $, we say that $ \mathcal{A} $ is \textit{finer} than $ \mathcal{B} $, denoted as $ \mathcal{A} \leq \mathcal{B}$, if 
\begin{equation}
\mathcal{A}(c) = \mathcal{A}(d)
\implies
\mathcal{B}(c) = \mathcal{B}(d)
\label{eq:refinement_def}
\end{equation}
for all $ c,d\in\mathcal{C} $. Conversely, $ \mathcal{B} $ is said to be \textit{coarser} than $ \mathcal{A} $. Roughly speaking, \eref{eq:refinement_def} means that if any pair of cells have the same color according to partition $ \mathcal{A} $, then they also have the same color according to $ \mathcal{B} $. In other words, if we merge some of the colors of $ \mathcal{A} $ together, we can obtain $ \mathcal{B} $. Conversely, we can obtain $ \mathcal{A} $ by starting with $ \mathcal{B} $ and splitting some of its colors into smaller ones. The \textit{trivial} partition, in which each color consists of a single cell, is the \textit{finest} and its rank is $ \vert\mathcal{C}\vert $. We now show that if $ \mathcal{A} \leq \mathcal{B} $ we can define a \textit{quotient partition} $ \mathcal{B}/\mathcal{A} $.
\begin{lemma}
	\label{lemma:quotient_partition}
	Consider a set of cells $ \mathcal{C} $ and the partitions $ \mathcal{A}:\mathcal{C}\to \mathcal{C}/\mathcal{A} $ and $ \mathcal{B}:\mathcal{C}\to \mathcal{C}/\mathcal{B} $. Then $ \mathcal{A} \leq \mathcal{B}$ if and only if there is some $ \mathcal{B}/\mathcal{A} : \mathcal{C}/\mathcal{A} \to \mathcal{C}/\mathcal{B}$ such that $ \mathcal{B}/\mathcal{A}\circ\mathcal{A} = \mathcal{B} $.
	\hfill$ \square $
\end{lemma}
\begin{proof}
	Note that $ \mathcal{B}/\mathcal{A}\circ\mathcal{A} = \mathcal{B} $ means that $ \mathcal{B}/\mathcal{A}(\mathcal{A}(c)) = \mathcal{B}(c) $ for all $ c\in \mathcal{C} $. This means that $ \mathcal{B}/\mathcal{A} $ is the function that maps $ \mathcal{A}(c)\in \mathcal{C}/\mathcal{A} $ into $ \mathcal{B}(c) \in \mathcal{C}/\mathcal{B}$ for all $ c\in\mathcal{C} $. Note that this is enough to define $ \mathcal{B}/\mathcal{A} $ on its whole domain since $ \mathcal{A} $ is surjective. That is, for every element $ k\in\mathcal{C}/\mathcal{A} $ there is some $ c\in\mathcal{C} $ such that $ \mathcal{A}(c) = k $. Finally, $ \mathcal{B}/\mathcal{A} $ exists if and only if such a function is well-defined. That is, for every $ k\in\mathcal{C}/\mathcal{A} $, the mapping of $ k=\mathcal{A}(c) $ into $ \mathcal{B}(c) $ has to be completely independent of the particular choice of $ c\in\mathcal{C} $, which is equivalent to $ \mathcal{A}\leq \mathcal{B} $.
\end{proof}
In particular, if $ \mathcal{A}\leq\mathcal{B} $, the partition $ \mathcal{B}/\mathcal{A} $ describes how to merge the colors of $ \mathcal{A} $ into the colors of $ \mathcal{B} $. Furthermore, note that $ \mathcal{B}/\mathcal{A} $ is uniquely defined and is also surjective. If we consider the particular case $ \mathcal{B}=\mathcal{A} $, then $ \mathcal{A}/\mathcal{A} : \mathcal{C}/\mathcal{A} \to \mathcal{C}/\mathcal{A}$ is such that $ \mathcal{A}/\mathcal{A}\circ\mathcal{A} = \mathcal{A} $. That is, $ \mathcal{A}/\mathcal{A} $ acts as the identity map in the set $ \mathcal{C}/\mathcal{A} $ and it is the trivial partition in that set.
\begin{exmp}
	\label{exmp:partition_quotient_example}
	Consider the set of cells $ \mathcal{C} = \{a,b,c,d,e\} $ on which we define the partitions $ \mathcal{A} = \{\{a,b\},\{c\},\{d,e\}\} $  $\mathcal{B} = \{\{a,b\},\{c,d,e\}\} $. We denote the quotient sets as $ \mathcal{C}/\mathcal{A}=\{ab,c,de\} $ and $ \mathcal{C}/\mathcal{B} = \{ab,cde\} $. Consider that the mappings $ \mathcal{A}:\mathcal{C}\to \mathcal{C}/\mathcal{A} $ and $ \mathcal{B}:\mathcal{C}\to \mathcal{C}/\mathcal{B} $ are defined in the expected way. Then since $ \mathcal{A} \leq \mathcal{B}$, the quotient partition $ \mathcal{B}/\mathcal{A} : \mathcal{C}/\mathcal{A} \to \mathcal{C}/\mathcal{B} $ is such that $ \mathcal{B}/\mathcal{A}(ab) = ab $ and $ \mathcal{B}/\mathcal{A}(c) = \mathcal{B}/\mathcal{A}(de) = cde$.\\
	Using the set of colors notation, we can write $ \mathcal{B}/\mathcal{A} = \{\{ab\},\{c,de\}\} $. Finally, note that $ rank(\mathcal{B}/\mathcal{A})=rank(\mathcal{B}) =2 $.
	\hfill$ \square $
\end{exmp}
It should be clear that $ rank(\mathcal{B}/\mathcal{A})=rank(\mathcal{B}) $ is true in general, since it always corresponds to the size of their common image set $ \mathcal{C}/\mathcal{B} $.
\begin{lemma}
	\label{lemma:partition_quotient_preserves_partial_order}
	The partition quotient preserves the partial order relation $ \leq $. That is, for all partitions $ \mathcal{A},\mathcal{B}_1,\mathcal{B}_2 $ on $ \mathcal{C} $ such that $ \mathcal{A} \leq \mathcal{B}_1,\mathcal{B}_2  $, we have $ \mathcal{B}_1 \leq \mathcal{B}_2 $ if and only if $ \mathcal{B}_1/\mathcal{A} \leq \mathcal{B}_2/\mathcal{A}$.
	\hfill$ \square $
\end{lemma}
\begin{proof}
	Firstly, note that $ \mathcal{B}_1/\mathcal{A} $ and $ \mathcal{B}_2/\mathcal{A} $ are both partitions on the set $ \mathcal{C}/\mathcal{A} $, therefore the statement $ \mathcal{B}_1/\mathcal{A} \leq \mathcal{B}_2/\mathcal{A}$ is meaningful.\\
	Since $ \mathcal{A} \leq \mathcal{B}_1,\mathcal{B}_2  $ by assumption, we can, using \cref{lemma:quotient_partition}, write $ \mathcal{B}_1 \leq \mathcal{B}_2 $ as $ \mathcal{B}_1/\mathcal{A}(\mathcal{A}(c)) = \mathcal{B}_1/\mathcal{A}(\mathcal{A}(d))
	\implies
	\mathcal{B}_2/\mathcal{A}(\mathcal{A}(c)) = \mathcal{B}_2/\mathcal{A}(\mathcal{A}(d)) $ for all $ c,d\in\mathcal{C} $.\\
	We have to show that this is equivalent to $ \mathcal{B}_1/\mathcal{A}(k) = \mathcal{B}_1/\mathcal{A}(l)
	\implies
	\mathcal{B}_2/\mathcal{A}(k) = \mathcal{B}_2/\mathcal{A}(l) $ for all $ k,l\in \mathcal{C}/\mathcal{A} $. The forward direction holds because	$ \mathcal{A} $ is surjective. That is, for all $ k,l\in \mathcal{C}/\mathcal{A} $ there are some $ c,d\in\mathcal{C} $ such that $ \mathcal{A}(c)= k $ and $ \mathcal{A}(d) = l $. The backwards direction is immediate since for all $ c,d\in\mathcal{C} $, we have $ \mathcal{A}(c),\mathcal{A}(d)\in \mathcal{C}/\mathcal{A}$.
\end{proof}
\begin{lemma}
	\label{lemma:partition_quotient_canceling}
	Consider partitions $ \mathcal{A},\mathcal{B}_1,\mathcal{B}_2 $ on $ \mathcal{C} $ such that $ \mathcal{A} \leq \mathcal{B}_1 \leq \mathcal{B}_2  $. Then $ (\mathcal{B}_2/\mathcal{A})/(\mathcal{B}_1/\mathcal{A}) 
	=
	\mathcal{B}_2/\mathcal{B}_1$.
	\hfill$ \square $
\end{lemma}
\begin{proof}
	From \cref{lemma:partition_quotient_preserves_partial_order}, we have $ \mathcal{B}_1/\mathcal{A} \leq \mathcal{B}_2/\mathcal{A} $. Therefore, $ (\mathcal{B}_2/\mathcal{A})/(\mathcal{B}_1/\mathcal{A}) $ is well-defined.\\
	Furthermore, from \cref{lemma:quotient_partition}, $ (\mathcal{B}_2/\mathcal{A})/(\mathcal{B}_1/\mathcal{A})  $ and $ \mathcal{B}_2/\mathcal{B}_1  $ are both mappings in $ \mathcal{C}/\mathcal{B}_1\to \mathcal{C}/\mathcal{B}_2 $. We now show that $ (\mathcal{B}_2/\mathcal{A})/(\mathcal{B}_1/\mathcal{A})(k) 
	=
	\mathcal{B}_2/\mathcal{B}_1(k) $ for all $ k\in\mathcal{C}/\mathcal{B}_1 $. Note that $ k\in\mathcal{C}/\mathcal{B}_1 $ if and only if there is some $ c\in\mathcal{C} $ such that $ \mathcal{B}_1(c)=k $. Since $ \mathcal{A}\leq \mathcal{B}_1 $, we have $ k = \mathcal{B}_1/\mathcal{A}(\mathcal{A}(c)) $. Then
	\begin{eqnarray*}
	(\mathcal{B}_2/\mathcal{A})/(\mathcal{B}_1/\mathcal{A})(k) 
	&=
	(\mathcal{B}_2/\mathcal{A})/(\mathcal{B}_1/\mathcal{A})\circ\mathcal{B}_1/\mathcal{A}(\mathcal{A}(c)) \\
	&=
	\mathcal{B}_2/\mathcal{A}\circ\mathcal{A}(c) \\
	&=
	\mathcal{B}_2(c).
	\end{eqnarray*}
	Since $ \mathcal{B}_1\leq \mathcal{B}_2 $, we have $ \mathcal{B}_2 = \mathcal{B}_2/\mathcal{B}_1\circ \mathcal{B}_1 $. Then
	\begin{eqnarray*}
	\mathcal{B}_2(c) 
	&=
	\mathcal{B}_2/\mathcal{B}_1\circ \mathcal{B}_1(c)
	\\
	&=
	\mathcal{B}_2/\mathcal{B}_1(k).
	\end{eqnarray*}
\end{proof}
It is often convenient to establish an order on a set of cells. That is, to associate with each cell a distinct integer from $ 1 $ to $ n $, where $ n $ is the size of that set. We now see that this allows us to represent partitions using matrices.\\
Consider we identify $ \mathcal{C} $ with $ \{1,\ldots,\vert \mathcal{C}\vert\} $ and $ \mathcal{C}/\mathcal{A} $ with $ \{1,\ldots,\vert \mathcal{C}/\mathcal{A}\vert\} $. Then we can represent a partition $ \mathcal{A}:\mathcal{C}\to \mathcal{C}/\mathcal{A} $ through a \textit{partition matrix} (also called characteristic matrix) $ P \in \{0,1\}^{\vert \mathcal{C}\vert \times \vert \mathcal{C}/\mathcal{A}\vert }$, such that $ \left[P\right]_{ck} = 1 $ if $ \mathcal{A}(c) = k $ and $ \left[P\right]_{ck} = 0 $ otherwise. That is, rows corresponds to the cells and columns correspond to the colors, with $ 1 $ encoding that the cell of that row maps into the color associated with that column. This is illustrated in the following example.
\begin{exmp}
	Consider the same sets of cells and partitions as in \cref{exmp:partition_quotient_example}.
	For $ \mathcal{C} $ we use the indexing $ (a,b,c,d,e) = (1,2,3,4,5) $, for $ \mathcal{C}/\mathcal{A} $ we index $ (ab,c,de) = (1,2,3) $, and we index $ \mathcal{C}/\mathcal{B} $ according to $ (ab,cde) = (2,1)$.
	We index $ ab $ differently as a member of $ \mathcal{C}/\mathcal{A} $ than as a member of $ \mathcal{C}/\mathcal{B} $. This is not an issue since an ordering is a property within a given set, not something intrinsic to an element.
	Using the mentioned indexing, the partitions $ \mathcal{A},\mathcal{B},\mathcal{B}/\mathcal{A} $ are represented through the partition matrices $ P_{\mathcal{A}},P_{\mathcal{A}},P_{\mathcal{B}/\mathcal{A}} $, which are given by
	\begin{eqnarray*}
	P_{\mathcal{A}} =
	\left[\begin{array}{ccc} 
		1 & 0 & 0\\
		1 & 0 & 0\\
		0 & 1 & 0\\
		0 & 0 & 1\\
		0 & 0 & 1 
	\end{array}\right]
	,\quad
	P_{\mathcal{B}} =
	\left[\begin{array}{ccc} 
	0 & 1\\
	0 & 1\\
	1 & 0\\
	1 & 0\\
	1 & 0
	\end{array}	\right]
	,\quad
	P_{\mathcal{B}/\mathcal{A}} =
	\left[\begin{array}{cc} 
	0 & 1\\
	1 & 0\\
	1 & 0
	\end{array}\right]
	.
	\end{eqnarray*}
	\hfill$ \square $
\end{exmp}
These matrices are related by $ P_{\mathcal{A}} P_{\mathcal{B}/\mathcal{A}} = P_{\mathcal{B}}$. This is equivalent to $ \mathcal{B}/\mathcal{A}\circ\mathcal{A} = \mathcal{B} $. It is more clear that these formulas are analogous if we consider the transposed version $ P_{\mathcal{B}/\mathcal{A}}^{\top}   P_{\mathcal{A}}^{\top}  = P_{\mathcal{B}}^{\top}$. In this work, we considered it more useful to define partition matrices the way we did instead of the transposed alternative.\\
Given a partition $ \mathcal{A} $, we can index its related sets $ \mathcal{C} $ and $ \mathcal{C}/\mathcal{A} $ in different ways. This means that $ \mathcal{A} $ can be represented by multiple partition matrices that are related to each other by a reordering of rows and columns. This is not an issue as long as we keep things consistent by always using the same assigned ordering when constructing other partition matrices that also involve $ \mathcal{C} $ and $ \mathcal{C}/\mathcal{A} $.\\
We often use the partition and its matrix interchangeably, that is, $ P_{\mathcal{A}} \leq \mathcal{B} $ or $ P_{\mathcal{A}} \leq P_{\mathcal{B}} $ to mean $ \mathcal{A} \leq \mathcal{B}$.\\
Given partition matrices $ P_{\mathcal{A}},P_{\mathcal{B}} $ such that $ P_{\mathcal{A}} \leq P_{\mathcal{B}} $, we have, by assumption, already assigned an ordering on all the relevant sets $ \mathcal{C} $, $ \mathcal{C}/\mathcal{A} $ and $ \mathcal{C}/\mathcal{B} $. Therefore there exists an unique partition matrix $ P_{\mathcal{B}/\mathcal{A}} $, representing $ \mathcal{B}/\mathcal{A} $ such that $ P_{\mathcal{A}}  P_{\mathcal{B}/\mathcal{A}} = P_{\mathcal{B}}  $.\\
The trivial partition can be represented by any $ \vert\mathcal{C}\vert\times\vert\mathcal{C}\vert $ permutation matrix, one of which is the identity.\\
The rank of a partition corresponds to the rank of any of its matrix representations. That is, $rank(\mathcal{A}) = rank(P_{\mathcal{A}}) $.\\
Given some matrix $ M $ of appropriate dimensions, $ PM $ is always well-defined as an expansion of $ M $, where its rows are replicated. In the case of $ MP $, we must be able to sum elements of $ M $. In our context, the sum operations are the previously mentioned monoid sum operations $ \| $.
\subsection{Lattices of partitions}\label{subsec:lattice_partitions} 
A lattice $ L $ is a partially ordered set such that given any two elements $ a,b\in L $, there exists in $ L $ a \textit{least upper bound} or \textit{join} denoted by $ a \vee_L b $. Similarly, there is in $ L $ a \textit{greatest lower bound} or \textit{meet} denoted by $ a \wedge_L b $.
\begin{exmp}
	Consider \fref{fig:lattice_examples}, where we represent two partially ordered sets $ L $ and $ S $. We connect two different elements if and only if one is larger than the other (according to its assigned partial order $ \leq $) and they have no other element in-between. Furthermore, we present graphically the larger elements above the smaller terms. For instance, in \fref{fig:Lattice_example}, we have $ e\leq_L b $ and $ b\leq_L a $ so we connect them. However, we do not connect $ e-a $ despite $ e \leq_L a $ since $ b $ is in-between them. Note that $ L $ is a lattice since $ \vee_{L} $ and $ \wedge_{L} $ are well-defined for every pair of elements (e.g., $ b\vee_{L}d = a$ and $ b\wedge_{L} d =f $). On the other hand, $ S $ does not have this property. The set of elements larger than $ l $ and $ m $ is $ \{i,j,k\} $. Out of these, $ j,k $ are both smaller than $ i $, however, neither $ j\leq k $ nor $ k\leq j $. That is, they are non-comparable. Since $ \{i,j,k\} $ does not have a smallest element, $ l \vee_{S} m $ is not defined, which means that $ S $ is not a lattice.
	\begin{figure}[h]
		\centering
		\begin{subfigure}[t]{0.3\textwidth}
			\centering
			\begin{tikzpicture}[
node1/.style = {circle,minimum size=23,draw},
node2/.style = {circle,minimum size=23,draw,fill=white!75!black},
node3/.style = {circle,minimum size=23,draw,fill=white!50!black},
lightgray/.style = {rectangle,minimum size=20,draw,draw,fill=white!93!black},
pweak/.style = {rectangle,minimum size=20,draw,draw,fill=white!50!black},
proot/.style = {rectangle,minimum size=20,draw,draw,fill=white!75!black},
pstrong/.style = {rectangle,minimum size=20,draw},
noderect/.style = {rectangle,minimum size=20,draw},
edge1/.style = {>=latex,thick},
edgedash/.style = {>=latex,thick,dashed},
edge2/.style = {>=latex,thick,blue},
edge3/.style = {>=latex,thick,red}
]

\def\layer{1.3}
\def\del{0.75}

\node[noderect] at (0,0) (bot){$ h $};

\node[noderect] at ({-2*\del},{\layer} ) (12){$ e $};
\node[noderect] at ({0*\del},{\layer} ) (24){$ f $};
\node[noderect] at ({2*\del},{\layer} ) (23){$ g $};

\node[noderect] at ({-2*\del},{2*\layer} ) (124){$ b $};
\node[noderect] at ({0*\del},{2*\layer} ) (123){$ c $};
\node[noderect] at ({2*\del},{2*\layer} ) (234){$ d $};

\node[noderect] at (0,{3*\layer} ) (1234){$ a $};

\draw [-,edge1](bot) -- (12);
\draw [-,edge1](bot) -- (24);
\draw [-,edge1](bot) -- (23);

\draw [-,edge1](12) -- (124);
\draw [-,edge1](12) -- (123);

\draw [-,edge1](24) -- (124);
\draw [-,edge1](24) -- (234);

\draw [-,edge1](23) -- (123);
\draw [-,edge1](23) -- (234);

\draw [-,edge1](124) -- (1234);
\draw [-,edge1](123) -- (1234);
\draw [-,edge1](234) -- (1234);

\end{tikzpicture} 
			\caption{Lattice $ L $.}
			\label{fig:Lattice_example}
		\end{subfigure}
		\begin{subfigure}[t]{0.3\textwidth}
			\centering
			\begin{tikzpicture}[
node1/.style = {circle,minimum size=23,draw},
node2/.style = {circle,minimum size=23,draw,fill=white!75!black},
node3/.style = {circle,minimum size=23,draw,fill=white!50!black},
lightgray/.style = {rectangle,minimum size=20,draw,draw,fill=white!93!black},
pweak/.style = {rectangle,minimum size=20,draw,draw,fill=white!50!black},
proot/.style = {rectangle,minimum size=20,draw,draw,fill=white!75!black},
pstrong/.style = {rectangle,minimum size=20,draw},
noderect/.style = {rectangle,minimum size=20,draw},
edge1/.style = {>=latex,thick},
edgedash/.style = {>=latex,thick,dashed},
edge2/.style = {>=latex,thick,blue},
edge3/.style = {>=latex,thick,red}
]

\def\layer{1.3}
\def\del{0.75}

\node[noderect] at (0,0) (bot){$ n $};

\node[noderect] at ({-1*\del},{\layer} ) (12){$ l $};
\node[noderect] at ({1*\del},{\layer} ) (24){$ m $};

\node[noderect] at ({-1*\del},{2*\layer} ) (124){$ j $};
\node[noderect] at ({1*\del},{2*\layer} ) (123){$ k $};

\node[noderect] at (0,{3*\layer} ) (1234){$ i $};

\draw [-,edge1](bot) -- (12);
\draw [-,edge1](bot) -- (24);

\draw [-,edge1](12) -- (124);
\draw [-,edge1](12) -- (123);

\draw [-,edge1](24) -- (124);
\draw [-,edge1](24) -- (123);

\draw [-,edge1](124) -- (1234);
\draw [-,edge1](123) -- (1234);

\end{tikzpicture} 
			\caption{Non-lattice $ S $.}
			\label{fig:Non_lattice_example}
		\end{subfigure}
		\caption{Partially ordered sets $ L,S $ such that $ L $ is a lattice and $ S $ is not a lattice.}
		\label{fig:lattice_examples}
	\end{figure}
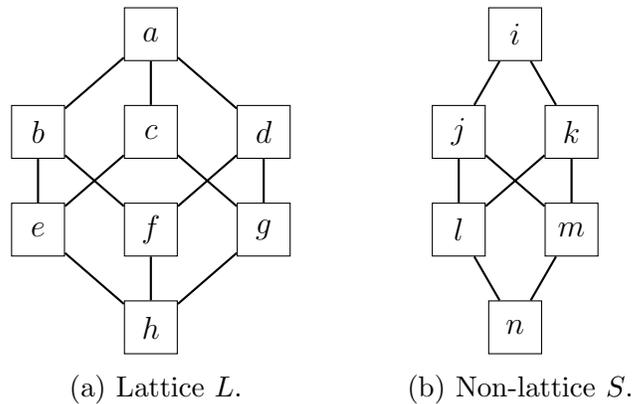
	\hfill$ \square $
\end{exmp} 
In this work, we are interested only in lattices of partitions, partially ordered according to the finer ($ \leq $) relation, described in \eref{eq:refinement_def}.\\
The set of all partitions on a finite set of cells $ \mathcal{C} $, partially ordered by the finer ($ \leq $) relation, forms a \textit{lattice} $ L_{\mathcal{C}} $. In this set, the join ($ \vee $) and meet ($ \wedge $) operations can be calculated according to \cref{lemma:join_chain_def,lemma:meet_def} respectively.
\begin{lemma}\label{lemma:join_chain_def}
	The partition given by $ \mathcal{A} = \mathcal{A}_1 \vee \mathcal{A}_2 $ is such that $ \mathcal{A}(c) = \mathcal{A}(d) $ if and only if there is a chain of cells $ c=c_1,\ldots,c_k = d $ such that, for each $ c_i,c_{i+1} $, with $ 1 \leq i < k $, we have either $ \mathcal{A}_1 (c_{i}) = \mathcal{A}_1 (c_{i+1})$ or $ \mathcal{A}_2 (c_{i}) = \mathcal{A}_2 (c_{i+1}) $.
	\hfill$ \square $
\end{lemma}
\begin{proof}
	Any partition $ \mathcal{A} $ that is simultaneous coarser than $  \mathcal{A}_1 $ and $  \mathcal{A}_2 $ has to obey (from \eref{eq:refinement_def})
	\begin{eqnarray*}
	\cases{ \mathcal{A}_1(c) = \mathcal{A}_1(d) \\
	\text{or}\\
	\mathcal{A}_2(c) =  \mathcal{A}_2(d)}
	\implies
	\mathcal{A}(c) = \mathcal{A}(d).
	\end{eqnarray*}
	For such partition, any chain of cells $ c=c_1,\ldots,c_k = d $ such that, for each $ c_i,c_{i+1} $, with $ 1 \leq i < k $, either $ \mathcal{A}_1(c_{i}) = \mathcal{A}_1(c_{i+1})$ or $ \mathcal{A}_2(c_{i}) = \mathcal{A}_2(c_{i+1}) $, implies that $ \mathcal{A}(c) = \mathcal{A}(d) $. The finest such partition $ \mathcal{A} $ is the one such that $ \mathcal{A}(c) = \mathcal{A}(d) $ if and only if there is such a chain. The existence of such chains induces an equivalence relation on the set of cells. Therefore this defines a valid partition.
\end{proof}
\begin{lemma}\label{lemma:meet_def}
	The partition given by $ \mathcal{A} = \mathcal{A}_1 \wedge \mathcal{A}_2 $ is such that $ \mathcal{A}(c) = \mathcal{A}(d) $ if and only if $ \mathcal{A}_1 (c) = \mathcal{A}_1 (d)$ and $ \mathcal{A}_2 (c) = \mathcal{A}_2 (d) $.
	\hfill$ \square $
\end{lemma}
\begin{proof}
	Any partition $ \mathcal{A} $ that is simultaneously finer than $ \mathcal{A}_1 $ and $ \mathcal{A}_2 $ has to obey (from \eref{eq:refinement_def})
	\begin{eqnarray*}
	\mathcal{A}(c) = \mathcal{A}(d)
	\implies
	\cases{\mathcal{A}_1(c) = \mathcal{A}_1(d),\\
		\mathcal{A}_2(c) = \mathcal{A}_2(d).}
	\end{eqnarray*}
	The coarsest such partition is created by making the implication into an equivalence. This induces an equivalence relation on the set of cells. Therefore it defines a valid partition.
\end{proof}
Not every subset of partitions forms a lattice. Furthermore, subsets of lattices that are themselves lattices might not be sublattices of the original lattice. That is, their join and meet operations might be different. With regard to lattices of partitions, the join is coarser that in \cref{lemma:join_chain_def} and the meet is finer than in \cref{lemma:meet_def}.\\
Denote by $ L_{\type{}} $ the subset of $ L_{\mathcal{C}} $ consisting on the partitions of $ \mathcal{C} $ that are finer than $ \type{} $. Note that $ L_{\type{}} $ remains closed under the same join ($ \vee $) and meet ($ \wedge $) operations. Therefore $ L_{\type{}} $ is a sublattice of $ L_{\mathcal{C}} $.\\
All the lattices $ L $ in this work are bounded, which means that they have a (\textit{maximum}/\textit{greatest element}/\textit{top}), denoted by $ \top_L $ and a (\textit{minimum}/\textit{least element}/\textit{bottom}), denoted by $ \bot_L $.
In particular, the top partitions of $ L_{\mathcal{C}} $ and $ L_{\type{}} $ are $ \top_{\mathcal{C}} = \{\mathcal{C}\} $ and $ \top_{\type{}} = \type{} $, respectively. Their bottom element is the trivial partition on $ \mathcal{C} $, which we denote by $ \bot_{\mathcal{C}} $. We write the trivial partition as $ \bot $ whenever the underlying set is clear from context.\\
We now show that the existence of a minimal element together with a join operation is enough to guarantee that a finite set forms a lattice.
\begin{lemma}
	Consider a finite partially ordered ($ \leq_L $) set $ L $ such that there is a minimal element $ \bot_L \in L $ and for every pair $ \mathcal{A}_1,\mathcal{A}_2 \in L $, there exists an element denoted $ \mathcal{A}_1 \vee_L \mathcal{A}_2 $ which is their least upper bond in $ L $. Then $ L $ is a lattice.
	\label{lemma:minimal_join_equal_lattice} 
	\hfill$ \square $
\end{lemma}
\begin{proof}
	Consider any pair of elements $ \mathcal{A}_1 ,\mathcal{A}_2  \in  L$. Call $ S $ the subset of $ L $ of the elements that are simultaneously smaller ($ \leq_L  $) than $ \mathcal{A}_1 $ and $\mathcal{A}_2 $. That is, $ S: = \left\{ \mathcal{P}\in L: \mathcal{P}\leq_L \mathcal{A}_1,\mathcal{A}_2 \right\} $.
	Note that $ S $ is finite. Furthermore, it is not empty since $ \bot_L \in S $. Then, to obtain the largest element of $ S $ we apply the join $ (\vee_L) $ operation over the whole set, obtaining $ \mathcal{B} = \bigvee_{\mathcal{P}\in S}^L \mathcal{P} $. By assumption, the result is in $ L $. Furthermore, since all the elements of $ S $ are smaller than $ \mathcal{A}_1 $ and $ \mathcal{A}_2 $, then $ \mathcal{B} $ is smaller as well. Therefore $ \mathcal{B} \in S $. By construction, $ \mathcal{B} $ is larger than every other element of $ S $, therefore it is an upper bound of $ S $. That is, $ \mathcal{B} \in S $ is the greatest lower bound of $ \mathcal{A}_1 ,\mathcal{A}_2  $ in $ L $, which we 
	denote by $ \mathcal{A}_1 \wedge_L \mathcal{A}_2 $, which means that $ L $ is a lattice.
\end{proof}
In this work, we have particular interest in lattices of partitions $ L $ in which the bottom partition is the trivial one $ (\bot_L = \bot) $ and the join is given according to \cref{lemma:join_chain_def} $ (\vee_{L} = \vee) $.
\begin{lemma}
	\label{lemma:bot_join_implies_cir}
	Consider a lattice of partitions $ L \subseteq L_{\type{}}$ such that $ \bot_L = \bot $ and $ \vee_L = \vee $. Then given any partition $ \mathcal{A}\in L_{\type{}} $, there is a partition $ \mathcal{B}\in L $ that is the coarsest one in $ L $ such that $ \mathcal{B} \leq \mathcal{A}$.
	\hfill $ \square $
\end{lemma}
\begin{proof}
	Call $ S $ the subset of $ L $ of the elements that are finer $ ( \leq ) $ than $ \mathcal{A} $. That is, $ S: = \left\{ \mathcal{P}\in L: \mathcal{P}\leq \mathcal{A} \right\} $.	Note that $ S $ is finite. Furthermore, it is not empty since $ \bot\in S $. Then, to obtain the coarsest element of $ S $ we apply the join $ (\vee_L = \vee) $ operation over the whole set, obtaining $ \mathcal{B} = \bigvee_{\mathcal{P}\in S} \mathcal{P} $. Then $ \mathcal{B} \in L$. Furthermore, since $ L \subseteq L_{\type{}} $ and $ \vee_L = \vee $, all the elements of $ S $ being finer than $ \mathcal{A} $ implies that $ \mathcal{B} $ is finer as well. Therefore $ \mathcal{B} \in S $. By construction, $ \mathcal{B} $ is coarser than every other element of $ S $, therefore it is an upper bound of $ S $. That is, $ \mathcal{B} \in S $ is the greatest lower bound of $ \mathcal{A}$ in $ L $.
\end{proof}
\begin{remark}
	\cref{lemma:bot_join_implies_cir} holds only because $ \bot_L = \bot $ and $ \vee_L = \vee  $. If $ \bot_L \neq \bot $, then it would not work for any $\mathcal{A} < \bot_L $ (or non-comparable). Furthermore, note that $ \mathcal{A}_1, \mathcal{A}_2 \leq \mathcal{A} $ only implies $ \mathcal{A}_1 \vee_L \mathcal{A}_2 \leq \mathcal{A} $ if those partitions are all in the lattice associated with $ \vee_L $. Since $ \vee_L = \vee $ we can apply this implication to the lattice $ L_{\type{}} $.
	\hfill$ \square $
\end{remark}
The correspondence between partitions $ \mathcal{A}\in L_{\type{}} $ and $ \mathcal{B}\in L $ described in \cref{lemma:bot_join_implies_cir} establishes a function in $ L_{\type{}} \to L $, which we denote by $ cir_L$.\\
\cref{lemma:minimal_join_equal_lattice} implies that a set $ L $ with a minimal partition $ \bot_L $ and a join $ \vee_{L} $ is automatically a lattice, therefore it has a meet operation $ \wedge_L $. Furthermore, in the case that $ L $ is a lattice of partitions such that $ \bot_L = \bot $ and $ \vee_L = \vee $, it is not guaranteed that $ \wedge_L = \wedge $. We know, however, that $ \mathcal{A}_1 \wedge_L \mathcal{A}_2 \leq \mathcal{A}_1 \wedge \mathcal{A}_2$. Then, using $ cir_L $, it is clear how to write $ \wedge_L $ as a function of $ \wedge $.
\begin{corollary}
	Consider a lattice of partitions $ L \subseteq L_{\type{}}$ such that $ \bot_L = \bot $ and $ \vee_L = \vee $. Then given partitions $ \mathcal{A}_1 ,\mathcal{A}_2  \in  L $, we have $ \mathcal{A}_1 \wedge_L \mathcal{A}_2 = cir_L(\mathcal{A}_1 \wedge \mathcal{A}_2)$.
	\label{cor:meet_cir_relation_general}
	\hfill $ \square $
\end{corollary}
The meet operation $ \wedge_L $ is meaningful only when applied to elements of $ L $ while $ cir_L $ can be applied to any element of $ L_{\type{}} $.\\ 
We now illustrate the $ cir_L $ operation in the following example.
\begin{exmp}
	\label{exmp:cir_exmp}
	\Fref{fig:cir_F_example_Lt} shows the lattice of all partitions finer than $ \type{} = \{\{1,2,3\},\{4,5\}\} $ and \fref{fig:cir_F_example_LF} shows some lattice $ L $, which contains the trivial partition $ \bot $ and is closed under the partition join $ \vee $. In this work, we use partitions to encode patterns of synchrony on a set of cells. Then the singleton elements of a partition correspond to cells that are not synchronized with any other cell. We omit them in order to simplify the notation (e.g., the partition $ \{\{1,2\},\{3\},\{4,5\}\} $ is simply represented as $ 12/45 $).
	 The lattices are colored such that each element of $ L_{\type{}} $ is of the same color of the element of $ L $ that $ cir_L $ maps to.
	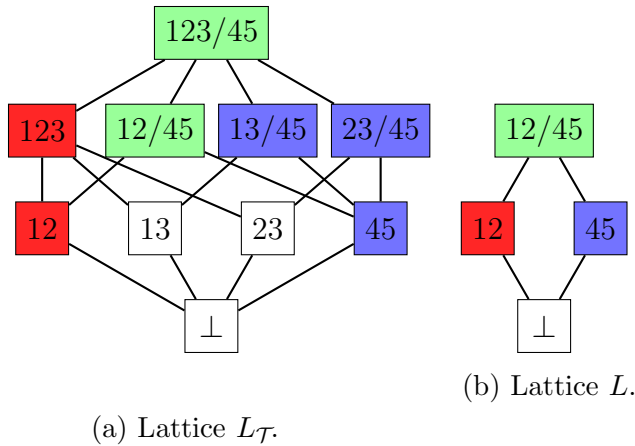
\begin{figure}[h]
		\centering
		\begin{subfigure}[t]{0.3\textwidth}
			\centering
			\begin{tikzpicture}[
node1/.style = {circle,minimum size=23,draw},
node2/.style = {circle,minimum size=23,draw,fill=white!75!black},
node3/.style = {circle,minimum size=23,draw,fill=white!50!black},
red/.style = {rectangle,minimum size=20,draw,draw,fill=white!15!red},
green/.style = {rectangle,minimum size=20,draw,draw,fill=white!60!green},
blue/.style = {rectangle,minimum size=20,draw,draw,fill=white!45!blue},
lightgray/.style = {rectangle,minimum size=20,draw,draw,fill=white!93!black},
pweak/.style = {rectangle,minimum size=20,draw,draw,fill=white!50!black},
proot/.style = {rectangle,minimum size=20,draw,draw,fill=white!75!black},
pstrong/.style = {rectangle,minimum size=20,draw},
noderect/.style = {rectangle,minimum size=20,draw},
edge1/.style = {>=latex,thick},
edgedash/.style = {>=latex,thick,dashed},
edge2/.style = {>=latex,thick,blue},
edge3/.style = {>=latex,thick,red}
]

\def\layer{1.3}
\def\del{0.75}

\node[pstrong] at (0,0) (bot){$ \bot $};

\node[red] at ({-3*\del},{\layer} ) (12){$ 12 $};
\node[pstrong] at ({-1*\del},{\layer} ) (13){$ 13 $};
\node[pstrong] at ({1*\del},{\layer} ) (23){$ 23 $};
\node[blue] at ({3*\del},{\layer} ) (45){$ 45 $};

\node[red] at ({-3*\del},{2*\layer} ) (123){$ 123 $};
\node[green] at ({-1*\del},{2*\layer} ) (12_45){$ 12/45 $};
\node[blue] at ({1*\del},{2*\layer} ) (13_45){$ 13/45 $};
\node[blue] at ({3*\del},{2*\layer} ) (23_45){$ 23/45 $};

\node[green] at (0,{3*\layer} ) (123_45){$ 123/45 $};

\draw [-,edge1](bot) -- (12);
\draw [-,edge1](bot) -- (13);
\draw [-,edge1](bot) -- (23);
\draw [-,edge1](bot) -- (45);

\draw [-,edge1](12) -- (123);
\draw [-,edge1](12) -- (12_45);

\draw [-,edge1](13) -- (123);
\draw [-,edge1](13) -- (13_45);

\draw [-,edge1](23) -- (123);
\draw [-,edge1](23) -- (23_45);

\draw [-,edge1](45) -- (12_45);
\draw [-,edge1](45) -- (13_45);
\draw [-,edge1](45) -- (23_45);

\draw [-,edge1](123) -- (123_45);
\draw [-,edge1](12_45) -- (123_45);
\draw [-,edge1](13_45) -- (123_45);
\draw [-,edge1](23_45) -- (123_45);

\end{tikzpicture} 
			\caption{Lattice $ L_{\type{}} $.}
			\label{fig:cir_F_example_Lt}
		\end{subfigure}
		\begin{subfigure}[t]{0.3\textwidth}
			\centering
			\begin{tikzpicture}[
node1/.style = {circle,minimum size=23,draw},
node2/.style = {circle,minimum size=23,draw,fill=white!75!black},
node3/.style = {circle,minimum size=23,draw,fill=white!50!black},
red/.style = {rectangle,minimum size=20,draw,draw,fill=white!15!red},
green/.style = {rectangle,minimum size=20,draw,draw,fill=white!60!green},
blue/.style = {rectangle,minimum size=20,draw,draw,fill=white!45!blue},
lightgray/.style = {rectangle,minimum size=20,draw,draw,fill=white!93!black},
pweak/.style = {rectangle,minimum size=20,draw,draw,fill=white!50!black},
proot/.style = {rectangle,minimum size=20,draw,draw,fill=white!75!black},
pstrong/.style = {rectangle,minimum size=20,draw},
noderect/.style = {rectangle,minimum size=20,draw},
edge1/.style = {>=latex,thick},
edgedash/.style = {>=latex,thick,dashed},
edge2/.style = {>=latex,thick,blue},
edge3/.style = {>=latex,thick,red}
]

\def\layer{1.3}
\def\del{0.75}

\node[pstrong] at (0,0) (bot){$ \bot $};
\node[red] at ({-1*\del},{\layer} ) (12){$ 12 $};
\node[blue] at ({1*\del},{\layer} ) (45){$ 45 $};
\node[green] at (0,{2*\layer} ) (12_45){$ 12/45 $};

\draw [-,edge1](bot) -- (12);
\draw [-,edge1](bot) -- (45);

\draw [-,edge1](12) -- (12_45);
\draw [-,edge1](45) -- (12_45);

\end{tikzpicture} 
			\caption{Lattice $ L $.}
			\label{fig:cir_F_example_LF}
		\end{subfigure}
		\caption{Illustration of a $ cir_L $ function over a suitable lattice of partitions $ L $.}
		\label{fig:cir_F_example}
	\end{figure}
	\hfill $ \square $
\end{exmp}	
\subsection{Lattice quotients}\label{subsec:lattice_quotients}
In this section, we define the quotient operation on sets of partitions. In particular, we show that for lattices of partitions $ L $ with the properties we are interested in ($ \vee_{L} = \vee $ and $ \bot_L = \bot $), all these properties are preserved under the quotient operation.
\begin{defi}
	Consider a set of partitions $ L $ on some set of cells $ \mathcal{C} $. Then for some $ \mathcal{A}\in L $, we define the quotient $ L/\mathcal{A} $ as the set of elements of the form $ \mathcal{B}/ \mathcal{A} $, for all $ \mathcal{B}\in L $ such that $ \mathcal{A}\leq \mathcal{B} $.
	\hfill$ \square $
\end{defi}
\begin{remark}
	$ L/\mathcal{A} $, which is a set of partitions defined on $ \mathcal{C}/\mathcal{A} $, always contains the trivial partition on that set (consider $ \mathcal{B} = \mathcal{A} $).
	\hfill$ \square $
\end{remark}
We now show that if $ L $ is a lattice, then the quotient $ L/\mathcal{A} $ is also a lattice in its own right and its join and meet operations are induced from the join and meet of the original lattice $ L $.
\begin{lemma}
	\label{lemma:lattice_quotient_operations}
	Consider a lattice of partitions $ L $ and some partition $ \mathcal{A}\in L $. Then $ L/\mathcal{A} $ is also a lattice and its join $( \vee_{L/\mathcal{A}} )$ and meet $( \wedge_{L/\mathcal{A}}  )$ operations are given by
	\numparts
	\begin{eqnarray}
	\label{eq:lattice_quotient_join}
	(\mathcal{B}_1/\mathcal{A})\vee_{L/\mathcal{A}} 	(\mathcal{B}_2/\mathcal{A})
	&=
	(\mathcal{B}_1\vee_L\mathcal{B}_2)/\mathcal{A},
	\\
	\label{eq:lattice_quotient_meet}
	(\mathcal{B}_1/\mathcal{A})\wedge_{L/\mathcal{A}} 	(\mathcal{B}_2/\mathcal{A})
	&=
	(\mathcal{B}_1\wedge_L\mathcal{B}_2)/\mathcal{A},
	\end{eqnarray}
	\endnumparts
	for any $ \mathcal{B}_1,\mathcal{B}_2 \in L $ such that $ \mathcal{A} \leq \mathcal{B}_1,\mathcal{B}_2 $, or equivalently, for any $ \mathcal{B}_1/\mathcal{A},\mathcal{B}_2/\mathcal{A} \in L/\mathcal{A} $.
	Furthermore, its top $ (\top_{L/\mathcal{A}} )$ and bottom $( \bot_{L/\mathcal{A}} )$ partitions are given by
	\numparts
	\begin{eqnarray}
	\label{eq:lattice_quotient_top}
	\top_{L/\mathcal{A}} 
	&= 
	\top_{L}/\mathcal{A},
	\\
	\label{eq:lattice_quotient_bot}
	\bot_{L/\mathcal{A}} 
	&= 
	\mathcal{A}/\mathcal{A}.
	\end{eqnarray}
	\endnumparts
	\hfill$ \square $
\end{lemma}
\begin{proof}
	Consider any partition $ \mathcal{P}/\mathcal{A} \in L/\mathcal{A}$ such that $ \mathcal{P}/\mathcal{A} \geq \mathcal{B}_1/\mathcal{A} $ and $ \mathcal{P}/\mathcal{A} \geq \mathcal{B}_2/\mathcal{A} $. From \cref{lemma:partition_quotient_preserves_partial_order}, this is equivalent to saying that $ \mathcal{P} \geq \mathcal{B}_1$ and $ \mathcal{P} \geq \mathcal{B}_2$. Since $ \mathcal{P},\mathcal{B}_1,\mathcal{B}_2 \in L$, this is equivalent to $ \mathcal{P}\geq \mathcal{B}_1\vee_L\mathcal{B}_2 $. Once again from \cref{lemma:partition_quotient_preserves_partial_order}, this is equivalent to $ \mathcal{P}/\mathcal{A} \geq (\mathcal{B}_1\vee_L\mathcal{B}_2)/\mathcal{A} $. Note that $ (\mathcal{B}_1\vee_L\mathcal{B}_2)/\mathcal{A} \in L/\mathcal{A}$. Furthermore, $  (\mathcal{B}_1\vee_L\mathcal{B}_2)/\mathcal{A}  $ is coarser than $ \mathcal{B}_1/\mathcal{A} $ and $ \mathcal{B}_2/\mathcal{A} $ and any partition that is coarser than them has to also be coarser than $  (\mathcal{B}_1\vee_L\mathcal{B}_2)/\mathcal{A} $. Then $  (\mathcal{B}_1\vee_L\mathcal{B}_2)/\mathcal{A} $ is the finest such partition, which means that it corresponds to the join $( \vee_{L/\mathcal{A}}) $ of $ L/\mathcal{A} $, which proves \eref{eq:lattice_quotient_join}. \Eref{eq:lattice_quotient_meet} is proven in a completely analogous way.\\
	Consider any partition $ \mathcal{P}/\mathcal{A} \in L/\mathcal{A} $. Then $ \mathcal{P}\in L $ is such that $ \mathcal{A} \leq \mathcal{P}\leq \top_{L} $. Then from \cref{lemma:partition_quotient_preserves_partial_order}, we have $ \mathcal{A}/\mathcal{A} \leq \mathcal{P}/\mathcal{A}\leq \top_{L}/\mathcal{A} $, which proves \eref{eq:lattice_quotient_top} and \eref{eq:lattice_quotient_bot}.
\end{proof}
\begin{lemma}
	\label{lemma:quotient_join}
	Consider partitions $\mathcal{A},\mathcal{B}_1,\mathcal{B}_2 $ such that $ \mathcal{A} \leq \mathcal{B}_1,\mathcal{B}_2 $. Then
	\begin{eqnarray}
	(\mathcal{B}_1/\mathcal{A})\vee	(\mathcal{B}_2/\mathcal{A})
	&=
	(\mathcal{B}_1\vee\mathcal{B}_2)/\mathcal{A}.
	\end{eqnarray}
	\hfill$ \square $	
\end{lemma}
\begin{proof}
	Firstly, note that both sides describe partitions on the same set. Assume $\mathcal{A},\mathcal{B}_1,\mathcal{B}_2 $ are partitions on a set of cells $ \mathcal{C} $. Then $ \mathcal{B}_1/\mathcal{A} $ and $ \mathcal{B}_2/\mathcal{A} $ are partitions on $ \mathcal{C}/\mathcal{A} $, and so is their join. Therefore the left hand side describes a partition on $ \mathcal{C}/\mathcal{A} $. It is clear that the right hand side is also a partition on $ \mathcal{C}/\mathcal{A} $.\\
	In order to prove that the two partitions are the same, we show that two cells are of the same color in the partition of left hand side if and only if they are also of the same color in the partition of the right hand side. That is,
	\begin{eqnarray*}
	(\mathcal{B}_1/\mathcal{A})\vee	(\mathcal{B}_2/\mathcal{A})(k)
	&=
	(\mathcal{B}_1/\mathcal{A})\vee	(\mathcal{B}_2/\mathcal{A})(l)
	\iff\\
	(\mathcal{B}_1\vee\mathcal{B}_2)/\mathcal{A}(k)
	&=
	(\mathcal{B}_1\vee\mathcal{B}_2)/\mathcal{A}(l)
	\end{eqnarray*}
	for all $ l,k\in \mathcal{C}/\mathcal{A} $.
	From \cref{lemma:join_chain_def}, 
	$ (\mathcal{B}_1/\mathcal{A})\vee	(\mathcal{B}_2/\mathcal{A})(k)
	=
	(\mathcal{B}_1/\mathcal{A})\vee	(\mathcal{B}_2/\mathcal{A})(l) $ is equivalent to the existence of a chain of cells $ k=k_1,\ldots,k_n = l $ in $ \mathcal{C}/\mathcal{A} $ such that, for each $ k_i,k_{i+1} $, with $ 1 \leq i < n $, we have either $ (\mathcal{B}_1/\mathcal{A}) (k_{i}) = (\mathcal{B}_1/\mathcal{A}) (k_{i+1})$ or $ (\mathcal{B}_2/\mathcal{A}) (k_{i}) = (\mathcal{B}_2/\mathcal{A}) (k_{i+1}) $. Note that $ k\in\mathcal{C}/\mathcal{A} $ if and only if there is some $ c\in\mathcal{C} $ such that $ \mathcal{A}(c)=k $. Equivalently, under some cell correspondence $\mathcal{A}(c_i) = k_i$, there is a chain of cells $ c=c_1,\ldots,c_n = d $ in $ \mathcal{C} $ such that, for each $ c_i,c_{i+1} $, with $ 1 \leq i < n $, we have either $ (\mathcal{B}_1/\mathcal{A}) (\mathcal{A}(c_{i})) = (\mathcal{B}_1/\mathcal{A}) (\mathcal{A}(c_{i+1})$ or $ (\mathcal{B}_2/\mathcal{A}) (\mathcal{A}(c_{i})) = (\mathcal{B}_2/\mathcal{A}) (\mathcal{A}(c_{i+1})) $.
	More simply, either $ \mathcal{B}_1 (c_{i}) = \mathcal{B}_1(c_{i+1})$ or $ \mathcal{B}_2 (c_{i}) = \mathcal{B}_2 (c_{i+1}) $. Then from \cref{lemma:join_chain_def} again, this is equivalent to $ \mathcal{B}_1\vee\mathcal{B}_2(c)
	=
	\mathcal{B}_1\vee\mathcal{B}_2(d) $. Since $ \mathcal{A} \leq \mathcal{B}_1,\mathcal{B}_2 $ by assumption, it is always true that $  \mathcal{A} \leq \mathcal{B}_1\vee\mathcal{B}_2 $. Therefore, what we have is equivalent to $( \mathcal{B}_1\vee\mathcal{B}_2)/\mathcal{A}(\mathcal{A}(c))
	=
	( \mathcal{B}_1\vee\mathcal{B}_2)/\mathcal{A}(\mathcal{A}(d)) $, which simplifies to $( \mathcal{B}_1\vee\mathcal{B}_2)/\mathcal{A}(k)
	=
	( \mathcal{B}_1\vee\mathcal{B}_2)/\mathcal{A}(l) $.
\end{proof}
The following is now immediate from \cref{lemma:lattice_quotient_operations,lemma:quotient_join}.
\begin{corollary}
	\label{cor:lattice_quotient_joins_standard}
	Consider a lattice of partitions $ L $, some partition $ \mathcal{A}\in L $ and its respective quotient lattice $ L/\mathcal{A} $. Then for the joins of those lattices, $ \vee_{L} = \vee $ with regard to partitions coarser than $ \mathcal{A} $ if and only if $ \vee_{L/\mathcal{A}} = \vee $. 
	\hfill$ \square $
\end{corollary}
In this work we have a particular interest in lattices of partitions that contain the trivial partition and whose join is determined by the partition join of \cref{lemma:join_chain_def}. We have shown that these properties are preserved under the lattice quotient operation. That is,
\begin{theorem}
	\label{thm:quotient_of_nice_lattice}
	Consider a lattice of partitions $ L $ on a set of cells $ \mathcal{C} $, such that $ \bot_L = \bot_{\mathcal{C}} $ and $ \vee_L = \vee $. Then given any partition $ \mathcal{A}_1\in  L $, we have $ L/\mathcal{A} $ is a lattice on the set $ \mathcal{C}/\mathcal{A} $ such that $ \bot_{L/\mathcal{A}} = \bot_{\mathcal{C}/\mathcal{A}} $ and $ \vee_{L/\mathcal{A}} = \vee $.
	\hfill$ \square $
\end{theorem}
We know from \cref{lemma:bot_join_implies_cir} that lattices with these properties have $ cir $ functions associated to them. We now show how these functions are related.
\begin{lemma}
	\label{lemma:Lattices_w_join_bot_have_cir}
	Consider a lattice of partitions $ L $ such that $ \bot_L = \bot $ and $ \vee_L = \vee $. Then given some partition $ \mathcal{A}\in L $, the lattice $ L/\mathcal{A} $ has a $ cir_{L/\mathcal{A}} : L_{\type{}}/\mathcal{A} \to L/\mathcal{A}$ function, which is related to the $ cir_{L} :L_{\type{}} \to L $ of the original lattice $ L $. In particular, for every $ \mathcal{B}/\mathcal{A}\in L_{\type{}}/\mathcal{A} $, we have
	\begin{eqnarray}
	cir_{L/{\mathcal{A}}}(\mathcal{B}/\mathcal{A})
	=
	cir_{L}(\mathcal{B})/\mathcal{A}. 
	\end{eqnarray}	
	\hfill$ \square $
\end{lemma}
\begin{proof}
	Firstly, since we consider elements $ \mathcal{B}/\mathcal{A}\in L_{\type{}}/\mathcal{A} $, we have $ \mathcal{A} \leq \mathcal{B} $ by assumption.\\
	By definition, $ cir_L(\mathcal{B}) $ is the maximal element of the set $ S: = \left\{ \mathcal{P}\in L: \mathcal{P}\leq \mathcal{B} \right\} $ (which we know exists from \cref{lemma:bot_join_implies_cir}). Then since $ \mathcal{A}\in S $, we have $ cir_L(\mathcal{B}) \geq \mathcal{A} $. Therefore $ cir_L(\mathcal{B})/\mathcal{A} \in L/\mathcal{A}$ exists and it corresponds to the maximal element of $ S/\mathcal{A} $.\\
	On the other hand, $ cir_{L/{\mathcal{A}}}(\mathcal{B}/\mathcal{A}) $ is by definition the maximal term of $ \left\{ \mathcal{P}/\mathcal{A}\in L/\mathcal{A}: \mathcal{P}/\mathcal{A}\leq \mathcal{B}/\mathcal{A} \right\} $, which is again the set $ S/\mathcal{A} $, concluding the proof. 
\end{proof}
\begin{exmp}
	In \fref{fig:Lattice_quotient_example_main} we have a lattice of partitions $ L $, on a set of cells $ \mathcal{C} = \{1,2,3,4\} $, such that $ \bot_{L} = \bot_{\mathcal{C}} $ and $ \vee_{L} = \vee $. Consider the partition $ \mathcal{A} = \{\{1\},\{2,4\},\{3\}\}$, which is in $ L $. We denote the elements of the quotient set $ \mathcal{C}/\mathcal{A} = \{1,24,3\} $ and illustrate the quotient lattice $ L/\mathcal{A} $ in \fref{fig:Lattice_quotient_example_quotient}. Note that $ L/\mathcal{A} $ is also such that $ \bot_{L/\mathcal{A}} = \bot_{\mathcal{C}/\mathcal{A}} $ and $ \vee_{L/\mathcal{A}} = \vee $.
	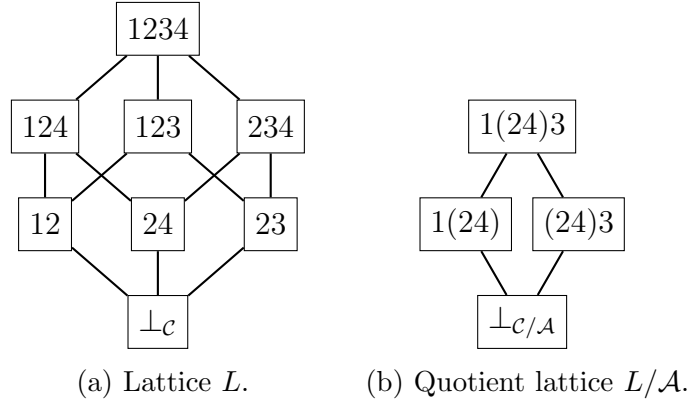
\begin{figure}[h]
		\centering
		\begin{subfigure}[t]{0.3\textwidth}
			\centering
			\begin{tikzpicture}[
node1/.style = {circle,minimum size=23,draw},
node2/.style = {circle,minimum size=23,draw,fill=white!75!black},
node3/.style = {circle,minimum size=23,draw,fill=white!50!black},
lightgray/.style = {rectangle,minimum size=20,draw,draw,fill=white!93!black},
pweak/.style = {rectangle,minimum size=20,draw,draw,fill=white!50!black},
proot/.style = {rectangle,minimum size=20,draw,draw,fill=white!75!black},
pstrong/.style = {rectangle,minimum size=20,draw},
noderect/.style = {rectangle,minimum size=20,draw},
edge1/.style = {>=latex,thick},
edgedash/.style = {>=latex,thick,dashed},
edge2/.style = {>=latex,thick,blue},
edge3/.style = {>=latex,thick,red}
]

\def\layer{1.3}
\def\del{0.75}

\node[noderect] at (0,0) (bot){$ \bot_{\mathcal{C}} $};

\node[noderect] at ({-2*\del},{\layer} ) (12){$ 12 $};
\node[noderect] at ({0*\del},{\layer} ) (24){$ 24 $};
\node[noderect] at ({2*\del},{\layer} ) (23){$ 23 $};

\node[noderect] at ({-2*\del},{2*\layer} ) (124){$ 124 $};
\node[noderect] at ({0*\del},{2*\layer} ) (123){$ 123 $};
\node[noderect] at ({2*\del},{2*\layer} ) (234){$ 234 $};

\node[noderect] at (0,{3*\layer} ) (1234){$ 1234 $};

\draw [-,edge1](bot) -- (12);
\draw [-,edge1](bot) -- (24);
\draw [-,edge1](bot) -- (23);

\draw [-,edge1](12) -- (124);
\draw [-,edge1](12) -- (123);

\draw [-,edge1](24) -- (124);
\draw [-,edge1](24) -- (234);

\draw [-,edge1](23) -- (123);
\draw [-,edge1](23) -- (234);

\draw [-,edge1](124) -- (1234);
\draw [-,edge1](123) -- (1234);
\draw [-,edge1](234) -- (1234);

\end{tikzpicture} 
			\caption{Lattice $ L $.}
			\label{fig:Lattice_quotient_example_main}
		\end{subfigure}
		\begin{subfigure}[t]{0.3\textwidth}
			\centering
			\begin{tikzpicture}[
node1/.style = {circle,minimum size=23,draw},
node2/.style = {circle,minimum size=23,draw,fill=white!75!black},
node3/.style = {circle,minimum size=23,draw,fill=white!50!black},
lightgray/.style = {rectangle,minimum size=20,draw,draw,fill=white!93!black},
pweak/.style = {rectangle,minimum size=20,draw,draw,fill=white!50!black},
proot/.style = {rectangle,minimum size=20,draw,draw,fill=white!75!black},
pstrong/.style = {rectangle,minimum size=20,draw},
noderect/.style = {rectangle,minimum size=20,draw},
edge1/.style = {>=latex,thick},
edgedash/.style = {>=latex,thick,dashed},
edge2/.style = {>=latex,thick,blue},
edge3/.style = {>=latex,thick,red}
]

\def\layer{1.3}
\def\del{0.75}

\node[noderect] at (0,0) (bot){$ \bot_{\mathcal{C}/\mathcal{A}} $};
\node[noderect] at ({-1*\del},{\layer} ) (124){$ 1(24) $};
\node[noderect] at ({1*\del},{\layer} ) (234){$ (24)3 $};
\node[noderect] at (0,{2*\layer} ) (1234){$ 1(24)3 $};

\draw [-,edge1](bot) -- (124);
\draw [-,edge1](bot) -- (234);

\draw [-,edge1](124) -- (1234);
\draw [-,edge1](234) -- (1234);

\end{tikzpicture} 
			\caption{Quotient lattice $ L/\mathcal{A} $.}
			\label{fig:Lattice_quotient_example_quotient}
		\end{subfigure}
		\caption{Illustration of a lattice $ L $ and its quotient lattice $ L/\mathcal{A} $.}
	\end{figure}
	\hfill$ \square $
\end{exmp}
\subsection{Polydiagonals}\label{subsec:polydiagonals}
We now relate a partition that encodes an equality-based synchrony pattern to the corresponding subset of the state set in the network.
\begin{defi}
	Given a partition $ \mathcal{A} \in  L_{\type{}} $, we call the subset of $ \xset $
	\begin{equation}
	\Delta_{\mathcal{A}}^{\xset} 
	:=
	\{ \mathbf{x}\in\xset\colon
	\mathcal{A}(c)
	=
	\mathcal{A}(d)
	\implies x_c = x_d\},
	\end{equation}
	the \textit{polydiagonal} of $ \mathcal{A} $ in $ \xset $.
	\label{defi:polydiag_defi}
	\hfill$ \square $
\end{defi}
This means that any $ \mathbf{x}\in \Delta_{\mathcal{A}}^{\xset} $ can be given by $ \mathbf{x} = P\overline{\mathbf{x}} $ for some $ \overline{\mathbf{x}} $, where $ P $ is a partition matrix of $ \mathcal{A} $.
Consider for instance $ \mathcal{A} = \{\{1,2\},\{3\}\} $, represented by $ P=
\left[\begin{array}{cc} 
1 & 0 \\ 1 & 0 \\ 0 & 1
\end{array} \right]
$. Then $ \mathbf{x} = P\overline{\mathbf{x}}  $ with $ \overline{\mathbf{x}} = 
\left[\begin{array}{c} 
x_{12} \\ x_3 
\end{array} \right]
$ gives us $ \mathbf{x} = 
\left[\begin{array}{c} 
x_{12} \\  x_{12} \\ x_3
\end{array} \right]
$.
\begin{remark}
	If the state sets $ \{\xset_i\}_{i\in\tset} $ have only one element, then it is irrelevant to talk about synchronism in the first place. For this reason, we assume that the state sets are non-empty and non-singleton. That is, we can always choose $ x_c \neq x_d $ with $ x_c,x_d \in\mathbb{X}_i $ for $ i = \type{}(c) = \type{}(d) $.
	\hfill$ \square $
\end{remark}
The partial order relationship between partitions ($ \leq $) induces the following inclusion partial order ($ \subseteq $) between polydiagonals.
\begin{lemma}
	Consider partitions $ \mathcal{A},\mathcal{B}  \in  L_{\type{}} $ and their respective polydiagonals $ \Delta_{\mathcal{A}}^{\xset}, \Delta_{\mathcal{B}}^{\xset} $. Then
	\begin{equation}
	\mathcal{A} \leq \mathcal{B} \Longleftrightarrow \Delta_{\mathcal{A}}^{\xset} \supseteq \Delta_{\mathcal{B}}^{\xset}.
	\end{equation}
	\label{lemma:subset_polydiags}
	\hfill$ \square $
\end{lemma}
\begin{proof}
	The forward direction is direct from \eref{eq:refinement_def} together with \cref{defi:polydiag_defi}. The backwards direction is proved by showing its contrapositive, that is, $\neg(\mathcal{A} \leq \mathcal{B}) \implies \neg(\Delta_{\mathcal{A}}^{\xset} \supseteq \Delta_{\mathcal{B}}^{\xset})$. If $ \neg(\mathcal{A} \leq \mathcal{B}) $, then there are $ c,d \in \mathcal{C} $ such that $ \mathcal{A}(c) = \mathcal{A}(d) $ and $ \mathcal{B}(c) \neq \mathcal{B}(d) $. Then, under the assumption that the state sets are non-singleton, there is $ \mathbf{x}\in \Delta_{\mathcal{B}}^{\xset} $ such that $ x_c\neq x_d $, that is, $ \mathbf{x} \notin \Delta_{\mathcal{A}}^{\xset} $, which proves the contrapositive.
\end{proof}
Moreover, the intersection of two polydiagonals is itself a polydiagonal. In particular, it is related to the join $(\vee) $ operation as follows.
\begin{lemma}
	Given partitions $ \mathcal{A}_1,\mathcal{A}_2   \in  L_{\type{}} $, we have $ \Delta_{\mathcal{A}_1\vee \mathcal{A}_2}^{\xset} =  \Delta_{\mathcal{A}_1}^{\xset}\cap\Delta_{\mathcal{A}_2}^{\xset} $.
	\label{lemma:join_equal_intersection}
	\hfill$ \square $
\end{lemma} 
\begin{proof}
	Since $ \mathcal{A}_1\vee \mathcal{A}_2 $ is coarser than both $ \mathcal{A}_1 $ and $ \mathcal{A}_2 $, we know from \cref{lemma:subset_polydiags} that 
	$ \Delta_{\mathcal{A}_1\vee \mathcal{A}_2}^{\xset} \subseteq \Delta_{\mathcal{A}_1}^{\xset} $ and $ \Delta_{\mathcal{A}_1\vee \mathcal{A}_2}^{\xset} \subseteq \Delta_{\mathcal{A}_2}^{\xset} $. Therefore $ \Delta_{\mathcal{A}_1\vee \mathcal{A}_2}^{\xset} \subseteq \Delta_{\mathcal{A}_1}^{\xset} \cap \Delta_{\mathcal{A}_2}^{\xset}$. \\
	We now prove the converse. Assume $\mathbf{x}\in \Delta_{\mathcal{A}_1}^{\xset} \cap \Delta_{\mathcal{A}_2}^{\xset}$. Then $\mathbf{x}\in \Delta_{\mathcal{A}_1}^{\xset} $ and $\mathbf{x}\in \Delta_{\mathcal{A}_2}^{\xset}$. This implies that for every chain of cells $ c=c_1,\ldots,c_k = d $ such that either $ \mathcal{A}_1(c_{i}) = \mathcal{A}_1(c_{i+1})$ or $ \mathcal{A}_2(c_{i}) = \mathcal{A}_2(c_{i+1}) $, we have $ x_c = x_d $. From \cref{lemma:join_chain_def}, we have $ \mathbf{x}\in \Delta_{\mathcal{A}_1\vee \mathcal{A}_2}^{\xset}$. Therefore $
	\Delta_{\mathcal{A}_1}^{\xset} \cap \Delta_{\mathcal{A}_2}^{\xset}
	\subseteq
	\Delta_{\mathcal{A}_1\vee \mathcal{A}_2}^{\xset} $.
\end{proof}
The union of polydiagonals need not be another polydiagonal. There exists, however, the smallest polydiagonal that contains the union of two polydiagnals. These properties are analogous to the intersection and union of vector subspaces.
\begin{lemma}
	Given partitions $ \mathcal{A}_1,\mathcal{A}_2   \in  L_{\type{}} $, we have $ \Delta_{\mathcal{A} }^{\xset} \supseteq  \Delta_{\mathcal{A}_1}^{\xset}\cup\Delta_{\mathcal{A}_2}^{\xset} $ if and only if $ \mathcal{A} \leq \mathcal{A}_1\wedge \mathcal{A}_2 $.
	\hfill$ \square $
\end{lemma} 
\begin{proof}
	Consider a partition $ \mathcal{A}  \in  L_{\type{}} $ such that $ \Delta_{ \mathcal{A}}^{\xset} \supseteq  \Delta_{\mathcal{A}_1}^{\xset}\cup\Delta_{\mathcal{A}_2}^{\xset} $. Then $ \Delta_{ \mathcal{A}}^{\xset} \supseteq  \Delta_{\mathcal{A}_1}^{\xset} $ and $ \Delta_{ \mathcal{A}}^{\xset} \supseteq  \Delta_{\mathcal{A}_2}^{\xset} $. From \cref{lemma:subset_polydiags}, this means that $ \mathcal{A} \leq \mathcal{A}_1$ and $ \mathcal{A} \leq \mathcal{A}_2$, therefore $ \mathcal{A} \leq \mathcal{A}_1 \wedge \mathcal{A}_2$.\\
	We now prove the converse. It is enough to show that $ \mathcal{A}_1 \wedge \mathcal{A}_2 $ satisfies the inclusion condition since from \cref{lemma:subset_polydiags}, any partition finer than it would also satisfy it.	
	Since $ \mathcal{A}_1\wedge \mathcal{A}_2 $ is finer than both $ \mathcal{A}_1 $ and $ \mathcal{A}_2 $, it follows that 
	$ \Delta_{\mathcal{A}_1 \wedge \mathcal{A}_2}^{\xset} \supseteq \Delta_{\mathcal{A}_1}^{\xset} $ and $ \Delta_{\mathcal{A}_1 \wedge \mathcal{A}_2}^{\xset} \supseteq \Delta_{\mathcal{A}_2}^{\xset} $. Therefore $ \Delta_{\mathcal{A}_1 \wedge \mathcal{A}_2}^{\xset} \supseteq \Delta_{\mathcal{A}_1}^{\xset} \cup \Delta_{\mathcal{A}_2}^{\xset}$.
\end{proof}
\subsection{Invariance of polydiagonals}\label{subsec:invariant_poly}
We now investigate the properties of a function that preserves equality-based synchrony patterns.
\begin{defi}
	If for a $ \mathcal{G} $-admissible function $ f\colon\xset\to\yset $ and a partition $ \mathcal{A}  \in  L_{\type{}}$ we have
	\begin{equation}
	f\left(\Delta_{\mathcal{A}}^{\xset}\right) \subseteq \Delta_{\mathcal{A}}^{\yset},
	\label{eq:polysynchronous}
	\end{equation}
	then $ \mathcal{A} $ is $ f $-\textit{invariant}.\\ Furthermore, if for $ F\subseteq\mathcal{F}_{\mathcal{G}} $, $ \mathcal{A} $ is $ f $-invariant for every $ f\in F $, then we say that $ \mathcal{A} $ is \mbox{$ F $-invariant}.
	\hfill$ \square $
\end{defi}
If $ \mathcal{A} $ is $ f $-invariant, then for every $ \mathbf{x} \in\xset  $ such that $ \mathbf{x} = P\overline{\mathbf{x}} $, with $ P $ representing $ \mathcal{A} $, there is $ \overline{\mathbf{y}} $ such that \mbox{$ f(P\overline{\mathbf{x}}) = P\overline{\mathbf{y}} $}. This means that there is a function $ \overline{f}\colon\overline{\xset}\to\overline{\yset} $ with sets $ \overline{\xset} := \xset^{\overline{\mathbf{k}}} $ and $ \overline{\yset} := \yset^{\overline{\mathbf{k}}} $, for an appropriate $ \overline{\mathbf{k}} \geq \mathbf{0}_{\vert\tset\vert} $, such that
\begin{equation}
f(P\overline{\mathbf{x}})
=
P\overline{f}(\overline{\mathbf{x}}).
\label{eq:f_in_invariant_space}
\end{equation}
Consider again $ \mathcal{A} = \{\{1,2\},\{3\}\} $, represented by $ P=
\left[\begin{array}{cc}
1 & 0 \\ 1 & 0 \\ 0 & 1
\end{array} \right]
$. Then $ \mathcal{A} $ is $ f $-invariant if $ f\left(
\left[\begin{array}{c}
x_{12} \\ x_{12} \\ x_3 
\end{array} \right]
\right) = 
\left[\begin{array}{c}
y_{12} \\ y_{12} \\ y_3 
\end{array} \right]
 $. That is, $ f\left( P 
\left[\begin{array}{c}
x_{12} \\ x_3 
\end{array} \right]
\right) = P
\left[\begin{array}{c}
y_{12} \\ y_3 
\end{array} \right]
$. This means that $ f $ induces a related function $ \overline{f}\left(
\left[\begin{array}{c}
x_{12} \\ x_3 
\end{array} \right]
\right) = 
\left[\begin{array}{c}
y_{12} \\ y_3 
\end{array} \right]
$.
\begin{corollary}
	The trivial partition $ \bot $ is always $ \mathcal{F}_{\mathcal{G}} $-invariant.
	\label{lemma:f_invarian_trivial}
	\hfill$ \square $
\end{corollary}
\begin{lemma}
	Consider partitions $ \mathcal{A}_1 ,\mathcal{A}_2  \in  L_{\type{}}$ and $ f\in \mathcal{F}_{\mathcal{G}}$ such that $ \mathcal{A}_1,\mathcal{A}_2  $ are both \mbox{$ f $-invariant}. Then $ \mathcal{A}_1 \vee \mathcal{A}_2 $ is also $ f $-invariant.
	\label{lemma:f_invarian_join}
	\hfill$ \square $
\end{lemma}
\begin{proof}
	Take any $ \mathbf{x}\in \Delta_{\mathcal{A}_1 \vee \mathcal{A}_2}^{\xset} = \Delta_{\mathcal{A}_1}^{\xset} \cap \Delta_{\mathcal{A}_2}^{\xset}$. Then $ \mathbf{x}\in  \Delta_{\mathcal{A}_1}^{\xset} $
	and $ \mathbf{x}\in  \Delta_{\mathcal{A}_2}^{\xset} $.	
	By assumption, we have $ f\left(\mathbf{x}\right) \in \Delta_{\mathcal{A}_1}^{\yset} $ and $ f\left(\mathbf{x}\right) \in \Delta_{\mathcal{A}_2}^{\yset} $, that is, $ f\left(\mathbf{x}\right) \in \Delta_{\mathcal{A}_1}^{\yset} \cap \Delta_{\mathcal{A}_2}^{\yset} = \Delta_{\mathcal{A}_1 \vee \mathcal{A}_2}^{\yset}$. Therefore $ \mathcal{A}_1 \vee \mathcal{A}_2 $ is \mbox{$ f $-invariant}.
\end{proof}
\begin{corollary}
	Consider partitions $ \mathcal{A}_1 ,\mathcal{A}_2  \in  L_{\type{}}$ and $ F\subseteq \mathcal{F}_{\mathcal{G}}$ such that $ \mathcal{A}_1,\mathcal{A}_2  $ are both \mbox{$ F $-invariant}. Then $ \mathcal{A}_1 \vee \mathcal{A}_2 $ is also $ F $-invariant.
	\label{lemma:F_invarian_join}
	\hfill$ \square $
\end{corollary}
\begin{proof}
	By definition, $ \mathcal{A}_1,\mathcal{A}_2 $ being $ F $-invariant implies that they are $ f $-invariant for every $ f\in F $. Then from \cref{lemma:f_invarian_join}, $ \mathcal{A}_1 \vee \mathcal{A}_2 $ is also $ f $-invariant for every $ f\in F $. That is, $ \mathcal{A}_1 \vee \mathcal{A}_2 $ is $ F $-invariant.
\end{proof}
\begin{remark}
	Being interested only in a particular subset of admissible functions $ F\subseteq \mathcal{F}_{\mathcal{G}} $ is quite natural. In particular, the definition of $ \mathcal{F}_{\mathcal{G}} $ does not include any type of smoothness assumption. In general, we could be interested in admissible functions that are constructed through of oracle components that have more properties than the minimal ones described in \cref{defi:oracle}. For instance, an oracle component such that a cell becomes insensitive to cells that are on the same state, corresponds, under the current formalism, to the following constraint.
	\begin{eqnarray}
	\label{eq:oracle_exo}
	\hat{f}_i
	\left(x; 
	\left[\begin{array}{c}
	w_{i_1} \\ 
	\mathbf{w}
	\end{array} \right]
	,
	\left[\begin{array}{c}
	x \\ 
	\mathbf{x}
	\end{array} \right]
	\right)
	=
	\hat{f}_i
	\left(x; 
	\mathbf{w},
	\mathbf{x}
	\right).
	\end{eqnarray}
	This assumption is present, for instance in the Kuramoto model.	This makes the cells of such a system always insensitive to self-loops.
	\hfill$ \square $
\end{remark}
We can now show that the sets of $ F $-invariant partitions form lattices.
\begin{theorem}
	Denote by $ L_{F} $ the subset of partitions in $ L_{\type{}} $ that are $ F $-invariant, with $ F\subseteq \mathcal{F}_{\mathcal{G}} $.  Then $ L_{F} $ is a lattice whose minimal element $ \bot_F $ is the trivial partition $ \bot $ and whose join operation $ \vee_F $ is the partition join $ \vee $ as described in \cref{lemma:join_chain_def}.
	\label{lemma:f_set_invariant_lattice}
	\hfill$ \square $
\end{theorem}
\begin{proof}
	Since $ L_{\type{}} $ is finite, then $ L_{F} $ is also finite. From \cref{lemma:f_invarian_trivial}, $ \bot \in L_F $ for all $ F\subseteq \mathcal{F}_{\mathcal{G}} $. Since $ \bot $ is the finest partition, $ \bot_F = \bot $.\\
	Consider any $ \mathcal{A}_1,\mathcal{A}_2 \in L_{F} $. Then from \cref{lemma:F_invarian_join}, $ \mathcal{A}_1 \vee \mathcal{A}_2 \in L_{F} $. Any partition coarser than $ \mathcal{A}_1 $ and $ \mathcal{A}_2 $ has to be coarser than $ \mathcal{A}_1 \vee \mathcal{A}_2 $. Therefore $ \vee_F = \vee $.
	Then from \cref{lemma:minimal_join_equal_lattice}, $ L_{F} $ is a lattice.
\end{proof}
\begin{remark}
	Note that $ L_{\emptyset} = L_{\type{}} $ since being $ \emptyset $-invariant is vacuously satisfied.
	\hfill$ \square $
\end{remark}
\begin{corollary}
	Denote by $ L_f $ (instead of by $ L_{\{f\}} $) the subset of partitions in $ L_{\type{}} $ that are $ f $-invariant, with $ f \in \mathcal{F}_{\mathcal{G}} $. Then for all $ F\subseteq \mathcal{F}_{\mathcal{G}} $, we have $ L_F = \bigcap_{f\in F} L_f $.
	\hfill$ \square $
\end{corollary}
\begin{corollary}
	\label{lemma:invariant_lattice_properties}
	For every $ F_1, F_2 \subseteq \mathcal{F}_{\mathcal{G}} $, we have
	\begin{enumerate}
		\item If $ F_1 \subseteq F_2 $, then $ L_{F_1} \supseteq L_{F_2} $.
		\label{lemma:monotonic_F_invariance}
		\item $ L_{F_1\cup F_2} = L_{F_1} \cap L_{F_2} $.
		\item $ L_{F_1\cap F_2} \supseteq L_{F_1} \cup L_{F_2} $.
	\end{enumerate} 
	\hfill$ \square $
\end{corollary}
From \cref{lemma:monotonic_F_invariance} of \cref{lemma:invariant_lattice_properties}, $ L_{\mathcal{F}_{\mathcal{G}}} $ is the smallest possible lattice of invariant partitions.\\
We have shown in \cref{lemma:bot_join_implies_cir} that for a lattice $ L $ such that $ \bot_L = \bot $ and $ \vee_L = \vee $, there exists of a function $ cir_L $ that assigns to each element in $ L_{\type{}} $ an element of $ L $. Since every $ F $-invariant lattice satisfies these assumptions, we have the following.
\begin{corollary} 
	Consider a $ F $-invariant lattice $ L_F $, with $ F\subseteq \mathcal{F}_{\mathcal{G}} $. Given any partition $ \mathcal{A}\in L_{\type{}} $, there is a partition $ \mathcal{B}\in L_F $ that is the coarsest one in $ L_F $ such that $ \mathcal{B} \leq \mathcal{A}$. This establishes the function $ cir _F \colon L_{\type{}} \to L_F $.
	\label{lemma:cir_F_invariant}
	\hfill $ \square $
\end{corollary}
\begin{corollary}
	Consider partitions $ \mathcal{A}_1 ,\mathcal{A}_2  \in  L_F$, with $ F\subseteq \mathcal{F}_{\mathcal{G}} $.
	Then $ \mathcal{A}_1 \wedge_F \mathcal{A}_2 = cir_F(\mathcal{A}_1 \wedge \mathcal{A}_2)$.
	\label{cor:meet_cir_relation}
	\hfill $ \square $
\end{corollary}
In summary, the join operation $ (\vee) $ described in \cref{lemma:join_chain_def} is fundamental with regard to the study of invariance in polydiagonals. In particular, it corresponds to the fact that the intersection of invariant polydiagonals gives us another invariant polydiagonal. On the other hand, the meet operation is not fixed. It is dependent on the particular lattice $ L $ and does not present a clear intuitive meaning. In fact, from \cref{lemma:minimal_join_equal_lattice}, its existence can be seen as a mere consequence of a minimal partition $ \bot_L $ together with some join operation $ \vee_L $. Since $ \vee_L = \vee $ for all the lattices we are interested in ($ F $-invariant lattices), the join operation is the most convenient of the two fundamental operations on lattices and we focus on it in this work.
\subsection{Balanced partitions}\label{subsec:balanced}
We now show that if the connectivity structure of a network $ \mathcal{G} $ respects certain conditions, it enforces certain polydiagonals to be invariant, regardless of the particular choice of admissible $ f\in\mathcal{F}_\mathcal{G} $.
\begin{defi}\label{defi:balanced_cond}
	Consider a network $ \mathcal{G} $ defined on a cell set $ \mathcal{C} $ with a cell type partition $ \type{} $ and an in-adjacency matrix $ M $. A partition $ \mathcal{A} \in L_{\type{}}$ with characteristic matrix $ P $ is said to be \textit{balanced} on $ \mathcal{G} $ if for all $ c,d\in\mathcal{C} $
	\begin{eqnarray}\label{eq:balanced_cond}
	\mathcal{A}(c)
	=
	\mathcal{A}(d)
	\implies
	\mathbf{m}_c P = \mathbf{m}_d P,
	\end{eqnarray}
	where $ \mathbf{m}_c, \mathbf{m}_d $ are the rows of matrix $ M $ corresponding to cells $ c $ and $ d $, respectively.
	\hfill$ \square $
\end{defi}
This property means that cells of the same color have equivalent colored in-neighborhoods. Furthermore, a partition is balanced if and only if there is a matrix $ Q $ of elements in the appropriate monoids $ \{\mathcal{M}_{ij}\}_{i,j\in\tset} $ such that
\begin{equation}\label{eq:balanced_cond_matrix_form}
M P = P Q.
\end{equation}
The matrix $ Q $ has particular significance and corresponds to the in-adjacency matrix of a network described in \cref{defi:quotient_network}.
\\
A balanced partition is usually indicated by the symbol $ \bp $ and we denote the set of all balanced partitions in a given network $ \mathcal{G} $ by $ \Lambda_{\mathcal{G}} $.\\
In \cite{stewart2007lattice} it was shown that for the unweighted formalism, $ \Lambda_{\mathcal{G}} $ forms a lattice under the partition refinement relation ($ \leq $), as described in \eref{eq:refinement_def}. We show that this follows easily from the results in \cref{subsec:invariant_poly}.
\begin{corollary}
	The trivial partition $ \bot $ is always balanced.
	\label{lemma:trivial_balanced}
	\hfill$ \square $
\end{corollary}
\begin{proof}
	For any $ M $, the condition \eref{eq:balanced_cond_matrix_form} is satisfied with $ P = I $ and $ Q = M $.
\end{proof}
The following was proven in \cite{sequeira2021commutative} for the weighted formalism. 
\begin{lemma}
	Consider balanced partitions $ \bp_1, \bp_2 \in\Lambda_{\mathcal{G}} $. Then $\bp_1 \vee \bp_2$ is also balanced.
	\label{lemma:join_balanced}
	\hfill $ \square $
\end{lemma}
Using \cref{lemma:minimal_join_equal_lattice} again, the following is an immediate consequence of \cref{lemma:trivial_balanced} and \cref{lemma:join_balanced}.
\begin{corollary}
	Given a network $ \mathcal{G} $, the set of balanced partitions $\Lambda_{\mathcal{G}} $ forms a lattice whose minimal element $ \bot_{\mathcal{G}} $ is the trivial partition $ \bot $ and whose join operation $ \vee_{\mathcal{G}} $ is the partition join $ \vee $ as described in 	\cref{lemma:join_chain_def}.
	\label{lemma:balanced_is_lattice}
	\hfill$ \square $
\end{corollary}
From \cref{lemma:bot_join_implies_cir}, the following is immediate.
\begin{corollary}
	Given any partition $ \mathcal{A}\in L_{\type{}} $, there is a balanced partition $ \bp \in \Lambda_{\mathcal{G}} $ that is the coarsest one in $ \Lambda_{\mathcal{G}} $ such that $ \bp \leq \mathcal{A}$.
	\label{lemma:cir_balanced}
	\hfill $ \square $
\end{corollary}
This implies the existence of a $ cir $ function from $ L_{\type{}} $ to $ \Lambda_{\mathcal{G}} $, which we denote by just $ cir $. Then we have the following.
\begin{corollary}
	Consider balanced partitions $ \bp_1, \bp_2 \in\Lambda_{\mathcal{G}} $.
	Then $ \bp_1 \wedge_{\mathcal{G}} \bp_2 = cir(\bp_1 \wedge \bp_2)$.
	\label{cor:meet_cir_relation_balanced}
	\hfill $ \square $
\end{corollary}
The particular $ cir $ function associated with $ \Lambda_{\mathcal{G}} $ is easy to compute and was extended in \cite{sequeira2021commutative} for the general weighted case.\\
In order to present the interesting properties of balanced partitions, we require the following result, which relates partitions and oracle components.
\begin{lemma}\label{lemma:partition_merge}
	For any oracle component $ \hat{f}_{i} \in \hat{\mathcal{F}}_{i} $, we have 
	\begin{eqnarray}
	\label{eq:partition_merge}
	\hat{f}_i(x;
	\mathbf{w},
	P\overline{\mathbf{x}}) 
	= 
	\hat{f}_i(x;
	P^{\top}\mathbf{w},
	\overline{\mathbf{x}}),
	\end{eqnarray}
	where $ P $ is a partition matrix of appropriate dimensions such that the vectors $ \mathbf{w} $ and $ P\overline{\mathbf{x}} $ have elements of matching cell types.
	\hfill$ \square $
\end{lemma}
\begin{remark}
	\cref{lemma:partition_merge} is valid for all inputs such that the evaluation is meaningful. That is, whenever the domain \eref{eq:oracle_domain_func} is respected. Furthermore, vectors $ \mathbf{w} $ and $ P\overline{\mathbf{x}} $ having elements of matching cell types is equivalent to $ P^{\top}\mathbf{w} $ and $ \overline{\mathbf{x}} $ having elements of matching cell types and the sum $ P^{\top}\mathbf{w} $ being well-defined. That is, each sum operates on elements of the same commutative monoid.
	\hfill$ \square $
\end{remark}
\begin{proof}[Proof of \cref{lemma:partition_merge}]
	The proof is by induction. In the base case, 	partition matrices of dimension $ n\times n $ with $ n>0 $ are in fact permutations, therefore it is direct from \cref{thm:oracle_permut} of \cref{defi:oracle}. Consider now fixed integers $ n,k $ such that $ 0<k<n $. Assume \eref{eq:partition_merge} applies to all partition matrices of dimension $ n\times (k+1) $ as long as it is applied to suitable (type matching) $ \mathbf{w} $ and $ \overline{\mathbf{x}} $. Any partition matrix $ P $ of dimension $ n\times k $ can be obtained by taking some partition matrix $ \overline{P} $ of dimension $ n\times (k+1) $ and merging together two of its columns. That is, $ P = \overline{P} p \permut $, with $ p =
	\left[\begin{array}{cc}
	1 & \mathbf{0}^{\top} \\ 1 & \mathbf{0}^{\top} \\ \mathbf{0} &\mathbf{I}_{k-1}
	\end{array} \right]
	 $ and where $ \permut $ is a permutation matrix of dimension $ k\times k $.\\ 
	Consider one such $ P $ and any suitable $ \mathbf{w} $ and $ \overline{\mathbf{x}} $. Then $ \hat{f}_i(x;
	\mathbf{w},
	P\overline{\mathbf{x}}) 
	=
	\hat{f}_i(x;
	\mathbf{w},
	\overline{P}(p \permut \overline{\mathbf{x}})) 
	$. By assumption, we call apply \eref{eq:partition_merge} with respect to $ \overline{P} $, which gets us $ \hat{f}_i(x;\overline{P}^{\top}\mathbf{w},p\permut \overline{\mathbf{x}}) $. Due to the particular shape of $ p $, applying \eref{eq:partition_merge} with regard to $ p $ is equivalent to \cref{thm:oracle_merge} of \cref{defi:oracle}. This gives us $ \hat{f}_i(x;p^{\top}\overline{P}^{\top}\mathbf{w},\permut \overline{\mathbf{x}}) $. Similarly, we can apply \eref{eq:partition_merge} with regard to $ \permut $ since it corresponds to \cref{thm:oracle_permut} of \cref{defi:oracle}. Since $ \permut $ is a permutation matrix, we have that $ \permut^{-1} = \permut^{\top} $. Therefore this becomes $ \hat{f}_i(x;\permut^{\top}p^{\top}\overline{P}^{\top}\mathbf{w}, \overline{\mathbf{x}}) 
	=
	\hat{f}_i(x;(\overline{P}p\permut)^{\top}\mathbf{w}, \overline{\mathbf{x}})
	=
	\hat{f}_i(x;P^{\top}\mathbf{w}, \overline{\mathbf{x}})
	$, which proves that \eref{eq:partition_merge} is satisfied for any partition matrix $ P $ of size $ n\times k $.
\end{proof}
A more convoluted version of \cref{lemma:partition_merge} was originally part of the definition of oracle components in \cite{sequeira2021commutative}. We had it in place of \cref{thm:oracle_permut,thm:oracle_merge} in \cref{defi:oracle}. We just proved \cref{lemma:partition_merge} using \cref{thm:oracle_permut,thm:oracle_merge}. Furthermore, it is straightforward that these items are just particular cases of \cref{lemma:partition_merge}. Therefore the two definitions are equivalent.\\
Having proven this, we can now state the following result, which underlines the importance of balanced partitions in the study of invariance.
\begin{theorem}\label{thm:balanced_implies_all}
	Consider a balanced partition $ \bp\in\Lambda_{\mathcal{G}} $ on a network $ \mathcal{G} $ and any $ \mathcal{G} $-admissible function $ f\in\mathcal{F}_{\mathcal{G}} $. Then $ \bp $ is $ f $-invariant. 
	\hfill$ \square $
\end{theorem}
\begin{proof}
	Consider any $ \bp \in \Lambda_{\mathcal{G}} $ and a state in the related polydiagonal $ \mathbf{x}\in\Delta_{\bp}^{\xset} $. That is, $ \mathbf{x} = P\overline{\mathbf{x}} $ for some $ \overline{\mathbf{x}} $, where $ P $ is a partition matrix of $ \bp $.\\
	For any pair of cells $ c,d \in \mathcal{C}$ such that $ \bp(c) = \bp(d) $, we have $ x_c =x_d $. Furthermore, from \cref{defi:balanced_cond}, we have $ P^{\top}\mathbf{m}_c^{\top} = P^{\top}\mathbf{m}_d^{\top} $. Therefore
	$
	\hat{f}_{i}\left(x_c; P^{\top}\mathbf{m}_c^{\top}, \overline{\mathbf{x}} \right)
	=
	\hat{f}_{i}\left(x_d; P^{\top}\mathbf{m}_d^{\top}, \overline{\mathbf{x}} \right)$ for any $ \hat{f}_{i} \in \hat{\mathcal{F}}_{i} $, with $ i = \type{}(c) = \type{}(d)$.\\
	Using \cref{lemma:partition_merge}, this becomes $ \hat{f}_{i}\left(x_c; \mathbf{m}_c^{\top}, P\overline{\mathbf{x}} \right)
	=
	\hat{f}_{i}\left(x_d; \mathbf{m}_d^{\top}, P\overline{\mathbf{x}} \right) $, which from \cref{defi:F_G_admissibility} is equivalent to $ f_c(P\overline{\mathbf{x}}) 
	= 
	f_d(P\overline{\mathbf{x}}) $. This means that for every $ \mathcal{G} $-admissible function $ f\in\mathcal{F}_{\mathcal{G}} $, there is a $ \overline{f} $ such that $ f(P\overline{\mathbf{x}}) = P\overline{f}(\overline{\mathbf{x}}) $. That is, $ \bp $ is $f $-invariant.
\end{proof}
This is equivalent to the following statement.
\begin{corollary}\label{cor:balanced_implies_all_simple}
	Given a network $ \mathcal{G} $, we have $  \Lambda_{\mathcal{G}} \subseteq L_{F}  $, for any $ F\subseteq \mathcal{F}_{\mathcal{G}} $.
	\hfill$ \square $
\end{corollary}
\begin{theorem}\label{thm:F_G_implies_balanced}
	Consider a partition $ \mathcal{A}\leq \type{} $ on some network $ \mathcal{G} $. If $ \mathcal{A} $ is $ \mathcal{F}_{\mathcal{G}} $-invariant, then $ \mathcal{A} $ is balanced on $ \mathcal{G} $.
	\hfill$ \square $
\end{theorem}
This is equivalent to the following statement.
\begin{corollary}
	\label{cor:F_G_implies_balanced_simple}
	Given a network $ \mathcal{G} $, we have $  \Lambda_{\mathcal{G}} \supseteq L_{\mathcal{F}_{\mathcal{G}}}  $.
	\hfill$ \square $
\end{corollary}	
Note that $ L_{\mathcal{F}_{\mathcal{G}}} $ is the smallest possible lattice of invariant partitions. In \cite{stewart2003symmetry,golubitsky2005patterns,golubitsky2006nonlinear}, \cref{thm:F_G_implies_balanced} was derived by proving a stronger result: there exists some subset $ F\subseteq \mathcal{F}_{\mathcal{G}} $ such that $  \Lambda_{\mathcal{G}} \supseteq L_{F} $. This type of result is of interest since we might be interested only in certain subclasses of admissible functions and not the full $ \mathcal{F}_{\mathcal{G}} $.\\
This stronger result, which was originally hidden away in their proof of the unweighted version of \cref{thm:F_G_implies_balanced}, was made explicit and generalized in \cite{sequeira2021commutative} for the general weighted formalism.\\
From \cref{cor:balanced_implies_all_simple,cor:F_G_implies_balanced_simple} the following is now immediate.
\begin{corollary}
	Given a network $ \mathcal{G} $, we have $  \Lambda_{\mathcal{G}} = L_{\mathcal{F}_{\mathcal{G}}}  $.
	\hfill$ \square $
\end{corollary}
\subsection{Quotient networks}\label{subsec:quotients}
In this section we describe how the behavior of a network $ \mathcal{G} $ when evaluated at some polydiagonal $ \Delta_\bp^{\xset}$ for some balanced partition $ \bp $ can be described by a smaller network $ \mathcal{Q} $.
\begin{defi}\label{defi:quotient_network}
	Consider a network $ \mathcal{G} $ defined on a cell set $ \mathcal{C}_{\mathcal{G}} $ with a cell type partition $ \type{G} $ and an in-adjacency matrix $ M $. Take a balanced partition $ \bp \in \Lambda_{\mathcal{G}} $.\\
	The \textit{quotient network} $ \mathcal{Q} $ of $ \mathcal{G} $ over $ \bp $, denoted $ \mathcal{Q} := \mathcal{G}/\bp $, is defined on a cell set $ \mathcal{C}_{\mathcal{Q}} := \mathcal{C}_{\mathcal{G}}/\bp $ with a cell type partition $ \type{Q} := \type{G}/\bp $ and an in-adjacency matrix $ Q $ given by $ M P = P Q $, where $ P $ represents $ \bp $.
	\hfill$ \square $
\end{defi}
\begin{remark}
	Assume that a particular ordering has been chosen for the sets of cells $ \mathcal{C}_{\mathcal{G}}$ and $ \mathcal{C}_{\mathcal{Q}} $. Then the partition $ P $ representing $ \bp $ and the in-adjacency  matrices $ M $ and $ Q $ are uniquely defined.
	\hfill$ \square $
\end{remark}
The following result corresponds to the merging of Lemmas 5.8. and 5.9. in \cite{sequeira2021commutative}, which we now present in a cleaner way.
\begin{lemma}\label{lemma:balanced_quotient_invariance}
	Consider a balanced partition $ \bp_{01}\in \Lambda_{\mathcal{G}_0}$ on a network $ \mathcal{G}_0 $ and its respective quotient network $ \mathcal{G}_1 = \mathcal{G}_0/\bp_{01} $. For a partition $ \bp_{02} $ such that $ \bp_{01}\leq \bp_{02} $, define $ \bp_{12} :=\bp_{02} /\bp_{01} $. Then $ \bp_{02}\in \Lambda_{\mathcal{G}_0} $ if and only if $ \bp_{12} \in \Lambda_{\mathcal{G}_1} $. Furthermore, if $ \bp_{02} $ and $ \bp_{12} $ satisfy this, then $ \mathcal{G}_0 / \bp_{02} = \mathcal{G}_{1} / \bp_{12} $.
	\hfill $ \square $
\end{lemma}
\begin{proof}
	Denote by $ {\type{}}_0  $, $ {\type{}}_1  $, the cell type partitions of networks $ \mathcal{G}_0 $, $ \mathcal{G}_1 $, respectively. Then $ {\type{}}_1  = {\type{}}_0 /\bp_{01} $ by definition.
	From \cref{lemma:partition_quotient_canceling}, we know $ {\type{}}_1/\bp_{12} = ({\type{}}_0 /\bp_{01})/(\bp_{02} /\bp_{01}) = {\type{}}_0 /\bp_{02}$. \\
	Consider now that $ M_{0} $, $ M_1 $ are the in-adjacency matrices of $ \mathcal{G}_0 $, $ \mathcal{G}_1 $, respectively, and that $ P_{01} $, $ P_{02} $, $ P_{12} $ are the partition matrices of $ \bp_{01} $, $ \bp_{02} $, $ \bp_{12} $. Then from $  \bp_{12} =\bp_{02} /\bp_{01} $, we have $ P_{02} = P_{01} P_{12} $. Moreover, since $ \bp_{01}\in \Lambda_{\mathcal{G}_0}$, we have $ M_0 P_{01} = P_{01} M_1$.\\
	In order to show that $ \bp_{02}\in \Lambda_{\mathcal{G}_0} $ if and only if $ \bp_{12} \in \Lambda_{\mathcal{G}_1} $, we prove
	\begin{equation*}	
	M_0 P_{02} = P_{02} M_2
	\Longleftrightarrow
	M_1 P_{12} = P_{12} M_2.
	\end{equation*}
	Expanding $ P_{02} $ in the left hand side, this becomes
	\begin{equation*}	
	M_0 P_{01} P_{12} = P_{01} P_{12} M_2
	\Longleftrightarrow
	M_1 P_{12} = P_{12} M_2.
	\end{equation*}
	From $ M_0 P_{01} = P_{01} M_1$, this can be written as
	\begin{equation*}	
	P_{01} (M_1 P_{12}) = P_{01} (P_{12} M_2)
	\Longleftrightarrow
	M_1 P_{12} = P_{12} M_2,
	\end{equation*}
	which is now clear since $ P_{01} $ has full column rank, that is, it is left invertible.
	If there is a matrix $ M_2 $ that satisfies these expressions, then the network $ \mathcal{G}_2 $ defined by the in-adjacency matrix $ M_2 $ and the cell type partition $ {\type{}}_2 = {\type{}}_1 / \bp_{12} = {\type{}}_0 / \bp_{02}$, is such that $ \mathcal{G}_2 = \mathcal{G}_0 / \bp_{02} = \mathcal{G}_{1} / \bp_{12} $.
\end{proof}
From \cref{thm:balanced_implies_all}, any balanced partition $ \bp \in \Lambda_{\mathcal{G}} $ is $ f $-invariant for any $ f\in\mathcal{F}_{\mathcal{G}} $. In \eref{eq:f_in_invariant_space} it was shown that for a partition $ \mathcal{A} $ and a function $ f $ such that $ \mathcal{A} $ is $ f $-invariant, then $ f $, when evaluated on $ \Delta_{\mathcal{A}}^{\xset} $ can be determined by a simpler function $ \overline{f} $. We will see that for the case of balanced partitions this function is particularly noteworthy.
\begin{defi}
	\label{thm:f_g_balanced_quotient}
	Consider a network $ \mathcal{G} $ and a balanced partition $ \bp \in \Lambda_{\mathcal{G}} $. Let $ f\in\mathcal{F}_{\mathcal{G}} $.
	The \textit{quotient function} $ g := f/\bp $ is defined by restricting $ f $ to the polydiagonal $ \Delta_\bp^{\xset}$. That is,
	\begin{eqnarray}
	f(P\overline{\mathbf{x}}) = Pg(\overline{\mathbf{x}}),
	\end{eqnarray}
	where P is a partition matrix of $ \bp $.
	\hfill$ \square $
\end{defi}
We now show that the quotient function is intimately related to the quotient network.
\begin{theorem}\label{thm:f_g_project_oracle}
	Consider networks $ \mathcal{G} $ and $\mathcal{Q} $ such that $ \mathcal{Q} = \mathcal{G}/\bp $ for some balanced partition $ \bp \in \Lambda_{\mathcal{G}} $. Then for any $ f\in\mathcal{F}_{\mathcal{G}} $, which is given by $ f = \left.\hat{f}\right|_{\mathcal{G}} $, for some $ \hat{f}\in\hat{\mathcal{F}}_{T} $, its quotient function $ g = f/\bp $ is given by $ g = \left.\hat{f}\right|_{\mathcal{Q}} $. Therefore $ g\in \mathcal{F}_{\mathcal{Q}}$.
	\hfill$ \square $
\end{theorem}
\begin{proof}
	By \cref{thm:f_g_balanced_quotient}, we have $ g_k(\overline{\mathbf{x}}) = f_c(P\overline{\mathbf{x}}) $ for all $ c \in \mathcal{C}_{\mathcal{G}}$ and $ k = \bp(c) $. Then from $ \mathcal{G} $-admissibility, $ f_c(P\overline{\mathbf{x}}) = \hat{f}_{i}\left(\overline{x}_k; \mathbf{m}^{\top}_c, P\overline{\mathbf{x}} \right)$, with $ i = \type{G}(c) $. From \cref{lemma:partition_merge}, $ \hat{f}_{i}\left(\overline{x}_k; \mathbf{m}^{\top}_c, P\overline{\mathbf{x}} \right) = \hat{f}_{i}\left(\overline{x}_k; P^{\top}\mathbf{m}^{\top}_c,\overline{\mathbf{x}} \right) $, and since $ \bp $ is balanced, this is equal to $ \hat{f}_i(\overline{x}_k;\mathbf{q}^{\top}_k,\overline{\mathbf{x}}) $. That is, $ g_k(\overline{\mathbf{x}}) =  \hat{f}_i(\overline{x}_k;\mathbf{q}^{\top}_k,\overline{\mathbf{x}}) $ for all $ k\in\mathcal{C}_{\mathcal{Q}} $ with $ i = \type{Q}(k) $. Therefore $ g =  (g_k)_{k\in\mathcal{C}_{\mathcal{Q}}}$ is $ \mathcal{Q} $-admissible, with $ g = \left.\hat{f}\right|_{\mathcal{Q}} $.
\end{proof}
The following is now immediate from \cref{thm:f_g_balanced_quotient} and \cref{thm:f_g_project_oracle}.
\begin{corollary}\label{cor:quotient_oracle_independence}
	Consider networks $ \mathcal{G} $ and $\mathcal{Q} $ such that $ \mathcal{Q} = \mathcal{G}/\bp $ for some balanced partition $ \bp \in \Lambda_{\mathcal{G}} $. Then for any $ \hat{f}^1,\hat{f}^2 \in\hat{\mathcal{F}}_{T} $, we have
	\begin{equation*}
	\left.\hat{f}^1\right|_{\mathcal{G}} = \left.\hat{f}^2\right|_{\mathcal{G}} \implies
	\left.\hat{f}^1\right|_{\mathcal{Q}} = \left.\hat{f}^2\right|_{\mathcal{Q}}.
	\end{equation*}
	\hfill$ \square $
\end{corollary}
Now that we understand the relationship between $ f\in\mathcal{F}_{\mathcal{G}} $ and its quotient $ g =f/\bp $ in terms of oracle functions, the following is clear.
\begin{corollary}
	\label{cor:f_g_lift}
	Consider networks $ \mathcal{G} $ and $\mathcal{Q} $ such that $ \mathcal{Q} = \mathcal{G}/\bp $ for some balanced partition $ \bp \in \Lambda_{\mathcal{G}} $. Then for any $ g\in\mathcal{F}_{\mathcal{Q}} $, there is some $ f\in\mathcal{F}_{\mathcal{G}} $ such that $ g = f/\bp $.
	\hfill$ \square $
\end{corollary}
\cref{cor:f_g_lift} refers only to existence, not to uniqueness. That is, it could be possible to have $ f_1,f_2 \in \mathcal{F}_{\mathcal{G}}$ such that $  f_1\neq f_2  $ but $ g = f_1/\bp = f_2/\bp$. However, they have to match when evaluated on the polydiagonal $ \Delta_\bp^{\xset} $.
\begin{exmp}
	Consider the given partition $ \mathcal{A} = \{ \{1,2\},\{3\} \} $ on the \textit{CCN} of \cref{exmp:admiss_funcs} (\fref{fig:Gadmissibility}). One partition matrix of $ \mathcal{A} $ is 
	\begin{equation}P =
	\left[\begin{array}{cc}
	1 & 0\\
	1 & 0\\
	0 & 1
	\end{array} \right]
	,
	\end{equation}
	in which each column identifies one of the colors of the partition.
	From this we obtain the product
	\begin{equation}MP =
	\left[\begin{array}{cc}
	1 & 1\\
	1 & 1\\
	2 & 1
	\end{array} \right]
	.
	\label{eq:MP_example}
	\end{equation}
	Rows $ 1 $ and $ 2 $ are the same and the respective cells are of the same cell type. That means that for any admissible $ f $ we have $ f_1(\mathbf{x}) = f_2(\mathbf{x}) $ when $ x_1 = x_2 $.\\
	Observe that this is in agreement with the functional form we wrote in \eref{eq:func_match1},\eref{eq:func_match2}.\\
	Since the rows of $ MP $ respect an equality relationship according to $ \mathcal{A} $, then $ \mathcal{A} $ is balanced and there is a quotient matrix $ Q $ that obeys the balanced condition \eref{eq:balanced_cond_matrix_form}. In fact, the quotient matrix $ Q  $ is
	\begin{equation}
	Q =
	\left[\begin{array}{cc}
	1 & 1\\
	2 & 1
	\end{array} \right]
	,
	\end{equation}
	which is directly obtained from $ MP $ by compressing its rows according to $ \mathcal{A} $.\\
	The behavior of this \textit{CCN} when $ x_1 = x_2 $ is then described by the smaller \textit{CCN} given by the quotient matrix $ Q $ which is represented in \fref{fig:Gadmissibility_quotientb}.
	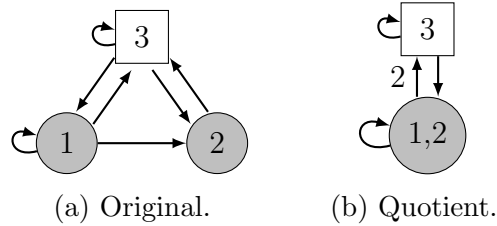
\begin{figure}[h]
		\centering
		\begin{subfigure}[t]{0.23\textwidth}
			\centering
			\begin{tikzpicture}[
node1/.style = {circle,minimum size=23,draw},
node2/.style = {circle,minimum size=23,draw,fill=white!75!black},
node3/.style = {circle,minimum size=23,draw,fill=white!50!black},
noderect/.style = {rectangle,minimum size=20,draw},
edge1/.style = {>=latex,thick},
edge2/.style = {>=latex,thick,blue},
edge3/.style = {>=latex,thick,red}
]
\node[node2] at (0,0)(n1){1};
\node[node2] at (2,0)(n2){2};
\node[noderect] at (1,{sqrt(2)})(n3){3};

\DoubleLine{n1}{n3}{<-,edge1}{}{->,edge1}{}
\draw [->,edge1](n1) -- (n2);
\DoubleLine{n2}{n3}{<-,edge1}{}{->,edge1}{}

\draw [->,edge1] (n1) edge[loop left,looseness=5] (n1);
\draw [->,edge1] (n3) edge[loop left,looseness=5] (n3);

%\node (minus123) at ($(n1)!0.4!(n2) + (0,0.3)$) {(-)};

%\DoubleLine{n1}{n2}{<-,edge1}{}{->,edge1}{}
%\DoubleLine{n1}{n3}{<-,edge1}{}{->,edge1}{}
%\DoubleLine{n2}{n3}{<-,edge1}{}{->,edge1}{}

\end{tikzpicture} 
			\caption{Original.}
			\label{fig:Gadmissibility_quotienta}
		\end{subfigure}
		\begin{subfigure}[t]{0.23\textwidth}
			\centering
			\begin{tikzpicture}[
node1/.style = {circle,minimum size=23,draw},
node2/.style = {circle,minimum size=23,draw,fill=white!75!black},
node3/.style = {circle,minimum size=23,draw,fill=white!50!black},
noderect/.style = {rectangle,minimum size=20,draw},
edge1/.style = {>=latex,thick},
edge2/.style = {>=latex,thick,blue},
edge3/.style = {>=latex,thick,red}
]
\node[node2] at (0,0)(n1){1,2};
\node[noderect] at (0,{sqrt(2)})(n3){3};
%\node[node1] at (1,{sqrt(2)})(n3){3};

%\DoubleLine{n1}{n3}{<-,edge1}{}{->,edge1}{}
\DoubleLine{n1}{n3}{->,edge1}{2}{<-,edge1}{}
%\draw [->,edge1](n3) -- (n1);
%\DoubleLine{n2}{n3}{<-,edge1}{}{->,edge1}{}

\draw [->,edge1] (n1) edge[loop left,looseness=5] (n1);
\draw [->,edge1] (n3) edge[loop left,looseness=5] (n3);

%\node (minus123) at ($(n1)!0.4!(n2) + (0,0.3)$) {(-)};

%\DoubleLine{n1}{n2}{<-,edge1}{}{->,edge1}{}
%\DoubleLine{n1}{n3}{<-,edge1}{}{->,edge1}{}
%\DoubleLine{n2}{n3}{<-,edge1}{}{->,edge1}{}

\end{tikzpicture} 
			\caption{Quotient.}
			\label{fig:Gadmissibility_quotientb}
		\end{subfigure}
		\caption{Color-coded network of \fref{fig:Gadmissibility} and its quotient over the balanced partition $ \{ \{1,2\},\{3\} \} $.}
		\label{fig:Gadmissibility_quotient}
	\end{figure}
	The coloring is a way of representing the partition $ \mathcal{A} = \{ \{1,2\},\{3\} \} $ over which the quotient is done. In both \fref{fig:Gadmissibility_quotienta} and \fref{fig:Gadmissibility_quotientb} each gray cell receives one connection from a gray cell and one connection from a white cell. On the other hand, each white cell receives a connection from a white cell and two connections from a gray cell.
	The function $ g = f/\bp $ has the following structure
	\begin{eqnarray}
	g_{12}(\mathbf{x}) &= \hat{f}_1(x_{12};
	\left[\begin{array}{cc}
	 1 & 1
	\end{array} \right]
	^{\top},\mathbf{x}),\\
	g_3(\mathbf{x}) &= \hat{f}_2(x_3;
	\left[\begin{array}{cc}
	 2 & 1
	 \end{array} \right]
	 ^{\top},\mathbf{x}),
	\end{eqnarray}
	where $ \hat{f}\in\hat{\mathcal{F}}_{T}  $ is any oracle function such that $ f = \left.\hat{f}\right|_{\mathcal{G}} $.
	\hfill$ \square $
\end{exmp}
We now extend the concept of quotient of admissible functions to sets of admissible functions.
\begin{defi}
	Consider networks $ \mathcal{G} $ and $ \mathcal{Q} $ such that $ \mathcal{Q} = \mathcal{G}/\bp $ for some $ \bp\in \Lambda_{\mathcal{G}} $. Given any subset of $ \mathcal{G} $-admissible functions $ F_{\mathcal{G}}\subseteq \mathcal{F}_{\mathcal{G}} $, define its quotient $ F_\mathcal{Q} = F_\mathcal{G}/\bp $ to be the subset of $ \mathcal{F}_{\mathcal{Q}} $ such that $ g\in F_\mathcal{Q}  $ if and only if there is some $ f\in F_{\mathcal{G}} $ such that $ g = f/\bp  $.
	\hfill$ \square $
\end{defi}
From \cref{cor:f_g_lift} it is immediate that $ \mathcal{F}_{\mathcal{G}}/\bp = \mathcal{F}_{\mathcal{Q}} $. That is, $ \mathcal{F}_{\mathcal{G}}/\bp = \mathcal{F}_{\mathcal{G}/\bp} $.\\
We are now ready to study the relation between the invariant lattices $ L_{F_{\mathcal{G}}} $ of a network $ \mathcal{G} $ and corresponding invariant lattice of its quotient network $ \mathcal{Q} = \mathcal{G}/\bp $.\\
The following is direct from \cref{lemma:balanced_quotient_invariance}.
\begin{corollary}
	\label{lemma:LG_LQ_relation_balanced}
	Consider networks $ \mathcal{G} $ and $ \mathcal{Q} $ such that $ \mathcal{Q} = \mathcal{G}/\bp $ for some $ \bp\in \Lambda_{\mathcal{G}} $. Then $ \Lambda_{\mathcal{Q}} =  \Lambda_{\mathcal{G}}/\bp $.
	\hfill$ \square $
\end{corollary}
We now generalize this to lattices of $ F $-invariant partitions.
\begin{theorem}
	\label{thm:LG_LQ_relation}
	Consider networks $ \mathcal{G} $ and $ \mathcal{Q} $ such that $ \mathcal{Q} = \mathcal{G}/\bp $ for some $ \bp\in \Lambda_{\mathcal{G}} $, and subsets $ F_{\mathcal{G}}\subseteq \mathcal{F}_{\mathcal{G}} $, $ F_{\mathcal{Q}}\subseteq \mathcal{F}_{\mathcal{Q}} $ such that $ F_\mathcal{Q} = F_\mathcal{G}/\bp $.\\
	Then for partitions $ \mathcal{A}_{\mathcal{G}} \leq \type{G}$ and $ \mathcal{A}_{\mathcal{Q}} \leq \type{Q}$ such that $ \mathcal{A}_{\mathcal{G}} \geq \bp $ and $ \mathcal{A}_{\mathcal{Q}}  = \mathcal{A}_{\mathcal{G}}/\bp  $, we have $ \mathcal{A}_{\mathcal{G}} \in L_{F_{\mathcal{G}}} $ if and only if $ \mathcal{A}_{\mathcal{Q}} \in L_{F_{\mathcal{Q}}} $.	That is, $  L_{F_{\mathcal{Q}}} = L_{F_{\mathcal{G}}} /\bp  $.
	\hfill$ \square $
\end{theorem}
\begin{proof}
	We consider the partitions $ \bp $, $ \mathcal{A}_{\mathcal{G}} $ and $ \mathcal{A}_{\mathcal{Q}} $ to be represented by partition matrices $ P $, $ P_{\mathcal{G}} $ and $ P_{\mathcal{Q}} $ respectively, such that $ P_{\mathcal{G}} = P P_{\mathcal{Q}} $.\\ 
	Assume $ \mathcal{A}_{\mathcal{G}} \in L_{F_{\mathcal{G}}} $. For any $ g\in F_{\mathcal{Q}} $, there is some $ f\in F_{\mathcal{G}} $ such that $ g=f/\bp $. Then since $ \bp $ is balanced, it follows from \cref{thm:f_g_balanced_quotient} that $ f(P_{\mathcal{G}}\overline{\mathbf{x}}) = f(PP_{\mathcal{Q}}\overline{\mathbf{x}}) = Pg(P_{\mathcal{Q}}\overline{\mathbf{x}})$.
	On the other hand, since $ \mathcal{A}_{\mathcal{G}} \in L_{F_{\mathcal{G}}} $, it follows that $ f(P_{\mathcal{G}}\overline{\mathbf{x}}) = P_{\mathcal{G}}\overline{f}(\overline{\mathbf{x}})  = P P_{\mathcal{Q}}\overline{f}(\overline{\mathbf{x}}) $ for some $ \overline{f} $. Therefore $ Pg(P_{\mathcal{Q}}\overline{\mathbf{x}}) =  P P_{\mathcal{Q}}\overline{f}(\overline{\mathbf{x}})$. Since $ P $ always has full column rank, it is left-invertible, which means that $ g(P_{\mathcal{Q}}\overline{\mathbf{x}}) =  P_{\mathcal{Q}}\overline{f}(\overline{\mathbf{x}})$. That is, $ \mathcal{A}_{\mathcal{Q}} $ is $ g $-invariant for any $ g\in F_{\mathcal{Q}} $, so $ \mathcal{A}_{\mathcal{G}} \in L_{F_{\mathcal{G}}} $ implies $ \mathcal{A}_{\mathcal{Q}} \in L_{F_{\mathcal{Q}}} $. We now prove the converse direction.\\
	Assume $ \mathcal{A}_{\mathcal{Q}} \in L_{F_{\mathcal{Q}}} $. For any $ f\in F_{\mathcal{G}} $, its quotient $ g=f/\bp $ is in $ F_{\mathcal{Q}} $. Then from the fact that $ \mathcal{A}_{\mathcal{Q}} \in L_{F_{\mathcal{Q}}} $ it follows that $ g(P_{\mathcal{Q}}\overline{\mathbf{x}}) = P_{\mathcal{Q}} \overline{g}(\overline{\mathbf{x}}) $ for some $ \overline{g} $. Multiplying on the left by $ P $ gives us $ Pg(P_{\mathcal{Q}}\overline{\mathbf{x}}) = PP_{\mathcal{Q}} \overline{g}(\overline{\mathbf{x}}) 
	=
	P_{\mathcal{G}} \overline{g}(\overline{\mathbf{x}})
	$. On the other hand, since $ \bp $ is balanced, \cref{thm:f_g_balanced_quotient} implies that $ Pg(P_{\mathcal{Q}}\overline{\mathbf{x}}) =
	f(PP_{\mathcal{Q}}\overline{\mathbf{x}}) = 
	f(P_{\mathcal{G}}\overline{\mathbf{x}})$. Therefore $ 
	f(P_{\mathcal{G}}\overline{\mathbf{x}}) = P_{\mathcal{G}} \overline{g}(\overline{\mathbf{x}}) $. That is, $ \mathcal{A}_{\mathcal{G}} $ is $ f $-invariant for any $ f\in F_{\mathcal{G}} $, so $ \mathcal{A}_{\mathcal{Q}} \in L_{F_{\mathcal{Q}}} $ implies $ \mathcal{A}_{\mathcal{G}} \in L_{F_{\mathcal{G}}} $, which completes the proof.
\end{proof}
This is in agreement with \cref{lemma:LG_LQ_relation_balanced} when we consider the particular case $ F_{\mathcal{G}} = \mathcal{F}_{\mathcal{G}} $ and $ F_{\mathcal{Q}} = \mathcal{F}_{\mathcal{Q}} $.
The following is now immediate from \cref{thm:LG_LQ_relation,lemma:Lattices_w_join_bot_have_cir}.
\begin{corollary}
	Consider networks $ \mathcal{G} $ and $ \mathcal{Q} $ such that $ \mathcal{Q} = \mathcal{G}/\bp $ for some $ \bp\in \Lambda_{\mathcal{G}} $, and subsets $ F_{\mathcal{G}}\subseteq \mathcal{F}_{\mathcal{G}} $, $ F_{\mathcal{Q}}\subseteq \mathcal{F}_{\mathcal{Q}} $ such that $ F_\mathcal{Q} = F_\mathcal{G}/\bp $.\\
	Then for $ \mathcal{A} \leq \type{G}$ such that $ \mathcal{A} \geq \bp $, we have
	\begin{eqnarray}
	cir_{F_{\mathcal{G}}}(\mathcal{A})/\bp = cir_{F_{\mathcal{Q}}}(\mathcal{A}/\bp). 
	\end{eqnarray}	
	\hfill$ \square $
\end{corollary}
\section{Network connectivity}\label{sec:net_connect}
In this section we summarize the definitions and notation necessary to study the connectivity of a directed network and relate those characteristics to its dynamics.
\subsection{Neighborhoods and reachability}\label{subsec:neigh_reach}
\begin{defi}
	The \textit{in-neighborhood} $ \mathcal{N}^{-} $ of a cell $ c \in \mathcal{C}$, is the subset of cells $ d \in \mathcal{C}$ such that the set of directed edges from $ d $ to $ c $ is not equivalent to an edge with zero weight.  
	Similarly, its out-neighborhood, denoted $ \mathcal{N}^{+}(c) $, is the subset of cells $ d \in \mathcal{C}$ such that the set of directed edges from $ c $ to $ d $ is not equivalent to an edge with zero weight.
	\hfill $ \square $
\end{defi}
In our context, this means that if $ M $ is an in-adjacency matrix of a network, then $ \mathcal{N}^{-}(c) = \{d\in\mathcal{C}\colon m_{cd} \neq 0_{ij} , i =\type{}(c) , j =\type{}(d)\} $. The commutative monoid structure allows us to encode arbitrary (finite) edges from a cell $ d $ to a cell $ c $ using a single element. This definition says that even if there are non-zero edges from $ d $ to $ c $, if their total effect is equivalent to a non-edge ($ 0_{ij} $), then $ d $ is not in $ \mathcal{N}^{-}(c) $.
\begin{remark}
	We often denote $ c\in \mathcal{N}^{-}(d) $, or equivalently, $ d\in \mathcal{N}^{+}(c) $, by $ c \edge d $.
	\hfill $ \square $
\end{remark}  
\begin{defi}
	The \textit{cumulative in-neighborhood} $ \mathcal{V}^{-} $ of a cell $ c \in \mathcal{C}$, is defined as $ \mathcal{V}^{-}(c) := c\cup \mathcal{N}^{-}(c) $.
	\hfill $ \square $
\end{defi}
\begin{defi}
	The $ k^{th} $ cumulative in-neighborhood $ \mathcal{V}_{k}^{-} $ of a cell $ c \in \mathcal{C}$, is defined recursively as
	\begin{eqnarray}
	\mathcal{V}_{0}^{-}(c) &:= c,
	\\
	\mathcal{V}_{k}^{-}(c) &:=
	\bigcup_{d\in\mathcal{V}_{k-1}^{-}(c)} \mathcal{V}^{-}(d)
	,\quad k > 0.
	\label{eq:in_cum_neigh_recurs}
	\end{eqnarray}
	That is, the set of cells from which there is a directed path of at most $ k $  edges that ends at $ c $. Note that $\mathcal{V}_1^{-} = \mathcal{V}^{-} $. The $ k^{th} $ cumulative \mbox{out-neighborhood} $ \mathcal{V}_{k}^{+}$ is defined similarly by replacing the signs.
	\hfill $ \square $
\end{defi}
\begin{lemma}
	The sequence $ \left(\mathcal{V}_{k}^{-} \right)_{k\geq 0}$ is monotonically increasing, that is,
	\begin{eqnarray*}
	\mathcal{V}_{k}^{-}(c) \subseteq \mathcal{V}_{k+1}^{-}(c)
	,\quad k\geq 0.
	\end{eqnarray*}
	Moreover, if $ \mathcal{V}_{k}^{-}(c) = \mathcal{V}_{k+1}^{-}(c) $ for some $ k\geq 0 $, then the recursion \eref{eq:in_cum_neigh_recurs} has reached a fixed point, which means that $ \mathcal{V}_{k}^{-}(c) = \mathcal{V}_{n}^{-}(c) $ for all $ n\geq k $.
	\hfill $ \square $
	\label{lemma:cum_sum_monotonic_conv}
\end{lemma}
This result motivates the following definition.
\begin{defi}
	The \textit{in-reachability} $ \mathcal{R}^{-} $ of a cell $ c \in \mathcal{C}$, is defined as
	\begin{eqnarray}
	\mathcal{R}^{-}(c)
	:= 
	\bigcup_{k \geq 0}
	\mathcal{V}_k^{-}(c).
	\label{eq:in_reachable}
	\end{eqnarray}
	That is, the set of cells from which there is a finite directed path that ends at $ c $.\\
	The \mbox{out-reachability} $ \mathcal{R}^{+} $ is defined similarly by replacing the signs.
	\hfill $ \square $
\end{defi}
\begin{remark}
	We often denote $ c\in\mathcal{R}^{-}(d) $, or equivalently, $ d\in\mathcal{R}^{+}(c) $, by $ c \dpath d $, illustrating that there is a direct path starting at cell $ c $ and ending at cell $ d $.
	\hfill $ \square $
\end{remark}
\begin{corollary}
	For any cell $ c $ we have $ \mathcal{V}_{k}^{-}(c) \subseteq \mathcal{R}^{-}(c) $ for all $ k\geq 0 $. Moreover, when considering a finite amount of cells, equality is achieved at some finite $ k $.
	\label{cor:cumulative_neigh_k_inside_reachable}
	\hfill $ \square $
\end{corollary}
\begin{corollary}
	If $ c\in\mathcal{R}^{-}(d) $, then \mbox{$ \mathcal{R}^{-}(c) \subseteq \mathcal{R}^{-}(d)$}. That is, if $ c\dpath d $, then for every cell $ e $ such that $ e\dpath c $ we also have $ e\dpath d $.
	\label{lemma:reachable_inclusion}
	\hfill $ \square $
\end{corollary}
\begin{exmp}
	\begin{figure}[h]
		\centering
		\begin{tikzpicture}[
node1/.style = {circle,minimum size=23,draw},
node2/.style = {circle,minimum size=23,draw,fill=white!75!black},
node3/.style = {circle,minimum size=23,draw,fill=white!50!black},
noderect/.style = {rectangle,minimum size=20,draw},
edge1/.style = {>=latex,thick},
edge2/.style = {>=latex,thick,blue},
edge3/.style = {>=latex,thick,red}
]
\node[node1] at (0,0)(n1){1};
\node[node1] at (1.5,0)(n2){2};
\node[node1] at (3,0)(n3){3};
\node[node1] at (4.5,0)(n4){4};

%\DoubleLine{n1}{n3}{<-,edge1}{}{->,edge1}{}
\draw [->,edge1](n1) -- (n2);
\draw [->,edge1](n2) -- (n3);
\draw [->,edge1](n3) -- (n4);
%\DoubleLine{n2}{n3}{<-,edge1}{}{->,edge1}{}

%\draw [->,edge1] (n1) edge[loop left,looseness=5] (n1);
%\draw [->,edge1] (n3) edge[loop above,looseness=5] (n3);

%\node (minus123) at ($(n1)!0.4!(n2) + (0,0.3)$) {(-)};

%\DoubleLine{n1}{n2}{<-,edge1}{}{->,edge1}{}
%\DoubleLine{n1}{n3}{<-,edge1}{}{->,edge1}{}
%\DoubleLine{n2}{n3}{<-,edge1}{}{->,edge1}{}

\end{tikzpicture} 
		\caption{Simple chain of 4 cells.}
		\label{fig:CCN_chain}
	\end{figure}
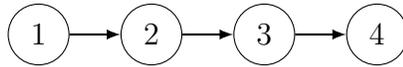
	Consider the simple network in \fref{fig:CCN_chain}. Cell $ 3 $ receives an edge from cell $ 2 $, that is, $ \mathcal{N}^{-}(3) = \{2\} $. Its cumulative in-neighborhood is $  \mathcal{V}^{-}(3) = 3\cup \mathcal{N}^{-}(3) = \{2,3\} $. Using the definition, its second cumulative in-neighborhood is $ \mathcal{V}_{2}^{-}(3) = \mathcal{V}^{-}(2) \cup \mathcal{V}^{-}(3) $, which results in $  \{1,2\} \cup \{2,3\} = \{1,2,3\} $, which are the cells that have a directed path to cell $ 3 $ with a length of two or less. This is already the maximal cumulative in-neighborhood of cell $ 3 $ since $ \mathcal{V}_{3}^{-}(3) = \mathcal{V}^{-}(1) \cup \mathcal{V}^{-}(2) \cup \mathcal{V}^{-}(3) $ again equals $  \{1,2,3\} $. That is, $ \mathcal{V}_{2}^{-}(3) = \mathcal{R}^{-}(3) $.\\
	Furthermore, the point at which the cumulative in-neighborhoods equals the in-reachability set depends on the particular cell of the network. For instance, $ \mathcal{V}_{0}^{-}(1) = \mathcal{R}^{-}(1) = \{1\}$ and $ \mathcal{V}_{3}^{-}(4) = \mathcal{R}^{-}(4) = \{1,2,3,4\}$.\\
	Finally, $ \mathcal{R}^{-}(1) \subset \mathcal{R}^{-}(2) \subset  \mathcal{R}^{-}(3) \subset \mathcal{R}^{-}(4)$ since each cell has a direct path to every cell that is identified with an higher number. In particular, the set inclusions are strict, that is, there are no two cells with the same in-reachability set. This would require directed loops, that is, $ \mathcal{R}^{-}(c) = \mathcal{R}^{-}(d) $ is equivalent to $ \mathcal{R}^{-}(c) \subseteq \mathcal{R}^{-}(d) $ and $ \mathcal{R}^{-}(d) \subseteq \mathcal{R}^{-}(c) $, which implies $ c\in \mathcal{R}^{-}(d) $ and $ d\in \mathcal{R}^{-}(c) $. That is, $ c\dpath d $ and $ d\dpath c $.
	\hfill $ \square $
\end{exmp}
\subsection{Dynamics from in-neighborhoods}
\label{subsec:dynamics_neigh}
Consider a network $ \mathcal{G} $ and an \mbox{$ \mathcal{G} $-admissible} state set $ \xset $ such that a state $ \mathbf{x}\in\xset $ evolves (either discretely or continuously) according to a \mbox{$ \mathcal{G} $-admissible} function $ f\in\mathcal{F}_\mathcal{G} $. That is,
\begin{equation}
\mathbf{x}^{+}/\dot{\mathbf{x}}
=
f(\mathbf{x}).
\label{eq:discrete_continuous_dynamics}
\end{equation}
From the definition of admissibility, the component $ f_c $ of an \mbox{$ \mathcal{G} $-admissible} function $ f $ is dependent only on the states associated with the cells in $ \mathcal{V}^{-}(c)$. This allows us to relate the dynamics of the system to the neighborhoods of cells.
We now show how $ \mathcal{V}_{k}^{-} $ in particular is related to the evolution of an admissible system in both the discrete and continuous cases. 
\begin{theorem}
	Consider a network that evolves discretely according to a function $ f\in\mathcal{F}_\mathcal{G} $. Then $ x_c[n],x_c[n+1],\ldots, x_c[n+k] $ are fully determined by the set of states $\{x_d[n]\}  $, with $ d\in\mathcal{V}_{k}^{-}(c) $.
	\hfill $ \square $
	\label{thm:evo_neigh_discrete}
\end{theorem}
\begin{proof}
	It is enough to just prove that $ x_c[n+k] $ is fully determined, the rest comes directly from the monotonicity of $ \left(\mathcal{V}_{k}^{-} \right)_{k\geq 0}$.\\
	The proof is by induction. The base case $ k=0 $ is trivial. Assume this to be true for some $ k\geq 0 $. Then $x_c[n+k+1] $ is fully determined by the set of states $\{x_d[n+1]\}  $ with $ d\in\mathcal{V}_{k}^{-}(c) $. From $ f $ being \mbox{$ \mathcal{G} $-admissible}, the states $\{x_d[n+1]\} $ themselves are fully determined by $ \{x_e[n]\}  $ with $ e\in\mathcal{V}_{1}^{-}(d) $ for each $ d\in\mathcal{V}_{k}^{-}(c) $. This means that $ x_c[n+k+1] $ is fully determined by the states $ \{x_d[n]\} $ with $ d\in\mathcal{V}_{k+1}^{-}(c) $, which proves the induction step.
\end{proof}
\begin{theorem}
	Consider a system that evolves continuously according to a function $ f\in\mathcal{F}_\mathcal{G} $. Then, assuming sufficient differentiability, the derivatives up to $ k^{th} $ order at time $ t $, that is, $ x_c(t) ,\dot{x}_c(t) ,\ldots,x_c^{(k)}(t) $ are fully determined by the set of states $ \{x_d(t)\} $, with $ d\in\mathcal{V}_{k}^{-}(c) $.
	\label{thm:evo_neigh_continuous}
	\hfill $ \square $
\end{theorem}
\begin{proof}
	It is enough to just prove that $ x_c^{(k)}(t) $ is fully determined, the rest comes directly from the monotonicity of $ \left(\mathcal{V}_{k}^{-} \right)_{k\geq 0}$.\\
	The proof is by induction. The base case $ k=0 $ is trivial. Assume this to be true for some $ k\geq 0 $. Then there is a function $ g $ such that
	\begin{eqnarray*}
	x_c^{(k)}(t) = g(\{x_d(t) \colon d\in\mathcal{V}_{k}^{-}(c)\}).
	\end{eqnarray*}
	That is, $ x_c^{(k)}(t)$ is fully determined by the set of states $ \{x_d(t) \} $ with $ d\in\mathcal{V}_{k}^{-}(c) $. Differentiating on both sides gives
	\begin{eqnarray*}
	x_c^{(k+1)}(t) &= \sum_{d\in\mathcal{V}_{k}^{-}(c)}\frac{\partial g}{\partial x_d} x_d^{(1)}(t).
	\end{eqnarray*}
	From $ f $ being \mbox{$ \mathcal{G} $-admissible}, the first derivatives $ \{x_d^{(1)}(t)\} $ are fully determined by $\{x_e(t)\} $ with \mbox{$ e\in\mathcal{V}_{1}^{-}(d) $} for each $ d\in\mathcal{V}_{k}^{-}(c) $. This means that $ x_c^{(k+1)}(t) $ is fully determined by the states $\{x_d(t)\}  $ with \mbox{$ d\in\mathcal{V}_{k+1}^{-}(c) $}. 
\end{proof}
We now show that knowledge about the \mbox{in-reachability} $ \mathcal{R}^{-} $ of a cell fully defines its evolution.
\begin{theorem}
	Consider a network that evolves either discretely or continuously, according to a function $ f\in\mathcal{F}_\mathcal{G} $. Then the whole trajectory $ \left(x_c[k]\right)_{k\geq n} / x_c(\cdot) $ is fully determined by the set of states $ x_d[n] /x_d(t) $ for $ d\in\mathcal{R}^{-}(c) $.
	\label{thm:Rin_subsystem}
	\hfill $ \square $
\end{theorem}
\begin{proof}
	From \cref{lemma:reachable_inclusion}, for any \mbox{in-reachability} set $ \mathcal{R}^{-}(c) = \mathcal{S} $, any cell $ d\in \mathcal{S} $ has its own \mbox{in-reachability} contained within that same set. That is, \mbox{$ \mathcal{R}^{-}(d) \subseteq \mathcal{S}$}. Since $ \mathcal{V}^{-}(d) \subseteq \mathcal{R}^{-}(d) $, we have $ \mathcal{V}^{-}(d) \subseteq \mathcal{S} $.\\
	By admissibility, the dynamics of a cell $ d $ is a function of the states of the cells in $ \mathcal{V}^{-}(d) $. Therefore we can constrain our network to the subset of cells $ \mathcal{S} $ while preserving all their dependencies within that same set. That is, knowledge about the initial conditions of the cells $ \mathcal{S} $ is enough to fully determine the evolution of the induced subsystem.
\end{proof}
\begin{remark}
	For the discrete time case (\cref{thm:evo_neigh_discrete}), this result is direct from \cref{cor:cumulative_neigh_k_inside_reachable}. However, to extend the continuous time case (\cref{thm:evo_neigh_continuous}) in the same manner, we would have to require the dynamics to be analytic, which is usually too much to ask for. Often, only the Lipschitz condition is assumed. Our approach in the previous proof works for both the discrete and continuous cases.
	\hfill $ \square $
\end{remark}
\begin{corollary}
	Consider a subset of cells $ \mathcal{S} $ in a network that is an \mbox{in-reachability} set. That is, $ \mathcal{S} = \mathcal{R}^{-}(c) \subseteq \mathcal{C} $ for some $ c\in\mathcal{C} $. Then for any solution $ \mathbf{x}(t) $ of the whole system, restricting $ \mathbf{x}(t) $ to the cells in $ \mathcal{S} $ gives us a valid solution to the subnetwork induced by $ \mathcal{S} $. Conversely, for a solution $ \mathbf{x}_{\mathcal{S}}(t) $ on the subnetwork, there is a solution on the whole network that is an extension of it.
	\label{cor:Rin_extension_constrain}
	\hfill $ \square $
\end{corollary}
\begin{proof}
	This is direct from \cref{thm:Rin_subsystem}. 
\end{proof}
\subsection{Strongly connected components and root dependency}\label{subsec:scc_rdc}
To study the \mbox{in-reachability} sets $ \mathcal{R}^{-} $ of the network, it is useful to decompose its graph into strongly connected components (\textit{SCC}).
\begin{defi}
	Two cells $ c,d\in\mathcal{C} $ are said to be \textit{strongly connected} if \mbox{$ \mathcal{R}^{-}(c) = \mathcal{R}^{-}(d) $}. That is, there are directed paths $ d \dpath c $ and $ c\dpath d $.
	\hfill $ \square $
\end{defi}
\begin{remark}
	The strongly connected property induces a partition on the set of cells $ \mathcal{C} $. The subsets of this partition are called the SCCs.
	\hfill $ \square $
\end{remark}
Since two cells in the same SCC have exactly the same \mbox{in-reachability} set, that is $ \mathcal{R}^{-}(c) = \mathcal{R}^{-}(d) $ for all $ c,d\in\mathcal{S}_i $, we simply refer to it as $ \mathcal{R}^{-}(\mathcal{S}_i) $.
\begin{defi}
	The \textit{condensation graph} is obtained by representing each SCC $ \mathcal{S}_i $ by a block and connecting $ \mathcal{S}_i \edge \mathcal{S}_j $ for $ i\neq j $, if there are cells \mbox{$ c_i\in\mathcal{S}_i$}, \mbox{$ c_j\in\mathcal{S}_j $} such that $ c_i\edge c_j $. 
	\hfill $ \square $
\end{defi}
The diagram obtained is \textit{blockwise acyclic}. For $ c_i\in\mathcal{S}_i, c_j\in\mathcal{S}_j $, $ i\neq j $, the existence of a directed path $ c_i\dpath c_j $ is equivalent to the existence of a directed path $ \mathcal{S}_i \dpath \mathcal{S}_j $ in the condensation graph. Moreover, if in the condensation graph there is a direct path $ \mathcal{S}_i \dpath \mathcal{S}_j $ then $ \mathcal{S}_i\subseteq \mathcal{R}^{-}(\mathcal{S}_j) $.\\
This decomposition can be done very efficiently in time $ O(|\mathcal{C}| + |\mathcal{E}|) $, where $ \mathcal{E} $ denotes the set of edges, using for instance Tarjan's algorithm \cite{tarjan1972depth}.\\
Building on the concept of SCCs, we are now ready to define a decomposition based on root dependency components (\textit{RDC}).
\begin{defi}\label{defi:root}
	An SCC $ \mathcal{S}_i $ is a \textit{root} if there are no other SCCs that have a directed path to it. That is, $ \mathcal{S}_i = \mathcal{R}^{-}(\mathcal{S}_i) $.
	\hfill $ \square $
\end{defi}
\begin{defi}
	Two cells $ c,d\in\mathcal{C} $ have the same \textit{root dependency} if $ \mathcal{R}^{-}(c) $, $\mathcal{R}^{-}(d) $ contain \textit{exactly} the same subset of roots.
	\hfill $ \square $
\end{defi}
\begin{remark}
	The property of having the same root dependency induces a partition on the set of cells $ \mathcal{C} $. The subsets of this partition are called the \textit{root dependency components}.
	Moreover, in network with $ n $ roots, this partition has at most $ 2^n -1 $ disjoint subsets, since there is no cell that does not depend on any root.
	\hfill $ \square $
\end{remark}
The following is straightforward from the definitions.
\begin{corollary}
	The partition formed by the SCCs is finer than the one formed by the RDCs.
	\hfill $ \square $
\end{corollary}
\begin{exmp}\label{exmp:SCC_decomp_example}
	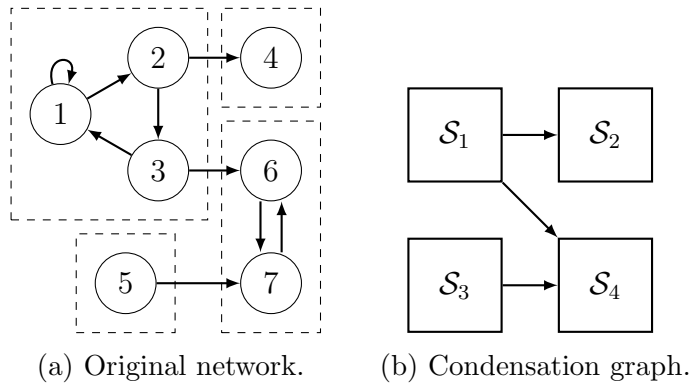
\begin{figure}[h]
		\centering
		\begin{subfigure}[t]{0.3\textwidth}
			\centering
			\begin{tikzpicture}[
node1/.style = {circle,minimum size=23,draw},
node2/.style = {circle,minimum size=23,draw,fill=white!75!black},
node3/.style = {circle,minimum size=23,draw,fill=white!50!black},
noderect/.style = {rectangle,minimum size=20,draw},
edge1/.style = {>=latex,thick},
edgedashed/.style = {>=latex,thick,dashed},
edge2/.style = {>=latex,thick,blue},
edge3/.style = {>=latex,thick,red}
]

\node[node1] at ({-3*sqrt(3)/12},0)(n5){5};

\node[node1] at (1.5,0)(n7){7};
\node[node1] at (1.5,1.5)(n6){6};
\node[node1] at (1.5,{2*1.5})(n4){4};

\node[node1] at ({-3*sqrt(3)/4},{1.5*1.5})(n1){1};
\node[node1] at (0,{2*1.5})(n2){2};
\node[node1] at (0,1.5)(n3){3};

\draw [->,edge1](n1) -- (n2);
\draw [->,edge1](n2) -- (n3);
\draw [->,edge1](n3) -- (n1);

\draw [->,edge1](n2) -- (n4);
\draw [->,edge1](n3) -- (n6);
\draw [->,edge1](n5) -- (n7);

\DoubleLine{n6}{n7}{<-,edge1}{}{->,edge1}{}

\node (m1box) [dashed, fit = (n1)(n2)(n3), inner sep=7pt, draw]{};

\node (m2box) [dashed, fit = (n5), inner sep=7pt, draw]{};

\node (m3box) [dashed, fit = (n4), inner sep=7pt, draw]{};

\node (m4box) [dashed, fit = (n6)(n7), inner sep=7pt, draw]{};

\draw [->,edge1] (n1) edge[loop above,looseness=5] (n1);

%\node[node1] at (0,0)(n1){1};
%\node[node1] at (2,0)(n2){2};
%\node[node2] at (1,{sqrt(2)})(n3){3};
%\node[node3] at (1,{2*sqrt(2)})(n4){4};

%\draw [->,edge1](n1) -- (n3);
%\draw [->,edge1](n2) -- (n3);
%\draw [->,edgedashed](n2) -- (n3);
%\node (minus23) at ($(n2)!0.4!(n3) + (-0.2,-0.2)$) {(-)};
%\draw [->,edge1](n3) -- (n4);

%\draw [->,edge1] (n4) edge[bend right,looseness=0.8] (n1);
%\draw [->,edge1] (n4) edge[bend left,looseness=0.8] (n2);

%\DoubleLine{n1}{n2}{<-,edge1}{}{->,edge1}{}
%\DoubleLine{n1}{n3}{<-,edge1}{}{->,edge1}{}
%\DoubleLine{n2}{n3}{<-,edge1}{}{->,edge1}{}

\end{tikzpicture} 
			\caption{Original network.}
			\label{fig:SCC_decomp_example_a}
		\end{subfigure}
		\begin{subfigure}[t]{0.3\textwidth}
			\centering
			\begin{tikzpicture}[
node1/.style = {circle,minimum size=23,draw},
node2/.style = {circle,minimum size=23,draw,fill=white!75!black},
node3/.style = {circle,minimum size=23,draw,fill=white!50!black},
noderect/.style = {rectangle,minimum size=20,draw},
edge1/.style = {>=latex,thick},
edge2/.style = {>=latex,thick,blue},
edge3/.style = {>=latex,thick,red},
block/.style = {draw, thick, rectangle, minimum height = 3em, minimum width = 3em},
dottedbox/.style = {draw=white!50!black, line width=1pt, dash pattern=on 3pt off 3pt, inner sep=4mm, rectangle, rounded corners}
]
\node[block] at (0,0)(S1){$ \mathcal{S}_{1} $};
\node[block] at (0,-2)(S3){$ \mathcal{S}_{3} $};
\node[block] at (2,0)(S2){$ \mathcal{S}_{2} $};
\node[block] at (2,-2)(S4){$ \mathcal{S}_{4} $};

\draw [->,edge1](S1) -- (S2);
%\draw [->,edge1](S2) -- (S3);
\draw [->,edge1](S1) -- (S4);
\draw [->,edge1](S3) -- (S4);

%\node (box) [dottedbox, fit = (S1) (S2)] {};

\end{tikzpicture} 
			\caption{Condensation graph.}
			\label{fig:balanced_types_condensation}
		\end{subfigure}	
		\caption{Decomposition of a network into its strongly connected components.}
		\label{fig:SCC_decomp_example}
	\end{figure}
	Consider the network in \fref{fig:SCC_decomp_example_a}, which has four different SCCs. In particular, $ \mathcal{S}_1 = \{1,2,3\} $, $ \mathcal{S}_2 = \{4\} $, $ \mathcal{S}_3 = \{5\} $ and $ \mathcal{S}_4 = \{6,7\} $. This induces the partition $ \{\{1,2,3\},\{4\},\{5\},\{6,7\}\} $ on the set of cells in the network. We form the condensation graph at 		\fref{fig:balanced_types_condensation} by representing each SCC by a block and connecting them appropriately. That is, we have $ 2 \edge 4 $, $ 3 \edge 6 $ and $ 5 \edge 7 $, which means that we need to connect $ \mathcal{S}_1 \edge \mathcal{S}_2 $, $ \mathcal{S}_1 \edge \mathcal{S}_4 $ and $ \mathcal{S}_3 \edge \mathcal{S}_4 $, respectively.\\
	Using the condensation graph, it is easy to see that the in-reachability sets of the SCCs are $ \mathcal{R}^{-}(\mathcal{S}_1) = \mathcal{S}_1 $, $ \mathcal{R}^{-}(\mathcal{S}_2) = \mathcal{S}_1\cup \mathcal{S}_2 $, $ \mathcal{R}^{-}(\mathcal{S}_3) = \mathcal{S}_3 $ and $ \mathcal{R}^{-}(\mathcal{S}_4) = \mathcal{S}_1\cup \mathcal{S}_3 \cup \mathcal{S}_4 $. This means that the network has two roots, $ \mathcal{S}_1 $ and $ \mathcal{S}_3 $. With two roots, we can partition the cells of the network into at most three RDCs. That is, those that depend on the root $ \mathcal{S}_1 $ but not $ \mathcal{S}_3 $ $ (\mathcal{S}_1 \cup \mathcal{S}_2) $, those that depend on $ \mathcal{S}_3 $ but not $ \mathcal{S}_1 $ $ (\mathcal{S}_3) $ and those that depend on both $ \mathcal{S}_1 $ and $ \mathcal{S}_3 $ $ (\mathcal{S}_4) $. Therefore the partition induced by the RDCs is $ \{\mathcal{S}_1 \cup\mathcal{S}_2 ,\mathcal{S}_3,\mathcal{S}_4\} =  \{\{1,2,3,4\},\{5\},\{6,7\}\} $, which is coarser that the partition of SCCs.
	\hfill $ \square $
\end{exmp}
\section{In-reachability based classification of synchrony partitions}\label{sec:part_classif}
Motivated by the influence of that different types of in-neighborhoods have on the dynamics of a network, we introduce an in-reachability based classification scheme for general partitions.
\subsection{Strong, rooted and weak partitions}
\label{subsec:SRW_partitions}
In this section we classify the colors of partitions according to their relationship to the structure of the network. For this purpose, we pay particular attention to the in-reachability sets, which fully determine the dynamical evolution of the cells, and the SCCs, which are the natural way of segmenting them.\\
Consider the network in \fref{fig:SCC_decomp_example}.
Note that $ \mathcal{S}_1 $ and $ \mathcal{S}_3 $ are roots, that is, $ \mathcal{R}^{-}(\mathcal{S}_1) = \mathcal{S}_1$ and $ \mathcal{R}^{-}(\mathcal{S}_3) = \mathcal{S}_3$. From \cref{thm:Rin_subsystem}, the evolution of each of those sets can be completely determined without regard to the rest of the network. That is, for any $ \mathcal{G} $-admissible function $ f $, we can restrict and evaluate it separately in the sets of cells $ \mathcal{S}_1 $, $ \mathcal{S}_3 $.\\
Consider a partition $ \mathcal{A} $ in this network such that there are cells in $ \mathcal{S}_1 $ and $ \mathcal{S}_3 $ that share the same color, that is, there are two cells $ c_1\in\mathcal{S}_1 $, $ c_3\in\mathcal{S}_3 $ such that $ \mathcal{A}(c_1) = \mathcal{A}(c_3) $.\\
Since the two SCCs evolve completely decoupled from one another, any disturbance on $ c_1 $ would not be felt by $ c_3 $ and vice-versa. Moreover, there is no cell that can simultaneously affect both $ c_1 $ and $ c_3 $ and act as a pacemaker to drive them to a common state. However, this lack of feedback between these cells does not mean that it is impossible for the synchrony pattern determined by $ \mathcal{A} $ to appear in a physical system. That is, for states sufficiently close to the polydiagonal $ \Delta_{\mathcal{A}}^{\xset} $ to be driven back to $ \Delta_{\mathcal{A}}^{\xset} $, or at least stay close to it.
This could be achieved if, for instance, both $ x_{c_1}(t) $ and $ x_{c_3}(t) $ converge to the same stable equilibrium point.\\
On the other hand, if $ x_{c_1}(t), x_{c_3}(t) $ converge to the same limit cycle, we would not expect such synchrony space to be stable, since there would be no mechanism that could counteract a possible phase offset. In particular, if $ \mathbf{x}_{\mathcal{S}_1}(t) $ is a solution for the subnetwork induced by $ \mathcal{S}_1 $, the time shifted version $ \mathbf{x}_{\mathcal{S}_1}(t-\delta) $ is also a solution. Therefore phase synchrony with $ \mathcal{S}_3 $ never happens unless we start with precise initial conditions.\\
Assume now that instead, there are two cells $ c_2\in\mathcal{S}_2 $, $ c_4\in\mathcal{S}_4 $ of the same color, that is $ \mathcal{A}(c_2) = \mathcal{A}(c_4)$. Their \mbox{in-reachability} sets are $ \mathcal{R}^{-}(\mathcal{S}_2) = \mathcal{S}_2 \cup \mathcal{S}_1 $ and $ \mathcal{R}^{-}(\mathcal{S}_4) = \mathcal{S}_4 \cup \mathcal{S}_1 \cup \mathcal{S}_3$ respectively. Now, although there is still no feedback between them, their \mbox{in-reachability} sets intersect in $ \mathcal{S}_1 $. Thus it could still be possible for $ c_2 $ and $ c_4 $ to maintain synchronism with non-trivial behavior if $ \mathcal{S}_1 $ is driving them to do so.\\
This shows that the structure of the network can make a crucial difference in the qualitative behavior of the invariant synchrony patterns, which motivates the following definitions.
\begin{defi}\label{def:color_connect_classif}
	A color $ A$ on a partition of a network $ \mathcal{G} $ is
	\begin{itemize}
		\item \textit{Strong} if all the cells of that color are in the same SCC. That is,
		\begin{eqnarray}
		c,d\in A \implies \mathcal{R}^{-}(c) = \mathcal{R}^{-}(d),
		\end{eqnarray}
		or equivalently,
		\begin{eqnarray}
		\bigcap_{c\in A} 	\mathcal{R}^{-}(c) 
		=
		\bigcup_{c\in A} 	\mathcal{R}^{-}(c).
		\end{eqnarray}
		\item \textit{Rooted} if it is not strong but there is some cell (root) in $ \mathcal{G} $ that has a directed path to all the cells of that color. That is,
		\begin{eqnarray}
		\emptyset \subset \bigcap_{c\in A} 	\mathcal{R}^{-}(c) 
		\subset 
		\bigcup_{c\in A} 	\mathcal{R}^{-}(c).
		\end{eqnarray}
		\item \textit{Weak} if it is neither strong nor rooted. That is,
		\begin{eqnarray}
		\bigcap_{c\in A} 	\mathcal{R}^{-}(c) = \emptyset .
		\end{eqnarray}
	\end{itemize}
	\hfill $ \square $
\end{defi}
Clearly, every color is of one, and only one, of these three types. The following properties are direct from the definition.
\begin{lemma}
	\label{lemma:color_strenght_comparison}
	Consider a strong color $ A_{s} $, a rooted color $ A_{r} $, and a weak color $ A_{w} $. Then the following is true
	\begin{itemize}
		\item If $ A \subseteq A_{s} $, then $ A $ is strong.
		\item If $ A \subseteq A_{r} $, then $ A $ is either rooted or strong.
		\item If $ A_{r} \subseteq A $, then $ A $ is either rooted or weak.
		\item If $ A_{w} \subseteq A $, then $ A $ is weak.
	\end{itemize}
	\hfill $ \square $
\end{lemma}
\cref{def:color_connect_classif} classifies a particular color of some partition on $ \mathcal{G} $ with respect to its connectivity structure. This classification scheme is independent of the underlying partition containing that color. Furthermore, we do not assume any particular structure on the underlying partitions, such as being balanced or being finer that the partition of cell types $ \type{G} $.\\
Using this classification scheme for individual colors, we classify a whole partition according to the following definition.
\begin{defi}
	\label{def:partition_connect_classif}
	A partition $ \mathcal{A} $ on a network $ \mathcal{G} $ is
	\begin{itemize}
		\item \textit{Strong} if all of its colors are strong.
		\item \textit{Rooted} if it is not strong but all of its colors are either rooted or strong. That is, it has at least one rooted color.
		\item \textit{Weak} if any of its colors is weak.
	\end{itemize}
	\hfill $ \square $
\end{defi}
Clearly, every partition is of one, and only one, of these three types.
Similarly to \cref{lemma:color_strenght_comparison}, the following properties are direct from the definitions.
\begin{lemma}\label{lemma:partition_comparison_lemma}
	Consider a strong partition $ \mathcal{A}_{s} $, a rooted partition $ \mathcal{A}_{r} $ and a weak partition $ \mathcal{A}_{w} $. Then the following is true
	\begin{itemize}
		\item If $ \mathcal{A} \leq \mathcal{A}_{s} $, then $ \mathcal{A} $ is strong.
		\item If $ \mathcal{A} \leq \mathcal{A}_{r} $, then $ \mathcal{A} $ is either rooted or strong.
		\item If $ \mathcal{A}_{r} \leq \mathcal{A} $, then $ \mathcal{A} $ is either rooted or weak.
		\item If $ \mathcal{A}_{w} \leq \mathcal{A} $, then $ \mathcal{A} $ is weak.
	\end{itemize}
	\hfill $ \square $
\end{lemma}
The following is straightforward.
\begin{corollary}
	\label{lemma:trivia_strong}
	If $ A $ is a singleton color, then is it strong. Furthermore, the trivial partition $ \bot $, which has only singleton colors, is always strong.
	\hfill $ \square $
\end{corollary}
We now relate the classification of partitions to the network connectivity according to the decomposition into SCCs and RDCs, as defined in \cref{subsec:scc_rdc}.
The two following results are direct from the definitions.
\begin{lemma}\label{lemma:strong_iff_finer_scc}
	A partition is strong if and only if it is finer than the partition of SCCs.
	\hfill $ \square $
\end{lemma}
We use the term \textit{non-weak} to denote partitions or colors that are not weak, that is, either rooted or strong.
\begin{lemma}\label{lemma:finer_rdc_implies_nw}
	A partition finer than the partition of RDCs is non-weak.
	\hfill $ \square $
\end{lemma}
In \cref{subsec:invariant_poly} we saw that for any particular subset of functions $ F\subseteq \mathcal{F}_{\mathcal{G}} $, the subset of partitions that are $ F $-invariant always forms a lattice $ L_{F} $. Furthermore, its minimal element is always the trivial partition $ \bot $, which is strong. Also, given any two partitions $ \mathcal{A}_{1},\mathcal{A}_{2} \in L_{F} $, their least upper bound is given by $ \mathcal{A}_1 \vee \mathcal{A}_2 $, where $ \vee $ denotes the partition join operation as defined in \cref{lemma:join_chain_def}. We now show how the join operation interacts with the proposed classification scheme.
\begin{lemma}\label{lemma:strong_join}
	For any pair of strong partitions \mbox{$ \mathcal{A}_{1},\mathcal{A}_{2}$} on a network $ \mathcal{G} $, their join \mbox{$ \mathcal{A} = \mathcal{A}_{1} \vee \mathcal{A}_{2} $} is strong.
	\hfill $ \square $
\end{lemma}
\begin{proof}
	Since $ \mathcal{A}_1, \mathcal{A}_2 $ are strong, from \cref{lemma:strong_iff_finer_scc}, they are finer than the partition of SCCs. Then $ \mathcal{A} = \mathcal{A}_{1} \vee \mathcal{A}_{2} $ is also finer than the partition of SCCs. From \cref{lemma:strong_iff_finer_scc} again, $ \mathcal{A} $ is strong. 
\end{proof}
This result, together with \cref{lemma:partition_comparison_lemma}, allows us to understand how the join operation affects the connectivity-based classification of general partitions. This is summarized in \tref{table:join_general}, where $ S $, $ R $ and $ W $ denote the partition classifications of strong, rooted and weak, respectively.
\begin{table}[h]
	\caption{Join table for general partitions.}
	\begin{center}
		\begin{tabular}{c | c c c}
			$ \vee $ & S & R & W  \\
			\cline{1-4}
			S & S & R/W & W \\
			R & R/W & R/W & W \\
			W & W & W & W
		\end{tabular}
	\end{center}
	\label{table:join_general}
\end{table}
So far we have not made any assumptions about the partitions. Moreover, we see in \tref{table:join_general} that there are entries in which the class is not completely defined. Specifically, there are cases where the result of the join could be either rooted or weak ($ R/W $).\\
Denote the subset of strong partitions in a lattice $ L_F $ by $ L_F^{S} $ and the subset of non-weak partitions by $
L_F^{NW} $. Then we have
\begin{eqnarray}
L_F^{S}
\subseteq 
L_F^{NW}
\subseteq
L_F.
\end{eqnarray}
From \cref{lemma:strong_join}, together with the fact that the trivial partition $ \bot $ is strong, we know that $
L_F^{S} $ always forms a sublattice of $ L_F $ with a top element $ \top_F^{S} $. On the other hand, $ L_F^{NW} $ might or might not be a lattice. This is illustrated in the following example.
\begin{exmp}
	\label{exmp:classification_example4}
	Consider the network in \fref{fig:class_net_example4} and its respective lattice of balanced partitions $ \Lambda $ in \fref{fig:class_lattice_example4}. Consider the full edges to have a weight of $ 1 $ and the dashed edges to have weights of $ -1 $.\\
	In the lattice schematics, the partitions are colored according to their type such that strong partitions are in white, rooted ones are light gray and weak ones are in dark gray.\\
	Note that $ \Lambda^{S} $, consisting of partitions in white, forms a sublattice of $ \Lambda $ with top partition $ \top^{S} = 12/34 $. On the other hand, $ \Lambda^{NW} $ does not form a lattice. In particular, if we join one of $ 12/45 $, $ 12/345 $ with one of $ 25/34 $, $ 125/34 $, we get $ 12345 $, which is a weak partition. 
	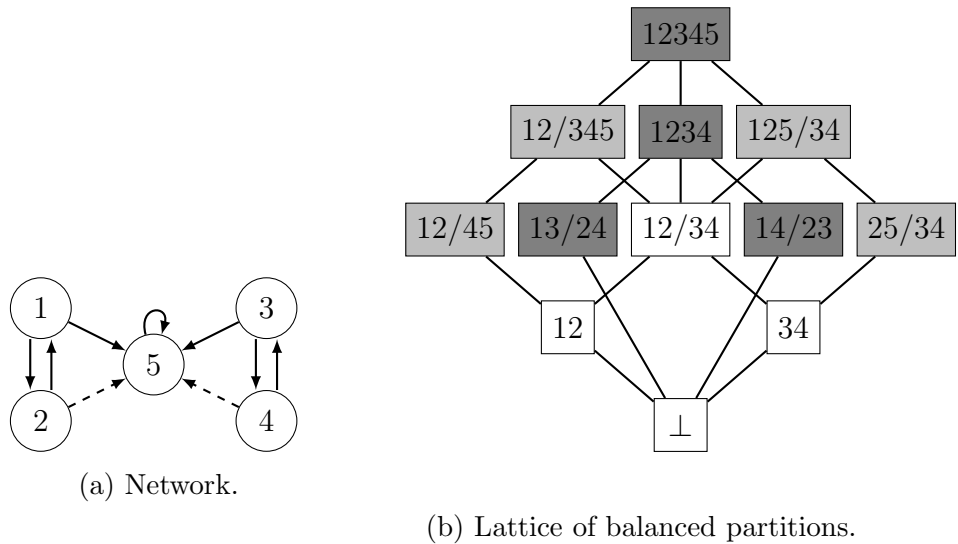
\begin{figure}[h]
		\centering
		\begin{subfigure}[t]{0.4\textwidth}
			\centering
			\begin{tikzpicture}[
node1/.style = {circle,minimum size=23,draw},
node2/.style = {circle,minimum size=23,draw,fill=white!75!black},
node3/.style = {circle,minimum size=23,draw,fill=white!50!black},
noderect/.style = {rectangle,minimum size=20,draw},
edge1/.style = {>=latex,thick},
edgedashed/.style = {>=latex,thick,dashed},
edge2/.style = {>=latex,thick,blue},
edge3/.style = {>=latex,thick,red}
]

\node[node1] at (-1.5,0.75)(n1){1};
\node[node1] at (-1.5,-0.75)(n2){2};
\node[node1] at (1.5,0.75)(n3){3};
\node[node1] at (1.5,-0.75)(n4){4};
\node[node1] at (0,0)(n5){5};

\draw [->,edge1](n1) -- (n5);
\draw [->,edgedashed](n2) -- (n5);
\draw [->,edge1](n3) -- (n5);
\draw [->,edgedashed](n4) -- (n5);

\DoubleLine{n1}{n2}{<-,edge1}{}{->,edge1}{}
\DoubleLine{n3}{n4}{<-,edge1}{}{->,edge1}{}

%\node (minus25) at ($(n2)!0.4!(n5) + (0.1,-0.2)$) {(-)};
%\node (minus45) at ($(n4)!0.4!(n5) + (-0.1,-0.2)$) {(-)};

\draw [->,edge1] (n5) edge[loop above,looseness=5] (n5);

\end{tikzpicture}
			\caption{Network.}
			\label{fig:class_net_example4}
		\end{subfigure}
		\begin{subfigure}[t]{0.4\textwidth}
			\centering
			\begin{tikzpicture}[
node1/.style = {circle,minimum size=23,draw},
node2/.style = {circle,minimum size=23,draw,fill=white!75!black},
node3/.style = {circle,minimum size=23,draw,fill=white!50!black},
pweak/.style = {rectangle,minimum size=20,draw,draw,fill=white!50!black},
proot/.style = {rectangle,minimum size=20,draw,draw,fill=white!75!black},
pstrong/.style = {rectangle,minimum size=20,draw},
noderect/.style = {rectangle,minimum size=20,draw},
edge1/.style = {>=latex,thick},
edgedash/.style = {>=latex,thick,dashed},
edge2/.style = {>=latex,thick,blue},
edge3/.style = {>=latex,thick,red}
]

\def\layer{1.3}
\def\del{0.75}

\node[pstrong] at (0,0) (bot){$ \bot $};

\node[pstrong] at ({-2*\del},{\layer} ) (12){$ 12 $};
\node[pstrong] at ({2*\del},{\layer} ) (34){$ 34 $};

\node[proot] at ({-4*\del},{2*\layer} ) (12_45){$ 12/45 $};
\node[pweak] at ({-2*\del},{2*\layer} ) (13_24){$ 13/24 $};
\node[pstrong] at (0,{2*\layer} ) (12_34){$ 12/34 $};
\node[pweak] at ({2*\del},{2*\layer} ) (14_23){$ 14/23 $};
\node[proot] at ({4*\del},{2*\layer} ) (25_34){$ 25/34 $};

\node[proot] at ({-2*\del},{3*\layer} ) (12_345){$ 12/345 $};
\node[pweak] at (0,{3*\layer} ) (1234){$ 1234 $};
\node[proot] at ({2*\del},{3*\layer} ) (125_34){$ 125/34 $};

\node[pweak] at (0,{4*\layer}) (12345){$ 12345 $};

\draw [-,edge1](bot) -- (12);
\draw [-,edge1](bot) -- (34);

\draw [-,edge1](bot) -- (13_24);
\draw [-,edge1](bot) -- (14_23);

\draw [-,edge1](12) -- (12_45);
\draw [-,edge1](12) -- (12_34);

\draw [-,edge1](34) -- (12_34);
\draw [-,edge1](34) -- (25_34);

\draw [-,edge1](12_45) -- (12_345);

\draw [-,edge1](13_24) -- (1234);

\draw [-,edge1](12_34) -- (12_345);
\draw [-,edge1](12_34) -- (1234);
\draw [-,edge1](12_34) -- (125_34);

\draw [-,edge1](14_23) -- (1234);

\draw [-,edge1](25_34) -- (125_34);

\draw [-,edge1](12345) -- (12_345);
\draw [-,edge1](12345) -- (1234);
\draw [-,edge1](12345) -- (125_34);

\end{tikzpicture} 
			\caption{Lattice of balanced partitions.}
			\label{fig:class_lattice_example4}
		\end{subfigure}
		\caption{A network and its lattice of balanced partitions.}
		\label{fig:classification_example4}
	\end{figure}
	\hfill $ \square $
\end{exmp}
By \cref{lemma:partition_comparison_lemma}, knowledge of the top partition $ \top_F $ of a lattice $ L_F $, with $ F\subseteq \mathcal{F}_{\mathcal{G}}$, can give us important information about the whole lattice.
\begin{corollary}
	If the top partition $\top_F $ of a lattice $  L_F $, with $ F\subseteq \mathcal{F}_{\mathcal{G}}$, is non-weak, then all of its partitions are non-weak. Moreover, if $ \top_F $ is strong, then all partitions are also strong.
	\hfill $ \square $
\end{corollary}
We now show how the top strong partition $ \top_{F}^{S} $ is given in terms of the $ cir_F $ function.
\begin{corollary}\label{lemma:maximal_strong}
	Consider a network $ \mathcal{G} $ with cell type partition $ \type{} $. Represent its SCCs according to a partition $ \mathcal{A} $. Then $ \top_F^{S} = cir_{F}(\type{} \wedge \mathcal{A})$.
	\hfill $ \square $
\end{corollary}
Note that $ L_F^{NW} $ is not necessarily a lattice and there might exist multiple locally maximal non-weak partitions, as in \cref{exmp:classification_example4}. In the following section we see that under some relatively tame assumptions, the resulting join table becomes much cleaner and we can guarantee that $ L_F^{NW} $ is a lattice with some top partition $ \top_F^{NW} $.
\subsection{Neighborhood color matching}
\label{subsec:neigh_color_matching}
In this section we present a sequence of progressively weaker assumptions about a partition on a network. We show that the weakest of them is enough to fix the remaining uncertain entries of \tref{table:join_general} into \tref{table:join_r_matched}.\\
We use the notational convention $ \mathcal{A}(\mathbf{s}) := \bigcup_{c\in\mathbf{s}}\mathcal{A}(c) $. That is, $ \mathcal{A}(\mathbf{s}) $ denotes a subset of colors that are present in the set of cells $ \mathbf{s} \subseteq \mathcal{C} $, according to the coloring assigned by $ \mathcal{A} $.
\begin{defi} \label{defi:U_matched}
	Consider a function $ \mathcal{U} $ that assigns to each cell a subset of cells, that is, $ \mathcal{U} \colon \mathcal{C} \to 2^{\mathcal{C}}$. Then a partition $ \mathcal{A} $ on $ \mathcal{C} $ is $ \mathcal{U} $-\textit{matched} if when we apply $ \mathcal{U} $ to cells of the same color, the	resulting subsets share the exact same subset of colors. That is,
	\begin{eqnarray}
	\mathcal{A}(c) = \mathcal{A}(d) \implies \mathcal{A}(\mathcal{U}(c)) = \mathcal{A}(\mathcal{U}(d)).
	\end{eqnarray}
	\hfill $ \square $
\end{defi}
\begin{corollary}
	\label{lemma:trivial_U_matched}
	The trivial partition $ \bot $ is $ \mathcal{U} $-matched for every function $ \mathcal{U} $.
	\hfill $ \square $
\end{corollary}
In this work, we are interested in the situation where the function $ \mathcal{U} $ in \cref {defi:U_matched} denotes a neighborhood as described in \cref{subsec:neigh_reach}, such as $ \mathcal{N}^{-} $, $ \mathcal{V}^{-} $, $ \mathcal{V}^{-}_{k} $ or $ \mathcal{R}^{-} $.
\begin{corollary}\label{lemma:n_matched_implies_v_matched}
	If a partition $ \mathcal{A} $ is $ \mathcal{N}^{-} $-matched then it is $ \mathcal{V}^{-} $-matched.
	\hfill $ \square $
\end{corollary}
\begin{proof}
	If $ \mathcal{A}(c) = \mathcal{A}(d)$, by assumption we have $ \mathcal{A}(\mathcal{N}^{-}(c)) = \mathcal{A}(\mathcal{N}^{-}(d)) $. Then we have
	\begin{eqnarray*}
	&
	\mathcal{A}(c)\cup \mathcal{A}(\mathcal{N}^{-}(c)) &= \mathcal{A}(d)\cup \mathcal{A}(\mathcal{N}^{-}(d))
	\\
	\implies \quad
	&
	\mathcal{A}(c\cup\mathcal{N}^{-}(c)) &=
	\mathcal{A}(d\cup\mathcal{N}^{-}(d))
	\\
	\implies \quad
	&
	\mathcal{A}(\mathcal{V}^{-}(c)) &=
	\mathcal{A}(\mathcal{V}^{-}(d)).
	\end{eqnarray*}
\end{proof}
\begin{lemma}\label{lemma:v_matched_implies_v_k_matched}
	If a partition $ \mathcal{A} $ is $ \mathcal{V}^{-} $-matched then it is $ \mathcal{V}_{k}^{-} $-matched for every $ k\geq1 $.
	\hfill $ \square $
\end{lemma}
\begin{proof}
	The proof is by induction. The base case $ k=1 $ is trivial since $ \mathcal{V}_{1}^{-} = \mathcal{V}^{-} $. We assume that the statement applies to a given $ k $. That is, $ \mathcal{A} $ is both $ \mathcal{V}^{-} $-matched and $ \mathcal{V}_{k}^{-} $-matched. Then
	\begin{eqnarray*}
	\mathcal{A}(\mathcal{V}_{k+1}^{-}(c)) = 
	\mathcal{A}\left(\bigcup_{c^{\star}\in\mathcal{V}_{k}^{-}(c)} \mathcal{V}^{-}(c^{\star})\right) 
	=
	\bigcup_{c^{\star}\in\mathcal{V}_{k}^{-}(c)} \mathcal{A}(\mathcal{V}^{-}(c^{\star})),
	\end{eqnarray*}
	where the first equality comes from \eref{eq:in_cum_neigh_recurs} and the second from how we defined the notation of applying $ \mathcal{A} $ to a set. Since $ \mathcal{A} $ is $ \mathcal{V}^{-} $-matched, $ \mathcal{A}(\mathcal{V}^{-}(c^{\star})) $ depends only on the color of the cell $ c^{\star} $. Moreover, since $ \mathcal{A} $ is also $ \mathcal{V}_{k}^{-} $-matched $ \mathcal{A}(c) = \mathcal{A}(d) $ implies $ \mathcal{A}(\mathcal{V}_{k}^{-}(c)) = \mathcal{A}(\mathcal{V}_{k}^{-}(d)) $, which means that $ c^{\star}\in\mathcal{V}_{k}^{-}(c) $ and $ d^{\star}\in\mathcal{V}_{k}^{-}(d) $ index the exact same set of colors. Therefore $ \mathcal{A}(\mathcal{V}_{k+1}^{-}(c)) = \mathcal{A}(\mathcal{V}_{k+1}^{-}(d)) $.\\
\end{proof}
\begin{corollary}
	If a partition defined on a finite set of cells is $ \mathcal{V}^{-} $-matched, then it is also $ \mathcal{R}^{-} $-matched.
	\label{lemma:v_matched_implies_r_matched}
	\hfill $ \square $
\end{corollary}
\begin{proof}
	This is direct from \cref{lemma:v_matched_implies_v_k_matched} and 	\cref{cor:cumulative_neigh_k_inside_reachable}.
\end{proof}
We have seen in \cref{lemma:n_matched_implies_v_matched} that a $ \mathcal{N}^{-} $-matched partition is also $ \mathcal{V}^{-} $-matched. The next very trivial example shows that the converse is not necessarily true.
\begin{exmp}\label{exmp:not_n_but_v_matched_example}
	Consider the network in \fref{fig:not_n_but_v_matched}, which is colored with only one color (white). 
	Note that $ \mathcal{N}^{-}(1) = \{\}$ is empty and $ \mathcal{N}^{-}(2) = \{1\}$ contains the white color. Therefore the partition is not $ \mathcal{N}^{-} $-matched. On the other hand, $ \mathcal{V}^{-}(1) = \{1\}$ and $ \mathcal{V}^{-}(2) = \{1,2\}$, so the partition is $ \mathcal{V}^{-} $-matched.
	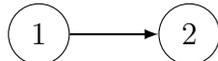
\begin{figure}[h]
		\centering
		\begin{tikzpicture}[
node1/.style = {circle,minimum size=23,draw},
node2/.style = {circle,minimum size=23,draw,fill=white!75!black},
node3/.style = {circle,minimum size=23,draw,fill=white!50!black},
noderect/.style = {rectangle,minimum size=20,draw},
edge1/.style = {>=latex,thick},
edgedashed/.style = {>=latex,thick,dashed},
edge2/.style = {>=latex,thick,blue},
edge3/.style = {>=latex,thick,red}
]

\node[node1] at (0,0)(n1){1};
\node[node1] at (2.0,0)(n2){2};

\draw [->,edge1](n1) -- (n2);

%\node (minus123) at ($(n3)!0.4!(n4) + (0.2,-0.2)$) {(-)};
%\DoubleLine{n2}{n3}{<-,edge1}{}{->,edge1}{}

%\draw [->,edge1] (n1) edge[loop left,looseness=5] (n1);
%\draw [->,edge1] (n3) edge[loop above,looseness=5] (n3);

%\node (minus123) at ($(n1)!0.4!(n2) + (0,0.3)$) {(-)};

%\DoubleLine{n1}{n2}{<-,edge1}{}{->,edge1}{}
%\DoubleLine{n1}{n3}{<-,edge1}{}{->,edge1}{}
%\DoubleLine{n2}{n3}{<-,edge1}{}{->,edge1}{}

\end{tikzpicture} 
		\caption{Partition that is not $ \mathcal{N}^{-} $-matched but is $ \mathcal{V}^{-} $-matched.}
		\label{fig:not_n_but_v_matched}
	\end{figure}
	\hfill $ \square $
\end{exmp}
We have also seen in \cref{lemma:v_matched_implies_r_matched} that a $ \mathcal{V}^{-} $-matched partition is also $ \mathcal{R}^{-} $-matched. In the next example we disprove the converse statement.
\begin{exmp}\label{exmp:not_v_but_r_matched_example}
	Consider the network in \fref{fig:not_v_but_r_matched}.
	Note that $ \mathcal{V}^{-}(1) = \{1,4\}$ contains only white colors and $ \mathcal{V}^{-}(4) = \{2,3,4\}$ contains white and gray colors. Therefore the partition is not $ \mathcal{V}^{-} $-matched. On the other hand, $ \mathcal{R}^{-}(1) = \mathcal{R}^{-}(4)$ and $ \mathcal{R}^{-}(2) = \mathcal{R}^{-}(3)$, so the partition is $ \mathcal{R}^{-} $-matched.
	\begin{figure}[h]
		\centering
		\begin{tikzpicture}[
node1/.style = {circle,minimum size=23,draw},
node2/.style = {circle,minimum size=23,draw,fill=white!75!black},
node3/.style = {circle,minimum size=23,draw,fill=white!50!black},
noderect/.style = {rectangle,minimum size=20,draw},
edge1/.style = {>=latex,thick},
edgedashed/.style = {>=latex,thick,dashed},
edge2/.style = {>=latex,thick,blue},
edge3/.style = {>=latex,thick,red}
]

\node[node1] at (-1,0)(n1){1};
\node[node2] at (0,1.0)(n2){2};
\node[node2] at (0,-1.0)(n3){3};
\node[node1] at (1.0,0)(n4){4};

\draw [->,edge1](n1) -- (n2);
\draw [->,edge1](n1) -- (n3);
\draw [->,edge1](n2) -- (n4);
\draw [->,edgedashed](n3) -- (n4);
\draw [->,edge1] (n4) edge[loop right,looseness=5] (n4);
\draw [->,edge1](n4) -- (n1);

%\node (minus123) at ($(n3)!0.4!(n4) + (0.2,-0.2)$) {(-)};
%\DoubleLine{n2}{n3}{<-,edge1}{}{->,edge1}{}

%\draw [->,edge1] (n1) edge[loop left,looseness=5] (n1);
%\draw [->,edge1] (n3) edge[loop above,looseness=5] (n3);

%\node (minus123) at ($(n1)!0.4!(n2) + (0,0.3)$) {(-)};

%\DoubleLine{n1}{n2}{<-,edge1}{}{->,edge1}{}
%\DoubleLine{n1}{n3}{<-,edge1}{}{->,edge1}{}
%\DoubleLine{n2}{n3}{<-,edge1}{}{->,edge1}{}

\end{tikzpicture} 
		\caption{Partition that is not $ \mathcal{V}^{-} $-matched but is $ \mathcal{R}^{-} $-matched.}
		\label{fig:not_v_but_r_matched}
	\end{figure}
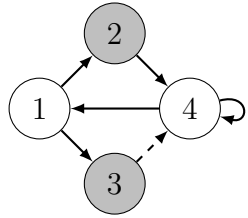
	\hfill $ \square $
\end{exmp}
In summary, we have shown that the sequence: $ \mathcal{N}^{-} $-matched , $ \mathcal{V}^{-} $-matched and $ \mathcal{R}^{-} $-matched lists progressively weaker assumptions. In \cref{exmp:not_v_but_r_matched_example} the in-reachability sets are, in fact, all the same. The following result should be obvious from the definitions.
\begin{corollary}\label{lemma:SCCs_are_r_matched}
	In a network that is a SCC, every partition is $ \mathcal{R}^{-} $-matched.
	\hfill $ \square $
\end{corollary}
More generally,
\begin{corollary}\label{lemma:strong_are_r_matched}
	Every strong partition is $ \mathcal{R}^{-} $-matched.
	\hfill $ \square $
\end{corollary}
We now show that the tamest assumption we described ($ \mathcal{R}^{-} $-matched) is enough to allow the following results. 
\begin{lemma}\label{lemma:nw_r_matched_finer_than_rdc}
	If a non-weak partition is $ \mathcal{R}^{-} $-matched, then it is finer than the partition of RDCs.
	\hfill $ \square $
\end{lemma}
\begin{proof}
	Consider some network with $ n $ roots $ \mathcal{S}_1,\ldots,\mathcal{S}_n $ and a partition $ \mathcal{A} $ that is non-weak and $ \mathcal{R}^{-} $-matched. Since they are roots, we have $\mathcal{S}_i = \mathcal{R}^{-}(\mathcal{S}_i) $ and $ \mathcal{S}_i \cap \mathcal{S}_j = \emptyset $ for all $ i\neq j $, which implies $ \mathcal{R}^{-}(\mathcal{S}_i) \cap \mathcal{R}^{-}(\mathcal{S}_j) = \emptyset $ for all $ i\neq j $. 
	Since $ \mathcal{A} $ is non-weak, $ \mathcal{R}^{-}(\mathcal{S}_i) \cap \mathcal{R}^{-}(\mathcal{S}_j) = \emptyset $ implies $ \mathcal{A}(\mathcal{S}_i) \cap \mathcal{A}(\mathcal{S}_j) = \emptyset $, for all $ i\neq j $. That is, each root in the network contains a set of colors distinct from every other root.\\
	Consider a cell $ c $ that is of a color that is present in one of the roots. That is, $ \mathcal{A}(c) =k  $ for some $ k\in\mathcal{A}(\mathcal{S}_i) $. Since $ \mathcal{A} $ is $ \mathcal{R}^{-} $-matched, $ \mathcal{A}(\mathcal{R}^{-}(c) ) = \mathcal{A}(\mathcal{S}_i)$. 
	Since each root has a set of colors distinct from every other root, $ \mathcal{A}(\mathcal{R}^{-}(c) ) \cap \mathcal{A}(\mathcal{S}_j) = \emptyset$ for all $ j\neq i $. 
	This implies $ \mathcal{R}^{-}(c)  \cap \mathcal{S}_j = \emptyset$ for all $ j\neq i $. That is, if some cell in the network shares its color with a root, then that cell cannot depend (in the $ \mathcal{R}^{-} $ sense) on any other roots. Since it is impossible not to depend on any roots at all, this implies $  \mathcal{R}^{-}(c) \supseteq \mathcal{S}_i $. That is, if a cell shares its color with a root, then it depends (in the $ \mathcal{R}^{-} $ sense) on that root (and no other roots).\\
	Finally, consider $ c,d $ such that $ \mathcal{A}(c) = \mathcal{A}(d) $. Then since $ \mathcal{A} $ is $ \mathcal{R}^{-} $-matched, $ \mathcal{A}(\mathcal{R}^{-}(c)) = \mathcal{A}(\mathcal{R}^{-}(d))$.
	Since $ \mathcal{R}^{-}(c) $ and $ \mathcal{R}^{-}(d) $ share the exact same set of colors, they also share the same subset of colors that are present in roots. 
	From what we have shown before, depending on a color shared by a root implies depending on the root itself.
	Therefore cells of the same color depend on exactly the same roots, so $ \mathcal{A} $ is finer than the partition of RDCs.	
\end{proof}
We now show how the top non-weak partition $ \top_F^{NW} $ is given in terms of the $ cir_F $ function for the case where all rooted partitions are $ \mathcal{R}^{-} $-matched. In \cref{exmp:R_matched_lattice_example} the partitions of interest satisfy this assumption. In the following section we understand why that is the case.
\begin{corollary}\label{lemma:maximal_ns_rooted}
	Consider a network $ \mathcal{G} $ with cell type partition $ \type{} $. Represent its RDCs according to a partition $ \mathcal{B} $. Assume all its rooted partitions are $ \mathcal{R}^{-} $-matched. Then $ \top_F^{NW} = cir_F(\type{} \wedge \mathcal{B})$. 
	\hfill $ \square $
\end{corollary}
We are now ready to prove the following result.
\begin{lemma}\label{lemma:non_weak_r_matched_join}
	For any pair of non-weak $ \mathcal{R}^{-} $-matched partitions $ \mathcal{A}_{nw_1},\mathcal{A}_{nw_2} $, their join $ \mathcal{A}_{nw} = \mathcal{A}_{nw_1} \vee \mathcal{A}_{nw_2} $ is also non-weak.
	\hfill $ \square $
\end{lemma}
\begin{proof}
	From \cref{lemma:nw_r_matched_finer_than_rdc}, $ \mathcal{A}_{nw_1},\mathcal{A}_{nw_2} $ are both finer than the partition of RDCs. Therefore their join $ \mathcal{A}_{nw} $ is also going to be finer. From \cref{lemma:finer_rdc_implies_nw}, it is also non-weak.
\end{proof}
The following result is straightforward from the general case illustrated in \tref{table:join_general}, together with \cref{lemma:non_weak_r_matched_join}.
\begin{corollary}\label{lemma:S_R_nspurious_join}
	Consider partitions $ \mathcal{A}_{s},\mathcal{A}_{r_1},\mathcal{A}_{r_2} $ such that $ \mathcal{A}_{s} $ is strong and $ \mathcal{A}_{r_1},\mathcal{A}_{r_2} $ are rooted and $ \mathcal{R}^{-} $-matched. Then $ \mathcal{A}_{s} \vee \mathcal{A}_{r_1} $ and $ \mathcal{A}_{r_1} \vee \mathcal{A}_{r_2} $ are rooted.
	\hfill $ \square $
\end{corollary}
This means that for the case where rooted partitions are $ \mathcal{R}^{-} $-matched, \tref{table:join_general} simplifies into \tref{table:join_r_matched}. Furthermore, under such conditions $ L_F^{NW} $ is a sublattice of $ L_F $. This is illustrated in the following example.
\begin{table}[h]
	\caption{Join table when rooted partitions are $ \mathcal{R}^{-} $-matched.}
	\begin{center}
		\begin{tabular}{c | c c c}
			$ \vee $ & S & R & W  \\
			\cline{1-4}
			S & S & R & W \\
			R & R & R & W \\
			W & W & W & W
		\end{tabular}
	\end{center}
	\label{table:join_r_matched}
\end{table}
\begin{exmp}
	\label{exmp:R_matched_lattice_example}
	Consider the network in \fref{fig:class_net_example1} and its respective lattice of balanced partitions $ \Lambda $ in \fref{fig:class_lattice_example1}. Note that $ \Lambda^{S} $ and $ \Lambda^{NW} $ are both lattices with top partitions $ \top^{S} = \bot $ and $ \top^{NW} = 13/24 $, respectively. Every balanced partition in this network is $ \mathcal{R}^{-} $-matched, therefore \tref{table:join_r_matched} applies. In the following section we will see that this fact is immediate from the network not allowing edge cancelings.
	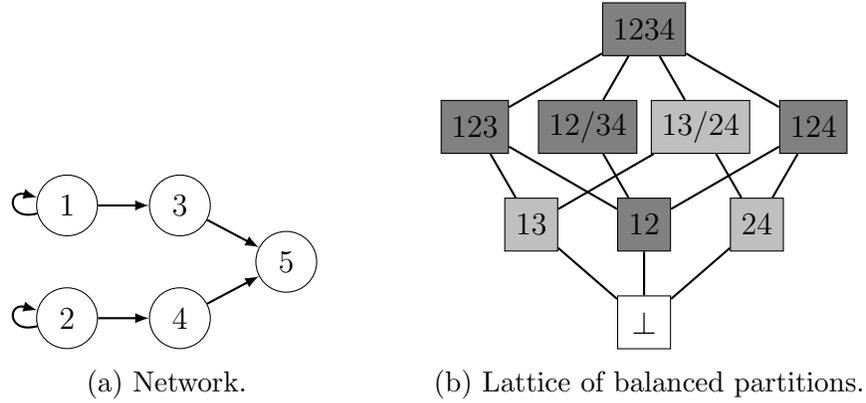
\begin{figure}[h]
		\centering
		\begin{subfigure}[t]{0.4\textwidth}
			\centering
			\begin{tikzpicture}[
node1/.style = {circle,minimum size=23,draw},
node2/.style = {circle,minimum size=23,draw,fill=white!75!black},
node3/.style = {circle,minimum size=23,draw,fill=white!50!black},
noderect/.style = {rectangle,minimum size=20,draw},
edge1/.style = {>=latex,thick},
edge2/.style = {>=latex,thick,blue},
edge3/.style = {>=latex,thick,red}
]
\node[node1] at (0,1.5)(n1){1};
\node[node1] at (0,0)(n2){2};
\node[node1] at (1.5,1.5)(n3){3};
\node[node1] at (1.5,0)(n4){4};
\node[node1] at ({1.5+sqrt(2)},0.75)(n5){5};

\draw [->,edge1] (n1) edge[loop left,looseness=5] (n1);
\draw [->,edge1] (n2) edge[loop left,looseness=5] (n2);

\draw [->,edge1](n1) -- (n3);
\draw [->,edge1](n2) -- (n4);
\draw [->,edge1](n3) -- (n5);
\draw [->,edge1](n4) -- (n5);

%\node (minus123) at ($(n1)!0.4!(n2) + (0,0.3)$) {(-)};

%\DoubleLine{n1}{n2}{<-,edge1}{}{->,edge1}{}
%\DoubleLine{n1}{n3}{<-,edge1}{}{->,edge1}{}
%\DoubleLine{n2}{n3}{<-,edge1}{}{->,edge1}{}

\end{tikzpicture}
			\caption{Network.}
			\label{fig:class_net_example1}
		\end{subfigure}
		\begin{subfigure}[t]{0.4\textwidth}
			\centering
			\begin{tikzpicture}[
node1/.style = {circle,minimum size=23,draw},
node2/.style = {circle,minimum size=23,draw,fill=white!75!black},
node3/.style = {circle,minimum size=23,draw,fill=white!50!black},
pweak/.style = {rectangle,minimum size=20,draw,draw,fill=white!50!black},
proot/.style = {rectangle,minimum size=20,draw,draw,fill=white!75!black},
pstrong/.style = {rectangle,minimum size=20,draw},
noderect/.style = {rectangle,minimum size=20,draw},
edge1/.style = {>=latex,thick},
edgedash/.style = {>=latex,thick,dashed},
edge2/.style = {>=latex,thick,blue},
edge3/.style = {>=latex,thick,red}
]

\def\layer{1.3}
\def\del{0.75}

\node[pstrong] at (0,0) (bot){$ \bot $};

\node[proot] at ({-2*\del},{\layer} ) (13){$ 13 $};
\node[pweak] at (0,{\layer} ) (12){$ 12 $};
\node[proot] at ({2*\del},{\layer} ) (24){$ 24 $};

\node[pweak] at ({-3*\del},{2*\layer} ) (123){$ 123 $};
\node[pweak] at ({-\del},{2*\layer} ) (12_34){$ 12/34 $};
\node[proot] at ({\del},{2*\layer} ) (13_24){$ 13/24 $};
\node[pweak] at ({3*\del},{2*\layer} ) (124){$ 124 $};

\node[pweak] at (0,{3*\layer}) (1234){$ 1234 $};

\draw [-,edge1](bot) -- (13);
\draw [-,edge1](bot) -- (12);
\draw [-,edge1](bot) -- (24);

\draw [-,edge1](13) -- (123);;
\draw [-,edge1](13) -- (13_24);

\draw [-,edge1](12) -- (123);
\draw [-,edge1](12) -- (12_34);
\draw [-,edge1](12) -- (124);

\draw [-,edge1](24) -- (13_24);
\draw [-,edge1](24) -- (124);

\draw [-,edge1](1234) -- (123);
\draw [-,edge1](1234) -- (12_34);
\draw [-,edge1](1234) -- (13_24);
\draw [-,edge1](1234) -- (124);

\end{tikzpicture} 
			\caption{Lattice of balanced partitions.}
			\label{fig:class_lattice_example1}
		\end{subfigure}
		\caption{A network and its lattice of balanced partitions.}
		\label{fig:classification_example1}
	\end{figure}
	\hfill $ \square $
\end{exmp}
\subsection{Neighborhood color invariance}
\label{subsec:neigh_color_invariance}
We now introduce a property that is stronger than \cref{defi:U_matched}, that applies only to balanced partitions, since it is related to the respective quotient network.
\begin{defi} \label{defi:U_invariant}
	Consider a balanced partition $ \bp \in \Lambda_{\mathcal{G}}$ on a network $ \mathcal{G} $ and its respective quotient network $ \mathcal{Q} = \mathcal{G}/\bp $. Take a particular type of neighborhood $ \mathcal{U}\in \{\mathcal{N}^{-},\mathcal{V}^{-},\mathcal{V}_{k}^{-},\mathcal{R}^{-}\} $ such that $ \mathcal{U}_{\mathcal{G}}$ and $\mathcal{U}_{\mathcal{Q}} $ are the corresponding functions on $ \mathcal{G} $ and $ \mathcal{Q} $, respectively.
	Then $ \bp $ is $ \mathcal{U} $-\textit{invariant} if
	\begin{eqnarray}
	d\in \mathcal{U}_{\mathcal{G}}(c)
	\implies
	\bp (d) \in \mathcal{U}_{\mathcal{Q}}(\bp (c)),
	\end{eqnarray}
	or equivalently, 
	\begin{eqnarray}
	\bp(\mathcal{U}_{\mathcal{G}}(c))
	\subseteq
	\mathcal{U}_{\mathcal{Q}}(\bp (c))
	\label{eq:U_invariant_subset}
	\end{eqnarray}
	for all $ c \in \mathcal{C}_{\mathcal{G}} $.
	\hfill $ \square $
\end{defi}
The converse inclusion in \eref{eq:U_invariant_subset} is always satisfied.
\begin{lemma}\label{lemma:U_invariant_converse}
	Consider a balanced partition $ \bp \in \Lambda_{\mathcal{G}}$ on a network $ \mathcal{G} $ and its respective quotient network $ \mathcal{Q} = \mathcal{G}/\bp $. 	
	Then for each color $ A\in\bp $, which is mapped into the cell $ k_A\in \mathcal{C}_{\mathcal{Q}} $, we have
	\begin{eqnarray}
	\label{lemma:U_invariant_converse_explicit_eq}
	k_A \in \mathcal{U}_{\mathcal{Q}}(\bp (c))
	\implies
	A\cap \mathcal{U}_{\mathcal{G}}(c)
	\neq \emptyset,
	\end{eqnarray}
	or equivalently,
	\begin{eqnarray}
	\mathcal{U}_{\mathcal{Q}}(\bp (c))
	\subseteq
	\bp(\mathcal{U}_{\mathcal{G}}(c))
	\end{eqnarray}
	for all $ c \in \mathcal{C}_{\mathcal{G}} $.
	\hfill $ \square $	
\end{lemma}
\begin{proof}
	Firstly, we define $ B\in\bp $ to be the color of $ c $, mapping into the cell $ k_B\in\mathcal{C}_{\mathcal{Q}} $.\\
	Assume $ k_A \in \mathcal{N}_{\mathcal{Q}}^{-}(k_B) $ . Then from the definition of $ \mathcal{N}^{-} $, we have a non-zero entry $ q_{k_B k_A} \neq 0_{ij}$, with $ i =\type{Q}(k_B) =\type{G}(B) $ and $ j =\type{Q}(k_A) = \type{G}(A)$ in the in-adjacency matrix $ Q $ associated with the quotient network $ \mathcal{Q} $. Then from the definition of quotient network, $ \sum_{d\in A \cap  \mathcal{N}_{\mathcal{G}}^{-}(c)} w_{cd} = q_{k_B k_A} $. Since $  q_{k_B k_A}\neq 0_{ij} $, this means that $ A \cap  \mathcal{N}_{\mathcal{G}}^{-}(c) $ is non-empty. That is, the statement is true for the case $ \mathcal{U} = \mathcal{N}^{-} $.\\
	We now prove the case $ \mathcal{U} = \mathcal{V}^{-} $.
	Assume $ k_A \in \mathcal{V}_{\mathcal{Q}}^{-}(\bp (c)) $. Then $ k_A \in \{k_B \cup \mathcal{N}_{\mathcal{Q}}^{-}(k_B)\} $. Consider the case $ k_A = k_B $. Then $ c \in A \cap  \mathcal{V}_{\mathcal{G}}^{-}(c) $, which makes the set non-empty. Consider now the case $ k_A \in \mathcal{N}_{\mathcal{Q}}^{-}(k_B) $. Then since the statement is true for $ \mathcal{U} = \mathcal{N}^{-} $, $ A \cap \mathcal{N}_{\mathcal{G}}^{-}(c) $ is non-empty. Therefore $ A \cap \mathcal{V}_{\mathcal{G}}^{-}(c)  \supseteq A \cap \mathcal{N}_{\mathcal{G}}^{-}(c) $ is also non-empty, which concludes the proof for $ \mathcal{U} = \mathcal{V}^{-} $.\\
	We now prove the case $ \mathcal{U} = \mathcal{V}_{k}^{-} $ for every $ k\geq1 $. The proof is by induction. The base case $ k=1 $ is trivial since $ \mathcal{V}_{1}^{-} = \mathcal{V}^{-} $. Assume it to be true for a given $ k $.
	Consider $ k_A \in \mathcal{V}_{k+1 \,\mathcal{Q}}^{-}(k_B) $. Then $ k_A \in \bigcup_{k_C\in\mathcal{V}_{k\,\mathcal{Q}}^{-}(k_B)} \mathcal{V}^{-}_{\mathcal{Q}}(k_C) $. That is, $ k_A\in \mathcal{V}^{-}_{\mathcal{Q}}(k_C) $ for at least one particular $ k_C\in\mathcal{V}_{k\,\mathcal{Q}}^{-}(k_B) $. Then since the case $ \mathcal{U} = \mathcal{V}_{k}^{-} $ is true by assumption, we have $ C\cap \mathcal{V}_{k\,\mathcal{G}}^{-}(c) \neq \emptyset $, where $ C\in\bp $ is the color that maps into cell $ k_C $. We choose a particular cell $ d \in C\cap \mathcal{V}_{k\,\mathcal{G}}^{-}(c)$. Then $ \bp(d) = k_C $. Furthermore, since the case $ \mathcal{U} = \mathcal{V}^{-} $ is true, $ k_A\in \mathcal{V}^{-}_{\mathcal{Q}}(k_C) $ implies $ A\cap \mathcal{V}^{-}_{\mathcal{G}}(d)
	\neq \emptyset $. Finally, note that $ A \cap \mathcal{V}_{k+1 \,\mathcal{G}}^{-}(c) \supseteq A\cap \mathcal{V}^{-}_{\mathcal{G}}(d)$ since $ d \in \mathcal{V}_{k\,\mathcal{G}}^{-}(c)$, which means that $ A \cap \mathcal{V}_{k+1 \,\mathcal{G}}^{-}(c) $ is also non-empty. This concludes the induction step.\\
	Finally, the case $ \mathcal{U} = \mathcal{R}^{-}$ is immediate from \cref{cor:cumulative_neigh_k_inside_reachable}.
\end{proof}
The following is immediate from \cref{defi:U_invariant} and \cref{lemma:U_invariant_converse}.
\begin{corollary}\label{lemma:U_invariant_equiv}
	Consider a balanced partition $ \bp \in \Lambda_{\mathcal{G}}$ on a network $ \mathcal{G} $ and its respective quotient network $ \mathcal{Q} = \mathcal{G}/\bp $.	
	Then $ \bp $ is $ \mathcal{U} $-invariant if and only if
	\begin{eqnarray}
	\bp(\mathcal{U}_{\mathcal{G}}(c))
	=
	\mathcal{U}_{\mathcal{Q}}(\bp (c))
	\end{eqnarray}
	for all $ c \in \mathcal{C}_{\mathcal{G}} $.
	\hfill $ \square $	
\end{corollary}
We now show that \cref{defi:U_invariant} is a stronger property than the one in \cref{defi:U_matched}.
\begin{lemma}\label{lemma:u_invariant_implies_u_matched}
	If a balanced partition $ \bp $ is $ \mathcal{U} $-invariant,
	then it is $ \mathcal{U} $-matched.
	\hfill $ \square $
\end{lemma}
\begin{proof}
	Consider cells $ c,d\in\mathcal{C}_{\mathcal{G}} $ in a network $ \mathcal{G} $ such that $ \bp(c) = \bp(d) $. Then $ 	\mathcal{U}_{\mathcal{Q}}(\bp (c))
	=
	\mathcal{U}_{\mathcal{Q}}(\bp (d)) $ in the quotient network $ \mathcal{Q} = \mathcal{G}/\bp $. Since $ \bp $ is $ \mathcal{U} $-invariant, from \cref{lemma:U_invariant_equiv} we have $ \bp(\mathcal{U}_{\mathcal{G}}(c)) = \bp(\mathcal{U}_{\mathcal{G}}(d)) $. Therefore $ \bp $ is $ \mathcal{U} $-matched.
\end{proof}
\begin{lemma}\label{lemma:A_invariant_quotient_relation}
	Consider a $ \mathcal{U} $-invariant balanced partition $ \bp \in \Lambda_{\mathcal{G}}$ on a network $ \mathcal{G} $ and its respective quotient network $ \mathcal{Q} = \mathcal{G}/\bp $.
	Then for a partition $ \mathcal{A} \geq \bp $, we have
	\begin{eqnarray}
	\mathcal{A}(\mathcal{U}_{\mathcal{G}}(c)) = \mathcal{A}/\bp (\mathcal{U}_{\mathcal{Q}}(\bp(c)))
	\end{eqnarray}
	for all $ c \in \mathcal{C}_{\mathcal{G}} $.
	\hfill $ \square $
\end{lemma}
\begin{proof}
	Because $ \mathcal{A} \geq \bp $, we have $ \mathcal{A}(\mathcal{U}_{\mathcal{G}}(c)) 
	= 
	\mathcal{A}/\bp(\bp(\mathcal{U}_{\mathcal{G}}(c)))  $. Since $ \bp $ is $ \mathcal{U} $-invariant, from \cref{lemma:U_invariant_equiv}, this becomes $ \mathcal{A}/\bp (\mathcal{U}_{\mathcal{Q}}(\bp(c))) $.
\end{proof}
\begin{lemma}
	Consider a $ \mathcal{U} $-invariant balanced partition $ \bp \in \Lambda_{\mathcal{G}}$ on a network $ \mathcal{G} $ and its corresponding quotient network $ \mathcal{Q} = \mathcal{G}/\bp $.
	Then a partition $ \mathcal{A} \geq \bp $ is $ \mathcal{U} $-matched in $ \mathcal{G} $ if and only if $ \mathcal{A}/\bp $ is $ \mathcal{U} $-matched in $ \mathcal{Q} $.
	\hfill $ \square $
\end{lemma}
\begin{proof}
	Firstly, $ \mathcal{A}/\bp $ being $ \mathcal{U} $-matched in $ \mathcal{Q} $, by definition, means that $ \mathcal{A}/\bp(\bp(c)) = \mathcal{A}/\bp(\bp(d)) $ implies $ \mathcal{A}/\bp(\mathcal{U}_{\mathcal{Q}}(\bp(c))) = \mathcal{A}/\bp(\mathcal{U}_{\mathcal{Q}}(\bp(d))) $. This simplifies into $ \mathcal{A}(c) = \mathcal{A}(d) $ implies $ \mathcal{A}/\bp(\mathcal{U}_{\mathcal{Q}}(\bp(c))) = \mathcal{A}/\bp(\mathcal{U}_{\mathcal{Q}}(\bp(d))) $. Therefore we have to prove that if $ \mathcal{A}(c) = \mathcal{A}(d) $, then $ \mathcal{A}(\mathcal{U}_{\mathcal{G}}(c)) = \mathcal{A}(\mathcal{U}_{\mathcal{G}}(d)) $ is equivalent to
	$ \mathcal{A}/\bp(\mathcal{U}_{\mathcal{Q}}(\bp(c))) = \mathcal{A}/\bp(\mathcal{U}_{\mathcal{Q}}(\bp(d))) $. Since $ \bp $ is $ \mathcal{U} $-invariant, this is immediate from \cref{lemma:A_invariant_quotient_relation}.
\end{proof}
\begin{lemma}
	Consider a $ \mathcal{U} $-invariant balanced partition $ \bp_{01}\in \Lambda_{\mathcal{G}}$ on a network $ \mathcal{G} $ and its corresponding quotient network $ \mathcal{Q}_1 = \mathcal{G}/\bp_{01} $. Then a balanced partition $ \bp_{02} \geq \bp_{01}$ is $ \mathcal{U} $-invariant in $ \mathcal{G} $ if and only if $ \bp_{12} :=\bp_{02} /\bp_{01} $ is $ \mathcal{U} $-invariant in $ \mathcal{Q}_{1} $.
	\hfill $ \square $
\end{lemma}
\begin{proof}
	From \cref{lemma:U_invariant_equiv}, $ \bp_{02} $ being $ \mathcal{U} $-invariant in $ \mathcal{G} $ means that $ 	\bp_{02}(\mathcal{U}_{\mathcal{G}}(c))
	= \mathcal{U}_{\mathcal{Q}_2}(\bp_{02} (c)) $, with $ \mathcal{Q}_2 := \mathcal{G}/\bp_{02} $, for all $ c\in\mathcal{C}_{\mathcal{G}} $. 
	This can be rewritten, using $ \bp_{12} $ as $ 	\bp_{12}(\bp_{01}(\mathcal{U}_{\mathcal{G}}(c)))
	= \mathcal{U}_{\mathcal{Q}_2}(\bp_{12}(\bp_{01} (c))) $. Since by assumption, $ \bp_{01} $ is $ \mathcal{U} $-invariant, this can be equivalently written as $ 	\bp_{12}(\mathcal{U}_{\mathcal{Q}_1}(\bp_{01}(c))
	= \mathcal{U}_{\mathcal{Q}_2}(\bp_{12}(\bp_{01} (c))) $, for all $ c\in\mathcal{C}_{\mathcal{G}} $. Using the mapping $ d = \bp_{01}(c) $, it is easy to see that this is equivalent to $ 	\bp_{12}(\mathcal{U}_{\mathcal{Q}_1}(d)
	= \mathcal{U}_{\mathcal{Q}_2}(\bp_{12}(d)) $, for all $ d \in \mathcal{C}_{\mathcal{Q}_1} $.
	Since from \cref{lemma:balanced_quotient_invariance}, $ \mathcal{Q}_2 = \mathcal{Q}_1 / \bp_{12} $, this is equivalent to $ \bp_{12} $ being $ \mathcal{U} $-invariant in $ \mathcal{Q}_{1} $.
\end{proof}
Similarly to the $ \mathcal{U} $-matched case, we have the following results.
\begin{corollary}
	\label{lemma:trivial_U_invariant}
	The trivial partition $ \bot $ is $ \mathcal{U} $-invariant for every $ \mathcal{U}\in \{\mathcal{N}^{-},\mathcal{V}^{-},\mathcal{V}_{k}^{-},\mathcal{R}^{-}\} $.
	\hfill $ \square $
\end{corollary}
\begin{corollary}\label{lemma:n_invariant_implies_v_invariant}
	If a balanced partition $ \bp $ is $ \mathcal{N}^{-} $-invariant, then it is $ \mathcal{V}^{-} $-invariant.
	\hfill $ \square $
\end{corollary}
\begin{proof}
	Consider cells $ c,d\in\mathcal{C}_{\mathcal{G}} $ in a network $ \mathcal{G} $ such that $ c\in \mathcal{V}^{-}_{\mathcal{G}}(d) $. Then $ c\in \{d \cup \mathcal{N}^{-}_{\mathcal{G}}(d)\} $. Consider the case $ c=d $. Then $ \bp (c) \in \mathcal{V}^{-}_{\mathcal{Q}}(\bp (c)) $ in the quotient network $ \mathcal{Q}=\mathcal{G}/\bp $ is immediate from the definition of $ \mathcal{V}^{-} $. Consider now that $ c\in  \mathcal{N}^{-}_{\mathcal{G}}(d) $. Then since $ \bp $ is $ \mathcal{N}^{-} $-invariant, $ \bp (c) \in \mathcal{N}^{-}_{\mathcal{Q}}(\bp (d)) $, which implies $ \bp (c) \in \mathcal{V}^{-}_{\mathcal{Q}}(\bp (d)) $.
\end{proof}
\begin{lemma}\label{lemma:v_invariant_implies_v_k_invariant}
	If a balanced partition $ \bp $ is $ \mathcal{V}^{-} $-invariant, then it is $ \mathcal{V}_{k}^{-} $-invariant for every $ k\geq1 $.
	\hfill $ \square $
\end{lemma}
\begin{proof}
	The proof is by induction. The base case $ k=1 $ is trivial since $ \mathcal{V}_{1}^{-} = \mathcal{V}^{-} $. We assume that the statement applies to a given $ k $. That is, $ \bp $ is both $ \mathcal{V}^{-} $-invariant and $ \mathcal{V}_{k}^{-} $-invariant.
	Consider cells $ c,d\in\mathcal{C}_{\mathcal{G}} $ in a network $ \mathcal{G} $ such that $ c\in \mathcal{V}^{-}_{k+1 \, \mathcal{G}}(d) $. Then $ c \in \bigcup_{d^{\star}\in\mathcal{V}_{k\,\mathcal{G}}^{-}(d)} \mathcal{V}^{-}_{\mathcal{G}}(d^{\star}) $. That is, $ c\in \mathcal{V}^{-}_{\mathcal{G}}(d^{\star}) $ for at least one particular $ d^{\star}\in\mathcal{V}_{k\,\mathcal{G}}^{-}(d) $.\\
	Then from $ \bp $ being $ \mathcal{V}^{-} $-invariant and $ \mathcal{V}_{k}^{-} $-invariant, we have $ \bp(c)\in \mathcal{V}^{-}_{\mathcal{Q}}(\bp(d^{\star})) $	and $ \bp(d^{\star})\in\mathcal{V}_{k\,\mathcal{Q}}^{-}(\bp(d)) $, respectively, in the quotient network $ \mathcal{Q}=\mathcal{G}/\bp $. Therefore $ \bp(c)\in \mathcal{V}_{k+1\,\mathcal{Q}}^{-}(\bp(d))  $, which means that $ \bp $ is $ \mathcal{V}_{k+1}^{-} $-invariant.
\end{proof}
\begin{corollary}
	If a balanced partition $ \bp $ defined on a finite set of cells is $ \mathcal{V}^{-} $-invariant, then it is also $ \mathcal{R}^{-} $-invariant.
	\label{lemma:v_invariant_implies_r_invariant}
	\hfill $ \square $
\end{corollary}
\begin{proof}
	This is direct from \cref{lemma:v_invariant_implies_v_k_invariant} and 	\cref{cor:cumulative_neigh_k_inside_reachable}.
\end{proof}
The concept of $ \mathcal{U} $-invariance generalizes the concept of spurious partitions, which was defined in \cite{aguiar2017patterns}. In particular, it corresponds to partitions not being $ \mathcal{N}^{-} $-invariant. This is illustrated in the following example.
\begin{exmp}\label{exmp:spurious_partition_example}
	Consider the network in \fref{fig:quotient_cancelinga}. For a general admissible function $ f\in\mathcal{F}_{\mathcal{G}} $, $ f_3 $ depends on the states of cells $ 1, 2 $. However, when the state is in $ \Delta_\bp^{\xset} $ with \mbox{$ \bp = \{ \{1,2\},\{3\},\{4\} \} $}, the total effect of cells $ 1,2 $ on cell $ 3 $ cancels and $ 3 $ acts as if there were no edges coming from those cells. 
	\begin{figure}[h]
		\centering
		\begin{subfigure}[t]{0.3\textwidth}
			\centering
			\begin{tikzpicture}[
node1/.style = {circle,minimum size=23,draw},
node2/.style = {circle,minimum size=23,draw,fill=white!75!black},
node3/.style = {circle,minimum size=23,draw,fill=white!50!black},
noderect/.style = {rectangle,minimum size=20,draw},
edge1/.style = {>=latex,thick},
edgedashed/.style = {>=latex,thick,dashed},
edge2/.style = {>=latex,thick,blue},
edge3/.style = {>=latex,thick,red}
]
\node[node1] at (0,0)(n1){1};
\node[node1] at (2,0)(n2){2};
\node[node2] at (1,{sqrt(2)})(n3){3};
\node[node3] at (1,{2*sqrt(2)})(n4){4};

\draw [->,edge1](n1) -- (n3);
%\draw [->,edge1](n2) -- (n3);
\draw [->,edgedashed](n2) -- (n3);
%\node (minus23) at ($(n2)!0.4!(n3) + (-0.2,-0.2)$) {(-)};
\draw [->,edge1](n3) -- (n4);

\draw [->,edge1] (n4) edge[bend right,looseness=0.8] (n1);
\draw [->,edge1] (n4) edge[bend left,looseness=0.8] (n2);

%\DoubleLine{n1}{n2}{<-,edge1}{}{->,edge1}{}
%\DoubleLine{n1}{n3}{<-,edge1}{}{->,edge1}{}
%\DoubleLine{n2}{n3}{<-,edge1}{}{->,edge1}{}

\end{tikzpicture} 
			\caption{Original network.}
			\label{fig:quotient_cancelinga}
		\end{subfigure}
		\begin{subfigure}[t]{0.3\textwidth}
			\centering
			\begin{tikzpicture}[
node1/.style = {circle,minimum size=23,draw},
node2/.style = {circle,minimum size=23,draw,fill=white!75!black},
node3/.style = {circle,minimum size=23,draw,fill=white!50!black},
noderect/.style = {rectangle,minimum size=20,draw},
edge1/.style = {>=latex,thick},
edgedashed/.style = {>=latex,thick,dashed},
edge2/.style = {>=latex,thick,blue},
edge3/.style = {>=latex,thick,red}
]
\node[node1] at (0,0)(n1){1 2};
\node[node2] at (2,0)(n2){3};
\node[node3] at (1,{sqrt(2)})(n3){4};

\draw [->,edge1](n2) -- (n3);
\draw [->,edge1](n3) -- (n1);
\DoubleLine{n1}{n2}{->,edgedashed}{}{->,edge1}{}

%\node (minus123) at ($(n1)!0.4!(n2) + (0,0.3)$) {(-)};

%\DoubleLine{n1}{n2}{<-,edge1}{}{->,edge1}{}
%\DoubleLine{n1}{n3}{<-,edge1}{}{->,edge1}{}
%\DoubleLine{n2}{n3}{<-,edge1}{}{->,edge1}{}

\end{tikzpicture} 
			\caption{Edge canceling.}
			\label{fig:quotient_cancelingb}
		\end{subfigure}
		\begin{subfigure}[t]{0.3\textwidth}
			\centering
			\begin{tikzpicture}[
node1/.style = {circle,minimum size=23,draw},
node2/.style = {circle,minimum size=23,draw,fill=white!75!black},
node3/.style = {circle,minimum size=23,draw,fill=white!50!black},
noderect/.style = {rectangle,minimum size=20,draw},
edge1/.style = {>=latex,thick},
edge2/.style = {>=latex,thick,blue},
edge3/.style = {>=latex,thick,red}
]
\node[node1] at (0,0)(n1){1 2};
\node[node2] at (2,0)(n2){3};
\node[node3] at (1,{sqrt(2)})(n3){4};

\draw [->,edge1](n2) -- (n3);
\draw [->,edge1](n3) -- (n1);
%\DoubleLine{n1}{n2}{->,edge1}{}{->,edge1}{(-)}

%\DoubleLine{n1}{n2}{<-,edge1}{}{->,edge1}{}
%\DoubleLine{n1}{n3}{<-,edge1}{}{->,edge1}{}
%\DoubleLine{n2}{n3}{<-,edge1}{}{->,edge1}{}

\end{tikzpicture} 
			\caption{Quotient network.}
			\label{fig:quotient_cancelingc}
		\end{subfigure}
		\caption{Example of a spurious (not $ \mathcal{N}^{-} $-invariant) partition.}
		\label{fig:quotient_canceling_example}
	\end{figure}
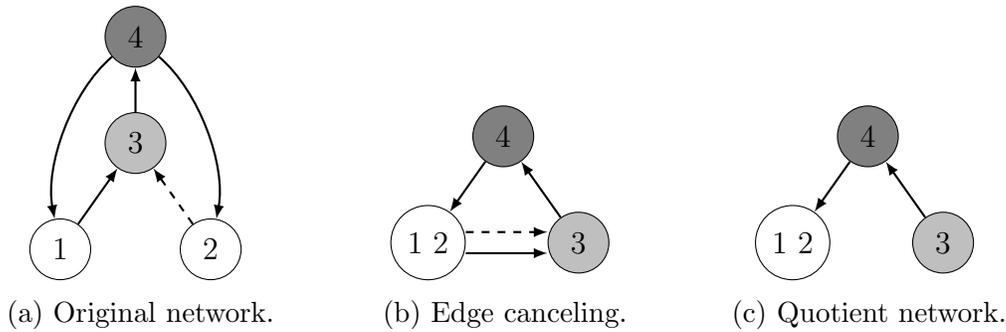
	\hfill $ \square $
\end{exmp}
The partition in \cref{exmp:spurious_partition_example} is $ \mathcal{N}^{-} $-matched despite not being $ \mathcal{N}^{-} $-invariant. That is, while being $ \mathcal{N}^{-} $-invariant is a sufficient condition for a partition to be $ \mathcal{N}^{-} $-matched, it is not a necessary one. We now present an example that clarifies why edge canceling in a balanced partition that is not $ \mathcal{N}^{-} $-invariant might lead to it not being $ \mathcal{N}^{-} $-matched.
\begin{exmp}\label{exmp:not_n_invariant_not_n_matched}
	Consider the network in \fref{fig:not_n_invariant_not_n_matched}, which is colored according to a balanced partition that is not $ \mathcal{N}^{-} $-invariant (that is, it is spurious). Note that \mbox{$ \mathcal{N}^{-}(1) = \{1\}$} contains only white colors and $ \mathcal{N}^{-}(4) = \{1,2,3\}$ contains white and gray colors. The fact that the edges coming from cells $ 2 $ and $ 3 $ cancel each other is exactly what allows this partition to be balanced despite this difference.
	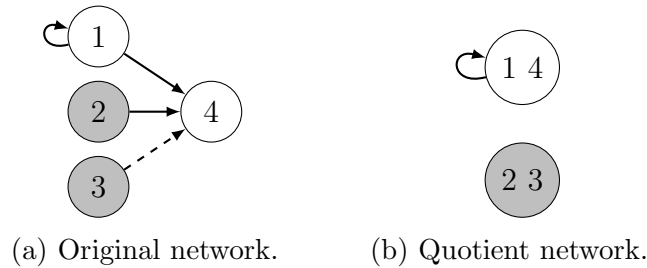
\begin{figure}[h]
		\centering
		\begin{subfigure}[t]{0.3\textwidth}
			\centering
			\begin{tikzpicture}[
node1/.style = {circle,minimum size=23,draw},
node2/.style = {circle,minimum size=23,draw,fill=white!75!black},
node3/.style = {circle,minimum size=23,draw,fill=white!50!black},
noderect/.style = {rectangle,minimum size=20,draw},
edge1/.style = {>=latex,thick},
edgedashed/.style = {>=latex,thick,dashed},
edge2/.style = {>=latex,thick,blue},
edge3/.style = {>=latex,thick,red}
]
\node[node1] at (0,1.0)(n1){1};
\node[node2] at (0,0)(n2){2};
\node[node2] at (0,-1.0)(n3){3};
\node[node1] at (1.5,0)(n4){4};

%\DoubleLine{n1}{n3}{<-,edge1}{}{->,edge1}{}
\draw [->,edge1] (n1) edge[loop left,looseness=5] (n1);
\draw [->,edge1](n1) -- (n4);
\draw [->,edge1](n2) -- (n4);
\draw [->,edgedashed](n3) -- (n4);

%\node (minus123) at ($(n3)!0.4!(n4) + (0.2,-0.2)$) {(-)};
%\DoubleLine{n2}{n3}{<-,edge1}{}{->,edge1}{}

%\draw [->,edge1] (n1) edge[loop left,looseness=5] (n1);
%\draw [->,edge1] (n3) edge[loop above,looseness=5] (n3);

%\node (minus123) at ($(n1)!0.4!(n2) + (0,0.3)$) {(-)};

%\DoubleLine{n1}{n2}{<-,edge1}{}{->,edge1}{}
%\DoubleLine{n1}{n3}{<-,edge1}{}{->,edge1}{}
%\DoubleLine{n2}{n3}{<-,edge1}{}{->,edge1}{}

\end{tikzpicture} 
			\caption{Original network.}
			\label{fig:not_n_invariant_not_n_matched}
		\end{subfigure}
		\begin{subfigure}[t]{0.3\textwidth}
			\centering
			\begin{tikzpicture}[
node1/.style = {circle,minimum size=23,draw},
node2/.style = {circle,minimum size=23,draw,fill=white!75!black},
node3/.style = {circle,minimum size=23,draw,fill=white!50!black},
noderect/.style = {rectangle,minimum size=20,draw},
edge1/.style = {>=latex,thick},
edgedashed/.style = {>=latex,thick,dashed},
edge2/.style = {>=latex,thick,blue},
edge3/.style = {>=latex,thick,red}
]
\node[node1] at (0,1.0)(n14){1 4};
\node[node2] at (0,-0.5)(n23){2 3};

%\DoubleLine{n1}{n3}{<-,edge1}{}{->,edge1}{}
\draw [->,edge1] (n14) edge[loop left,looseness=5] (n14);

%\node (minus123) at ($(n3)!0.4!(n4) + (0.2,-0.2)$) {(-)};
%\DoubleLine{n2}{n3}{<-,edge1}{}{->,edge1}{}

%\draw [->,edge1] (n1) edge[loop left,looseness=5] (n1);
%\draw [->,edge1] (n3) edge[loop above,looseness=5] (n3);

%\node (minus123) at ($(n1)!0.4!(n2) + (0,0.3)$) {(-)};

%\DoubleLine{n1}{n2}{<-,edge1}{}{->,edge1}{}
%\DoubleLine{n1}{n3}{<-,edge1}{}{->,edge1}{}
%\DoubleLine{n2}{n3}{<-,edge1}{}{->,edge1}{}

\end{tikzpicture} 
			\caption{Quotient network.}
			\label{fig:not_n_invariant_not_n_matched_quotient}
		\end{subfigure}
		\caption{Example of a partition that is neither $ \mathcal{N}^{-} $-invariant nor $ \mathcal{N}^{-} $-matched.}
		\label{fig:not_n_invariant_not_n_matched_exmp}
	\end{figure}
	\hfill $ \square $
\end{exmp}
We have seen in \cref{lemma:n_invariant_implies_v_invariant} that a $ \mathcal{N}^{-} $-invariant partition is also $ \mathcal{V}^{-} $-invariant. The next example shows that the converse is not necessarily true.
\begin{exmp}\label{exmp:not_n_but_v_invariant}
	Consider the network in \fref{fig:not_n_but_v_invariant}, which is colored with a single color, according to the balanced partition $ \{\{1,2,3\}\} $. In this network, both the $ \mathcal{N}^{-} $ and $ \mathcal{V}^{-} $ in-neighborhoods of white cells contain white cells. On the other hand, in the quotient network in \fref{fig:not_n_but_v_invariant_quotient}, we see that $ \mathcal{N}^{-} $ of its only existing cell is empty. Therefore this partition is not  $\mathcal{N}^{-} $-invariant. It is, however, $\mathcal{V}^{-} $-invariant.
	\begin{figure}[h]
		\centering
		\begin{subfigure}[t]{0.3\textwidth}
			\centering
			\begin{tikzpicture}[
node1/.style = {circle,minimum size=23,draw},
node2/.style = {circle,minimum size=23,draw,fill=white!75!black},
node3/.style = {circle,minimum size=23,draw,fill=white!50!black},
noderect/.style = {rectangle,minimum size=20,draw},
edge1/.style = {>=latex,thick},
edgedashed/.style = {>=latex,thick,dashed},
edge2/.style = {>=latex,thick,blue},
edge3/.style = {>=latex,thick,red}
]
\node[node1] at (0,0)(n1){1};
\node[node1] at (2,0)(n2){2};
\node[node1] at (1,{sqrt(2)})(n3){3};

\DoubleLine{n1}{n2}{<-,edgedashed}{}{->,edge1}{}
\DoubleLine{n2}{n3}{<-,edgedashed}{}{->,edge1}{}
\DoubleLine{n3}{n1}{<-,edgedashed}{}{->,edge1}{}

%\draw [->,edge1](n2) -- (n3);
%\draw [->,edge1](n3) -- (n1);

%\node (minus123) at ($(n1)!0.4!(n2) + (0,0.3)$) {(-)};

%\DoubleLine{n1}{n2}{<-,edge1}{}{->,edge1}{}
%\DoubleLine{n1}{n3}{<-,edge1}{}{->,edge1}{}
%\DoubleLine{n2}{n3}{<-,edge1}{}{->,edge1}{}

\end{tikzpicture} 
			\caption{Original network.}
			\label{fig:not_n_but_v_invariant}
		\end{subfigure}
		\begin{subfigure}[t]{0.3\textwidth}
			\centering
			\begin{tikzpicture}[
node1/.style = {circle,minimum size=23,draw},
node2/.style = {circle,minimum size=23,draw,fill=white!75!black},
node3/.style = {circle,minimum size=23,draw,fill=white!50!black},
noderect/.style = {rectangle,minimum size=20,draw},
edge1/.style = {>=latex,thick},
edgedashed/.style = {>=latex,thick,dashed},
edge2/.style = {>=latex,thick,blue},
edge3/.style = {>=latex,thick,red}
]

\node[node1] at (0,0)(n1){1 2 3};

%\draw [->,edge1](n2) -- (n3);
%\draw [->,edge1](n3) -- (n1);

%\node (minus123) at ($(n1)!0.4!(n2) + (0,0.3)$) {(-)};

%\DoubleLine{n1}{n2}{<-,edge1}{}{->,edge1}{}
%\DoubleLine{n1}{n3}{<-,edge1}{}{->,edge1}{}
%\DoubleLine{n2}{n3}{<-,edge1}{}{->,edge1}{}

\end{tikzpicture} 
			\caption{Quotient network.}
			\label{fig:not_n_but_v_invariant_quotient}
		\end{subfigure}
		\caption{Example of a partition that is not $ \mathcal{N}^{-} $-invariant but is $ \mathcal{V}^{-} $-invariant.}	
		\label{fig:not_n_but_v_invariant_exmp}	
	\end{figure}
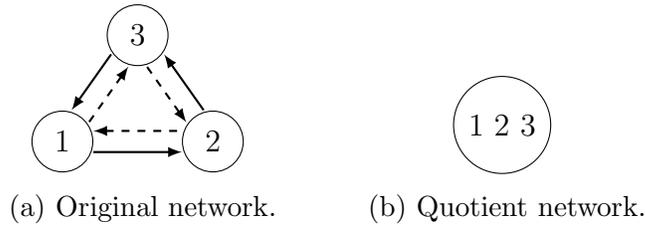	
	\hfill $ \square $
\end{exmp}
We have also seen in \cref{lemma:v_invariant_implies_r_invariant} that a $ \mathcal{V}^{-} $-invariant partition is also $ \mathcal{R}^{-} $-invariant. In the next example we disprove the converse statement.
\begin{exmp}\label{exmp:not_v_but_r_invariant}
	Consider the network in \fref{fig:not_v_but_r_invariant}, which is colored according to the balanced partition $ \{\{1\},\{2,3\},\{4\}\} $. In the original network $ \mathcal{G} $, we have $ \mathcal{V}^{-}(1) = \{1,2,3,4\} $. That is, white cells have white, light gray and dark gray colors in its $ \mathcal{V}^{-} $ neighborhood. On the other hand, in the quotient, the white cell has only white and dark gray colors in its $ \mathcal{V}^{-} $ neighborhood, so the partition is not $\mathcal{V}^{-} $-invariant. However, it is clear that the partition is $\mathcal{R}^{-} $-invariant, since in both the original network and in the quotient, all in-reachability sets $\mathcal{R}^{-} $ contain all the three colors of the partition.
	\begin{figure}[h]
		\centering
		\begin{subfigure}[t]{0.3\textwidth}
			\centering
			\begin{tikzpicture}[
node1/.style = {circle,minimum size=23,draw},
node2/.style = {circle,minimum size=23,draw,fill=white!75!black},
node3/.style = {circle,minimum size=23,draw,fill=white!50!black},
noderect/.style = {rectangle,minimum size=20,draw},
edge1/.style = {>=latex,thick},
edgedashed/.style = {>=latex,thick,dashed},
edge2/.style = {>=latex,thick,blue},
edge3/.style = {>=latex,thick,red}
]
\node[node1] at (0,0)(n1){1};
\node[node2] at (1.7,0)(n2){2};
\node[node2] at (0,-1.7)(n3){3};
\node[node3] at (1.7,-1.7)(n4){4};

\DoubleLine{n1}{n2}{->,edge1}{}{<-,edge1}{}
\DoubleLine{n1}{n3}{<-,edgedashed}{}{->,edge1}{}

\draw [->,edge1](n2) -- (n4);
\draw [->,edge1](n3) -- (n4);
\draw [->,edge1](n4) -- (n1);

%\draw [->,edge1](n2) -- (n3);
%\draw [->,edge1](n3) -- (n1);

%\node (minus123) at ($(n1)!0.4!(n2) + (0,0.3)$) {(-)};

%\DoubleLine{n1}{n2}{<-,edge1}{}{->,edge1}{}
%\DoubleLine{n1}{n3}{<-,edge1}{}{->,edge1}{}
%\DoubleLine{n2}{n3}{<-,edge1}{}{->,edge1}{}

\end{tikzpicture} 
			\caption{Original network.}
			\label{fig:not_v_but_r_invariant}
		\end{subfigure}
		\begin{subfigure}[t]{0.3\textwidth}
			\centering
			\begin{tikzpicture}[
node1/.style = {circle,minimum size=23,draw},
node2/.style = {circle,minimum size=23,draw,fill=white!75!black},
node3/.style = {circle,minimum size=23,draw,fill=white!50!black},
noderect/.style = {rectangle,minimum size=20,draw},
edge1/.style = {>=latex,thick},
edgedashed/.style = {>=latex,thick,dashed},
edge2/.style = {>=latex,thick,blue},
edge3/.style = {>=latex,thick,red}
]
\node[node1] at (0,0)(n1){1};
\node[node2] at (1,{sqrt(2)})(n23){2 3};
\node[node3] at (2,0)(n4){4};

\draw [->,edge1](n1) -- (n23);
\draw [->,edge1](n23) -- (n4);
\draw [->,edge1](n4) -- (n1);
%\DoubleLine{n1}{n2}{->,edgedashed}{}{->,edge1}{}

%\node (minus123) at ($(n1)!0.4!(n2) + (0,0.3)$) {(-)};

%\DoubleLine{n1}{n2}{<-,edge1}{}{->,edge1}{}
%\DoubleLine{n1}{n3}{<-,edge1}{}{->,edge1}{}
%\DoubleLine{n2}{n3}{<-,edge1}{}{->,edge1}{}

\end{tikzpicture} 
			\caption{Quotient network.}
			\label{fig:not_v_but_r_invariant_quotient}
		\end{subfigure}
		\caption{Example of a partition that is not $ \mathcal{V}^{-} $-invariant but is $ \mathcal{R}^{-} $-invariant.}	
		\label{fig:not_v_but_r_invariant_exmp}	
	\end{figure}
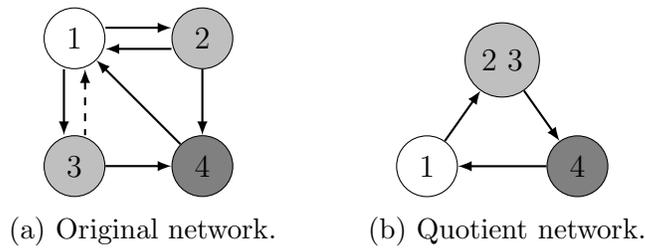	
	\hfill $ \square $
\end{exmp}
In summary, we have shown that the sequence: $ \mathcal{N}^{-} $-invariant , $ \mathcal{V}^{-} $-invariant and $ \mathcal{R}^{-} $-invariant lists progressively weaker assumptions.
\begin{remark}
	Refer back to \cref{exmp:R_matched_lattice_example}. The network contains only positive weights, so no matter which quotient we apply, there is no edge canceling. It follows that every balanced partition in that network is $ \mathcal{N}^{-} $-invariant. Then it is also $ \mathcal{R}^{-} $-invariant, which from \cref{lemma:u_invariant_implies_u_matched} means that they are $ \mathcal{R}^{-} $-matched.
	\hfill $ \square $
\end{remark}
We now show that the tamest assumption we defined in this section ($ \mathcal{R}^{-} $-invariant) is enough to allow the following results.
\begin{lemma}
	Consider a $ \mathcal{R}^{-} $-invariant balanced partition $ \bp \in \Lambda_{\mathcal{G}}$ on a network $ \mathcal{G} $ and its corresponding quotient network $ \mathcal{Q} = \mathcal{G}/\bp $. If a partition $ \mathcal{A} \geq \bp $ is strong in $ \mathcal{G} $, then $ \mathcal{A}/\bp $ is strong in $ \mathcal{Q} $.
	\hfill $ \square $
\end{lemma}
\begin{proof}
	Firstly, note that $ \mathcal{A}/\bp $ being strong in $ \mathcal{Q} $, by definition, means that $ \mathcal{A}/\bp(\bp(c)) = \mathcal{A}/\bp(\bp(d)) $ implies $ \mathcal{R}_{\mathcal{Q}}^{-}(\bp(c)) = \mathcal{R}_{\mathcal{Q}}^{-}(\bp(d))$. This simplifies to $ \mathcal{A}(c) = \mathcal{A}(d) $ implies $ \mathcal{R}_{\mathcal{Q}}^{-}(\bp(c)) = \mathcal{R}_{\mathcal{Q}}^{-}(\bp(d))$. Assume $ \mathcal{A}(c) = \mathcal{A}(d) $. Then from $ \mathcal{A} $ being strong in $ \mathcal{G} $, we have $ \mathcal{R}_{\mathcal{G}}^{-}(c) = \mathcal{R}_{\mathcal{G}}^{-}(d)$, which implies $\bp( \mathcal{R}_{\mathcal{G}}^{-}(c)) =
	\bp( \mathcal{R}_{\mathcal{G}}^{-}(d))$. Since $ \bp $ is $ \mathcal{R}^{-} $-invariant, \cref{lemma:U_invariant_equiv} implies $ 
	\mathcal{R}_{\mathcal{Q}}^{-}(\bp (c)) 
	=
	\mathcal{R}_{\mathcal{Q}}^{-}(\bp (d)) 
	$.
\end{proof}
\begin{lemma}
	Consider a $ \mathcal{R}^{-} $-invariant balanced partition $ \bp \in \Lambda_{\mathcal{G}}$ on a network $ \mathcal{G} $ and its corresponding quotient network $ \mathcal{Q} = \mathcal{G}/\bp $. If a partition $ \mathcal{A} \geq \bp $ is non-weak in $ \mathcal{G} $, then $ \mathcal{A}/\bp $ is non-weak in $ \mathcal{Q} $.
	\hfill $ \square $
\end{lemma}
\begin{proof}
	Assume $ \mathcal{A} $ is non-weak in $ \mathcal{G} $. Then for every color $ A\in\mathcal{A} $, we have $ \bigcap_{c\in A} \mathcal{R}_{\mathcal{G}}^{-}(c) \neq \emptyset $. Then $ \bp\left(\bigcap_{c\in A} 	\mathcal{R}_{\mathcal{G}}^{-}(c)\right) \neq \emptyset $. Note that $ \bigcap_{c\in A} \bp (\mathcal{R}_{\mathcal{G}}^{-}(c)) \supseteq
	\bp\left(\bigcap_{c\in A} 	\mathcal{R}_{\mathcal{G}}^{-}(c)\right) $, therefore $ \bigcap_{c\in A} 	\bp (\mathcal{R}_{\mathcal{G}}^{-}(c)) \neq \emptyset $. Since $ \bp $ is $ \mathcal{R}^{-} $-invariant, \cref{lemma:U_invariant_equiv} implies $ \bigcap_{c\in A} \mathcal{R}_{\mathcal{Q}}^{-}(\bp(c)) \neq \emptyset $. This can written as $ \bigcap_{\bp(c)\in A/\bp } \mathcal{R}_{\mathcal{Q}}^{-}(\bp(c)) \neq \emptyset $, so $ \mathcal{A}/\bp $ is non-weak in $ \mathcal{Q} $.
\end{proof}
These results are summarized in the left hand side of \tref{table:quotient_bp_r_invariant}, where, as before, $ S $, $ R $ and $ W $ denote the partition classifications of strong, rooted and weak, respectively. The right hand side is easily seen to be equivalent to the left one.
\begin{table}[h]
	\caption{Relation between partitions and their quotients over a $ \mathcal{R}^{-} $-invariant partition.}
	\begin{center}
		\begin{tabular}{cc}
			\multicolumn{2}{c}{$ \mathcal{A},\mathcal{G} \to \mathcal{A}/\bp,\mathcal{Q} $}     \\ \hline
			\multicolumn{1}{c|}{S} & S \\
			\multicolumn{1}{c|}{R} & S/R \\
			\multicolumn{1}{c|}{W} & S/R/W
		\end{tabular}
		\qquad
		\begin{tabular}{cc}
			\multicolumn{2}{c}{$ \mathcal{A}/\bp,\mathcal{Q} \to \mathcal{A},\mathcal{G} $}     \\ \hline
			\multicolumn{1}{c|}{S} & S/R/W \\
			\multicolumn{1}{c|}{R} & R/W \\
			\multicolumn{1}{c|}{W} & W
		\end{tabular}
	\end{center}
	\label{table:quotient_bp_r_invariant}
\end{table}
For \tref{table:quotient_bp_r_invariant} to apply, we require the partition we quotient over ($ \bp $) to be $ \mathcal{R}^{-} $-invariant. We now present some examples that show that these results do not apply if this assumption is not satisfied.
\begin{exmp}
	Consider the network $ \mathcal{G} $ in \fref{fig:A_bp_not_invariant_R_to_W}, which is colored according to the balanced partition $\bp = \{\{1,2\},\{3\}\} $.
	Note that $ \mathcal{R}_{\mathcal{G}}^{-}(3) = \{1,2,3\} $. That is, gray cells have white and gray colors in their $ \mathcal{R}_{\mathcal{G}}^{-} $ neighborhoods.
	On the other hand, in the quotient network $ \mathcal{Q} $ in \fref{fig:A_bp_not_invariant_R_to_W_quotient}, $ \mathcal{R}_{\mathcal{Q}}^{-}(3) = \{3\} $. That is, the gray cell has only the gray color in its $ \mathcal{R}_{\mathcal{Q}}^{-} $ neighborhood. Therefore $ \bp $ is not $\mathcal{R}^{-} $-invariant. Consider now the partition $ \mathcal{A} = \{\{1,2,3\}\} $. 
	Although this partition is rooted in $ \mathcal{G} $, its quotient $ \mathcal{A}/\bp = \{\{12,3\}\}$ is weak in $ \mathcal{Q} $.
	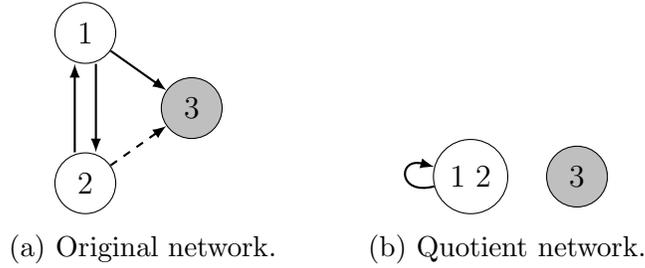
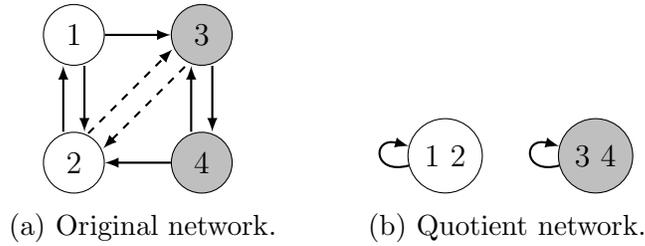
\begin{figure}[h]
		\centering
		\begin{subfigure}[t]{0.3\textwidth}
			\centering
			\begin{tikzpicture}[
node1/.style = {circle,minimum size=23,draw},
node2/.style = {circle,minimum size=23,draw,fill=white!75!black},
node3/.style = {circle,minimum size=23,draw,fill=white!50!black},
noderect/.style = {rectangle,minimum size=20,draw},
edge1/.style = {>=latex,thick},
edgedashed/.style = {>=latex,thick,dashed},
edge2/.style = {>=latex,thick,blue},
edge3/.style = {>=latex,thick,red}
]
\node[node1] at (0,2)(n1){1};
\node[node1] at (0,0)(n2){2};
\node[node2] at ({sqrt(2)},1)(n3){3};

\DoubleLine{n1}{n2}{->,edge1}{}{<-,edge1}{}
\draw [->,edge1](n1) -- (n3);
\draw [->,edgedashed](n2) -- (n3);

%\DoubleLine{n1}{n3}{<-,edgedashed}{}{->,edge1}{}

%\draw [->,edge1](n3) -- (n4);
%\draw [->,edge1](n4) -- (n1);

%\draw [->,edge1](n2) -- (n3);
%\draw [->,edge1](n3) -- (n1);

%\node (minus123) at ($(n1)!0.4!(n2) + (0,0.3)$) {(-)};

%\DoubleLine{n1}{n2}{<-,edge1}{}{->,edge1}{}
%\DoubleLine{n1}{n3}{<-,edge1}{}{->,edge1}{}
%\DoubleLine{n2}{n3}{<-,edge1}{}{->,edge1}{}

\end{tikzpicture} 
			\caption{Original network.}
			\label{fig:A_bp_not_invariant_R_to_W}
		\end{subfigure}
		\begin{subfigure}[t]{0.3\textwidth}
			\centering
			\begin{tikzpicture}[
node1/.style = {circle,minimum size=23,draw},
node2/.style = {circle,minimum size=23,draw,fill=white!75!black},
node3/.style = {circle,minimum size=23,draw,fill=white!50!black},
noderect/.style = {rectangle,minimum size=20,draw},
edge1/.style = {>=latex,thick},
edgedashed/.style = {>=latex,thick,dashed},
edge2/.style = {>=latex,thick,blue},
edge3/.style = {>=latex,thick,red}
]
\node[node1] at (0,0)(n12){1 2};
\node[node2] at ({sqrt(2)},0)(n3){3};

\draw [->,edge1] (n12) edge[loop left,looseness=5] (n12);

%\draw [->,edge1](n1) -- (n23);
%\draw [->,edge1](n23) -- (n4);
%\draw [->,edge1](n4) -- (n1);
%\DoubleLine{n1}{n2}{->,edgedashed}{}{->,edge1}{}

%\node (minus123) at ($(n1)!0.4!(n2) + (0,0.3)$) {(-)};

%\DoubleLine{n1}{n2}{<-,edge1}{}{->,edge1}{}
%\DoubleLine{n1}{n3}{<-,edge1}{}{->,edge1}{}
%\DoubleLine{n2}{n3}{<-,edge1}{}{->,edge1}{}

\end{tikzpicture} 
			\caption{Quotient network.}
			\label{fig:A_bp_not_invariant_R_to_W_quotient}
		\end{subfigure}
		\caption{Example of a quotient over a partition that is not $\mathcal{R}^{-} $-invariant.}
		\label{fig:A_bp_not_invariant_R_to_W_exmp}	
	\end{figure}	
	\hfill $ \square $
\end{exmp}
\begin{exmp}
	Consider the network $ \mathcal{G} $ in \fref{fig:A_bp_not_invariant_S_to_W}, which is colored according to the balanced partition $\bp = \{\{1,2\},\{3,4\}\} $. Note that $ \mathcal{G} $ consists of a single SCC. Therefore each cell has white and gray colors in its $ \mathcal{R}_{\mathcal{G}}^{-} $ neighborhood.	On the other hand, in the quotient network $ \mathcal{Q} $ in \fref{fig:A_bp_not_invariant_S_to_W_quotient}, $ \mathcal{R}_{\mathcal{Q}}^{-}(12) = \{12\} $. That is, the white cell has only the white color in its $ \mathcal{R}_{\mathcal{Q}}^{-} $ neighborhood. Therefore $ \bp $ is not $\mathcal{R}^{-} $-invariant. Consider now the partition $ \mathcal{A} = \{\{1,2,3,4\}\} $. 
	Although this partition is strong in $ \mathcal{G} $, its quotient $ \mathcal{A}/\bp = \{\{12,34\}\}$ is weak in $ \mathcal{Q} $.
	\begin{figure}[h]
		\centering
		\begin{subfigure}[t]{0.3\textwidth}
			\centering
			\begin{tikzpicture}[
node1/.style = {circle,minimum size=23,draw},
node2/.style = {circle,minimum size=23,draw,fill=white!75!black},
node3/.style = {circle,minimum size=23,draw,fill=white!50!black},
noderect/.style = {rectangle,minimum size=20,draw},
edge1/.style = {>=latex,thick},
edgedashed/.style = {>=latex,thick,dashed},
edge2/.style = {>=latex,thick,blue},
edge3/.style = {>=latex,thick,red}
]
\node[node1] at (0,0)(n1){1};
\node[node1] at (0,-1.7)(n2){2};
\node[node2] at (1.7,0)(n3){3};
\node[node2] at (1.7,-1.7)(n4){4};

\DoubleLine{n1}{n2}{->,edge1}{}{<-,edge1}{}
\DoubleLine{n3}{n4}{->,edge1}{}{<-,edge1}{}

\draw [->,edge1](n1) -- (n3);
\draw [->,edge1](n4) -- (n2);

\DoubleLine{n2}{n3}{->,edgedashed}{}{<-,edgedashed}{}

%\draw [->,edge1](n2) -- (n3);
%\draw [->,edge1](n3) -- (n1);

%\node (minus123) at ($(n1)!0.4!(n2) + (0,0.3)$) {(-)};

%\DoubleLine{n1}{n2}{<-,edge1}{}{->,edge1}{}
%\DoubleLine{n1}{n3}{<-,edge1}{}{->,edge1}{}
%\DoubleLine{n2}{n3}{<-,edge1}{}{->,edge1}{}

\end{tikzpicture} 
			\caption{Original network.}
			\label{fig:A_bp_not_invariant_S_to_W}
		\end{subfigure}
		\begin{subfigure}[t]{0.3\textwidth}
			\centering
			\begin{tikzpicture}[
node1/.style = {circle,minimum size=23,draw},
node2/.style = {circle,minimum size=23,draw,fill=white!75!black},
node3/.style = {circle,minimum size=23,draw,fill=white!50!black},
noderect/.style = {rectangle,minimum size=20,draw},
edge1/.style = {>=latex,thick},
edgedashed/.style = {>=latex,thick,dashed},
edge2/.style = {>=latex,thick,blue},
edge3/.style = {>=latex,thick,red}
]
\node[node1] at (0,0)(n12){1 2};
\node[node2] at (2,0)(n34){3 4};

\draw [->,edge1] (n12) edge[loop left,looseness=5] (n12);
\draw [->,edge1] (n34) edge[loop left,looseness=5] (n34);

%\draw [->,edge1](n1) -- (n23);
%\draw [->,edge1](n23) -- (n4);
%\draw [->,edge1](n4) -- (n1);
%\DoubleLine{n1}{n2}{->,edgedashed}{}{->,edge1}{}

%\node (minus123) at ($(n1)!0.4!(n2) + (0,0.3)$) {(-)};

%\DoubleLine{n1}{n2}{<-,edge1}{}{->,edge1}{}
%\DoubleLine{n1}{n3}{<-,edge1}{}{->,edge1}{}
%\DoubleLine{n2}{n3}{<-,edge1}{}{->,edge1}{}

\end{tikzpicture} 
			\caption{Quotient network.}
			\label{fig:A_bp_not_invariant_S_to_W_quotient}
		\end{subfigure}
		\caption{Example of a quotient over a partition that is not $\mathcal{R}^{-} $-invariant.}	
		\label{fig:A_bp_not_invariant_S_to_W_exmp}	
	\end{figure}	
	\hfill $ \square $
\end{exmp}
We now present examples where \tref{table:quotient_bp_r_invariant} does indeed apply.
\begin{exmp}
	\label{exmp:quotient_lattice_example}
	Consider the network in \fref{fig:quotient_net_example_main} and its respective lattice of balanced partitions $ \Lambda_{\mathcal{G}} $ in \fref{fig:quotient_lattice_example_main}. We define the quotient networks $ \mathcal{Q}_1 := \mathcal{G}/\bp_1 $, $ \mathcal{Q}_2 := \mathcal{G}/\bp_2 $ over the balanced partitions $ \bp_1 = \{\{1,2\},\{3\},\{4\}\} $ and $ \bp_2 = \{\{1\},\{2\},\{3,4\}\} $, respectively. Both $ \bp_1 $ and $ \bp_2 $ are $\mathcal{R}^{-} $-invariant, therefore \tref{table:quotient_bp_r_invariant} applies.\\
	The set of partitions in $ \Lambda_{\mathcal{G}} $ that are coarser than $ \bp_1 $ are $ \{\{1,2\},\{3\},\{4\}\} $ ($ \bp_1 $ itself) and $ \{\{1,2\},\{3,4\}\} $, which are both weak. In the lattice $ \Lambda_{\mathcal{Q}_1} $ these two partitions correspond to $ \bot_{\mathcal{Q}_1} $ and $ \{\{12\},\{3,4\}\} $, which are strong and rooted, respectively.\\
	The set of partitions in $ \Lambda_{\mathcal{G}} $ that are coarser than $ \bp_2 $ are $ \{\{1\},\{2\},\{3,4\}\} $ ($ \bp_2 $ itself) and $ \{\{1,2\},\{3,4\}\} $, which are rooted and weak respectively. In the lattice $ \Lambda_{\mathcal{Q}_2} $ these two partitions correspond to $\bot_{\mathcal{Q}_2} $ and $ \{\{1,2\},\{34\}\} $, which are strong and weak respectively.
	\begin{figure}[h]
		\centering
		\begin{subfigure}[t]{0.5\textwidth}
			\centering
			\begin{tikzpicture}[
node1/.style = {circle,minimum size=23,draw},
node2/.style = {circle,minimum size=23,draw,fill=white!75!black},
node3/.style = {circle,minimum size=23,draw,fill=white!50!black},
noderect/.style = {rectangle,minimum size=20,draw},
edge1/.style = {>=latex,thick},
edge2/.style = {>=latex,thick,blue},
edge3/.style = {>=latex,thick,red}
]
\node[node1] at (0,1.5)(n1){1};
\node[node1] at (0,0)(n2){2};
\node[node1] at (1.5,1.5)(n3){3};
\node[node1] at (1.5,0)(n4){4};

\draw [->,edge1](n1) -- (n3);
\draw [->,edge1](n1) -- (n4);

\draw [->,edge1](n2) -- (n3);
\draw [->,edge1](n2) -- (n4);

%\node (minus123) at ($(n1)!0.4!(n2) + (0,0.3)$) {(-)};

%\DoubleLine{n1}{n2}{<-,edge1}{}{->,edge1}{}
%\DoubleLine{n1}{n3}{<-,edge1}{}{->,edge1}{}
%\DoubleLine{n2}{n3}{<-,edge1}{}{->,edge1}{}

\end{tikzpicture}
			\caption{Network $ \mathcal{G} $.}
			\label{fig:quotient_net_example_main}
		\end{subfigure}
		\begin{subfigure}[t]{0.4\textwidth}
			\centering
			\begin{tikzpicture}[
node1/.style = {circle,minimum size=23,draw},
node2/.style = {circle,minimum size=23,draw,fill=white!75!black},
node3/.style = {circle,minimum size=23,draw,fill=white!50!black},
pweak/.style = {rectangle,minimum size=20,draw,draw,fill=white!50!black},
proot/.style = {rectangle,minimum size=20,draw,draw,fill=white!75!black},
pstrong/.style = {rectangle,minimum size=20,draw},
noderect/.style = {rectangle,minimum size=20,draw},
edge1/.style = {>=latex,thick},
edgedash/.style = {>=latex,thick,dashed},
edge2/.style = {>=latex,thick,blue},
edge3/.style = {>=latex,thick,red}
]

\def\layer{1.3}
\def\del{0.75}

\node[pstrong] at (0,0) (bot){$ \bot_{\mathcal{G}} $};

\node[pweak] at ({-\del},{\layer} ) (12){$ 12 $};
\node[proot] at ({\del},{\layer} ) (34){$ 34 $};

\node[pweak] at (0,{2*\layer} ) (12_34){$ 12/34 $};

\draw [-,edge1](bot) -- (12);
\draw [-,edge1](bot) -- (34);

\draw [-,edge1](12) -- (12_34);
\draw [-,edge1](34) -- (12_34);

\end{tikzpicture} 
			\caption{Lattice $ \Lambda_{\mathcal{G}} $.}
			\label{fig:quotient_lattice_example_main}
		\end{subfigure}
		
		\begin{subfigure}[t]{0.5\textwidth}
			\centering
			\begin{tikzpicture}[
node1/.style = {circle,minimum size=23,draw},
node2/.style = {circle,minimum size=23,draw,fill=white!75!black},
node3/.style = {circle,minimum size=23,draw,fill=white!50!black},
noderect/.style = {rectangle,minimum size=20,draw},
edge1/.style = {>=latex,thick},
edge2/.style = {>=latex,thick,blue},
edge3/.style = {>=latex,thick,red}
]
\node[node1] at (0,{1.5/2})(n12){12};
\node[node1] at (1.5,1.5)(n3){3};
\node[node1] at (1.5,0)(n4){4};

\draw [->,edge1](n12) -- (n3);
\draw [->,edge1](n12) -- (n4);

\node (dois12_3) at ($(n12)!0.4!(n3) + (0,0.3)$) {2};
\node (dois12_4) at ($(n12)!0.4!(n4) + (0,-0.3)$) {2};

%\node (minus123) at ($(n1)!0.4!(n2) + (0,0.3)$) {(-)};

%\DoubleLine{n1}{n2}{<-,edge1}{}{->,edge1}{}
%\DoubleLine{n1}{n3}{<-,edge1}{}{->,edge1}{}
%\DoubleLine{n2}{n3}{<-,edge1}{}{->,edge1}{}

\end{tikzpicture}
			\caption{Network $ \mathcal{Q}_1:=\mathcal{G}/\{\{1,2\},\{3\},\{4\}\} $.}
			\label{fig:quotient_net_example_Q1}
		\end{subfigure}
		\begin{subfigure}[t]{0.4\textwidth}
			\centering
			\begin{tikzpicture}[
node1/.style = {circle,minimum size=23,draw},
node2/.style = {circle,minimum size=23,draw,fill=white!75!black},
node3/.style = {circle,minimum size=23,draw,fill=white!50!black},
pweak/.style = {rectangle,minimum size=20,draw,draw,fill=white!50!black},
proot/.style = {rectangle,minimum size=20,draw,draw,fill=white!75!black},
pstrong/.style = {rectangle,minimum size=20,draw},
noderect/.style = {rectangle,minimum size=20,draw},
edge1/.style = {>=latex,thick},
edgedash/.style = {>=latex,thick,dashed},
edge2/.style = {>=latex,thick,blue},
edge3/.style = {>=latex,thick,red}
]

\def\layer{1.3}
\def\del{0.75}

\node[pstrong] at (0,0) (bot){$ \bot_{\mathcal{Q}_1} $};
\node[proot] at (0,{\layer} ) (12_34){$ (12)/34 $};

\draw [-,edge1](bot) -- (12_34);

\end{tikzpicture} 
			\caption{Lattice $ \Lambda_{\mathcal{Q}_1} $.}
			\label{fig:quotient_lattice_example_Q1}
		\end{subfigure}
		
		\begin{subfigure}[t]{0.5\textwidth}
			\centering
			\begin{tikzpicture}[
node1/.style = {circle,minimum size=23,draw},
node2/.style = {circle,minimum size=23,draw,fill=white!75!black},
node3/.style = {circle,minimum size=23,draw,fill=white!50!black},
noderect/.style = {rectangle,minimum size=20,draw},
edge1/.style = {>=latex,thick},
edge2/.style = {>=latex,thick,blue},
edge3/.style = {>=latex,thick,red}
]
\node[node1] at (0,1.5)(n1){1};
\node[node1] at (0,0)(n2){2};
\node[node1] at (1.5,{1.5/2})(n34){34};

\draw [->,edge1](n1) -- (n34);
\draw [->,edge1](n2) -- (n34);

%\node (minus123) at ($(n1)!0.4!(n2) + (0,0.3)$) {(-)};

%\DoubleLine{n1}{n2}{<-,edge1}{}{->,edge1}{}
%\DoubleLine{n1}{n3}{<-,edge1}{}{->,edge1}{}
%\DoubleLine{n2}{n3}{<-,edge1}{}{->,edge1}{}

\end{tikzpicture}
			\caption{Network $ \mathcal{Q}_2:=\mathcal{G}/\{\{1\},\{2\},\{3,4\}\} $.}
			\label{fig:quotient_net_example_Q2}
		\end{subfigure}
		\begin{subfigure}[t]{0.4\textwidth}
			\centering
			\begin{tikzpicture}[
node1/.style = {circle,minimum size=23,draw},
node2/.style = {circle,minimum size=23,draw,fill=white!75!black},
node3/.style = {circle,minimum size=23,draw,fill=white!50!black},
pweak/.style = {rectangle,minimum size=20,draw,draw,fill=white!50!black},
proot/.style = {rectangle,minimum size=20,draw,draw,fill=white!75!black},
pstrong/.style = {rectangle,minimum size=20,draw},
noderect/.style = {rectangle,minimum size=20,draw},
edge1/.style = {>=latex,thick},
edgedash/.style = {>=latex,thick,dashed},
edge2/.style = {>=latex,thick,blue},
edge3/.style = {>=latex,thick,red}
]

\def\layer{1.3}
\def\del{0.75}

\node[pstrong] at (0,0) (bot){$ \bot_{\mathcal{Q}_2} $};
\node[pweak] at (0,{\layer} ) (12_34){$ 12/(34) $};

\draw [-,edge1](bot) -- (12_34);

\end{tikzpicture} 
			\caption{Lattice $ \Lambda_{\mathcal{Q}_2} $.}
			\label{fig:quotient_lattice_example_Q2}
		\end{subfigure}
		\caption{Lattices of balanced partitions of a network and its quotients.}
		\label{fig:quotient_lattice_example}
	\end{figure}
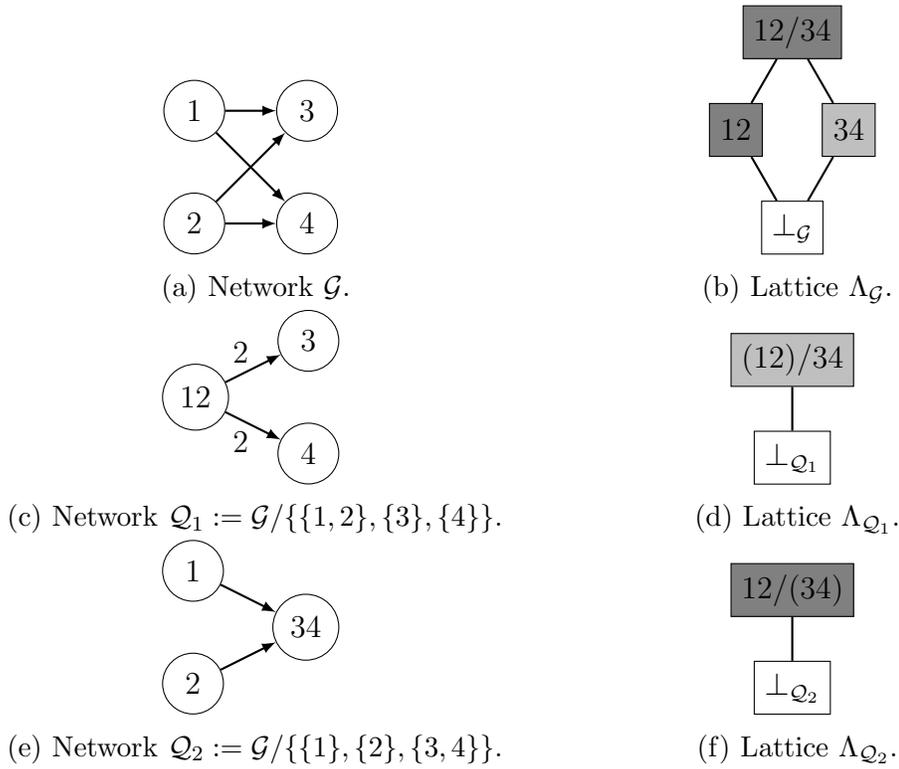
	\hfill $ \square $
\end{exmp}
We can always quotient a network over the trivial partition. That is, $ \mathcal{Q} := \mathcal{G}/\bot_{\mathcal{G}} $. Consider we encode $ \bot_{\mathcal{G}} $ through the identity mapping.
In such case we have $ \mathcal{G}=\mathcal{Q} $ and $  L_{F_{\mathcal{Q}}} = L_{F_{\mathcal{G}}} /\bot_{\mathcal{G}}  = L_{F_{\mathcal{G}}}$. Therefore every partition in $ L_{F_{\mathcal{G}}} $ maps to itself in $ L_{F_{\mathcal{Q}}} $. This implies the cases $ S \to S$, $ R \to R $ and $ W \to W $ in the left side of \tref{table:quotient_bp_r_invariant}.\\
On the other hand, for the case $ \mathcal{Q} := \mathcal{G}/\bp $ and $  L_{F_{\mathcal{Q}}} = L_{F_{\mathcal{G}}} / \bp$, for any $ \bp \in \Lambda_{\mathcal{G}} $ we have that $ \bp/\bp = \bot_{\mathcal{Q}}$. Since $ \bot_{\mathcal{Q}} $ is always strong in $ \mathcal{Q} $, this covers the cases $ S/R/W \to S $ in the left side of \tref{table:quotient_bp_r_invariant}. This means that most of the cases of \tref{table:quotient_bp_r_invariant} were forced. The remaining case $ W \to R $, was illustrated in \cref{exmp:quotient_lattice_example}. That is, the interest of this result lies in the fact that it excludes most of the non-forced cases.
\section{Conclusion}
In this paper we study the connection between the dynamics of a cell and its different types of  cumulative in-neighborhoods.\\
We first study the structure of general lattices of synchronism, which is the set of partitions that represent equality-based synchrony patterns that are invariant under a given subset of admissible functions.\\
We then analyze in a qualitative way the relation between the colors of a partition and the connectivity structure of the network.\\
This motivates the classification scheme developed in this work, which categorizes partitions into strong, rooted and weak. This scheme can be applied to generic partitions, without any assumptions such as being balanced, exo-balanced or any similar properties.\\
We study how these categories behave under the partition join operation $ (\vee) $, which corresponds to the intersection of synchrony patterns. We can think of these classes such that strong is a ``stronger'' property than rooted, and rooted is a ``stronger'' property than weak. Then under relatively tame assumptions ($ \mathcal{R}^{-} $-matched), the class of the output partition of a join is the weakest class in the input partitions.\\
We also relate the class of a partition $ \mathcal{A} $ in a network $ \mathcal{G} $ to its corresponding partition $ \mathcal{A}/\bp $ in the quotient network $ \mathcal{Q} := \mathcal{G}/\bp $. Then under certain assumptions ($ \bp $ being $ \mathcal{R}^{-} $-invariant), the class of $ \mathcal{A}/\bp $ is stronger or the same as the $ \mathcal{A} $.
\\
Multiple examples illustrating our classification scheme are provided.
\ack
This work was supported in part by the FCT Project RELIABLE (PTDC/EEI-AUT/3522/2020), funded by FCT/MCTES and in part by the National Science Foundation under Grant No.~ECCS-2029985. The work of P. Sequeira was supported by a Ph.D. Scholarship, grant SFRH/BD/119835/2016 from Fundação para a Ciência e a Tecnologia (FCT), Portugal (POCH program).
\section*{References}
\bibliographystyle{unsrt}
\bibliography{references}

\begin{thebibliography}{10}

\bibitem{stewart2003symmetry}
Ian Stewart, Martin Golubitsky, and Marcus Pivato.
\newblock Symmetry groupoids and patterns of synchrony in coupled cell
  networks.
\newblock {\em SIAM J. Appl. Dyn. Syst.}, 2(4):609--646, 2003.

\bibitem{golubitsky2005patterns}
Martin Golubitsky, Ian Stewart, and Andrei T\"{o}r\"{o}k.
\newblock Patterns of synchrony in coupled cell networks with multiple arrows.
\newblock {\em SIAM J. Appl. Dyn. Syst.}, 4(1):78--100, 2005.

\bibitem{golubitsky2006nonlinear}
Martin Golubitsky and Ian Stewart.
\newblock Nonlinear dynamics of networks: the groupoid formalism.
\newblock {\em Bull. Am. Math. Soc.}, 43(3):305--364, 2006.

\bibitem{sequeira2021commutative}
Pedro~M Sequeira, A~Pedro~Aguiar, and João Hespanha.
\newblock Commutative monoid formalism for weighted coupled cell networks and
  invariant synchrony patterns.
\newblock {\em SIAM J. Appl. Dyn. Syst.}, 20(3):1485--1513, 2021.

\bibitem{arenas2008synchronization}
Alex Arenas, Albert Díaz-Guilera, Jurgen Kurths, Yamir Moreno, and Changsong
  Zhou.
\newblock Synchronization in complex networks.
\newblock {\em Phys. Rep.}, 469(3):93--153, 2008.

\bibitem{dorfler2014synchronization}
Florian Dörfler and Francesco Bullo.
\newblock Synchronization in complex networks of phase oscillators: A survey.
\newblock {\em Automatica}, 50(6):1539--1564, 2014.

\bibitem{rodrigues2016kuramoto}
Francisco~A Rodrigues, Thomas K~DM Peron, Peng Ji, and Jürgen Kurths.
\newblock The {K}uramoto model in complex networks.
\newblock {\em Phys. Rep.}, 610:1--98, 2016.

\bibitem{aguiar2018synchronization}
Manuela A~D Aguiar and Ana Paula~S Dias.
\newblock {Synchronization and equitable partitions in weighted networks}.
\newblock {\em Chaos}, 28(7):073105, 2018.

\bibitem{neuberger2020invariant}
John~M Neuberger, N\'{a}ndor Sieben, and James~W Swift.
\newblock Invariant synchrony subspaces of sets of matrices.
\newblock {\em SIAM J. Appl. Dyn. Syst.}, 19(2):964--993, 2020.

\bibitem{aguiar2021synchrony}
Manuela Aguiar and Ana Dias.
\newblock Synchrony and antisynchrony in weighted networks.
\newblock {\em SIAM J. Appl. Dyn. Syst.}, 20(3):1382--1420, 2021.

\bibitem{stewart2007lattice}
Ian Stewart.
\newblock The lattice of balanced equivalence relations of a coupled cell
  network.
\newblock {\em Math. Proc. Camb. Philos. Soc.}, 143(1):165–183, 2007.

\bibitem{aguiar2014lattice}
Manuela A~D Aguiar and Ana Paula~S Dias.
\newblock The lattice of synchrony subspaces of a coupled cell network:
  Characterization and computation algorithm.
\newblock {\em J. Nonlinear Sci.}, 24:949--996, 2014.

\bibitem{moreira2015special}
C\'{e}lia~Sofia Moreira.
\newblock Special jordan subspaces and synchrony subspaces in coupled cell
  networks.
\newblock {\em SIAM J. Appl. Dyn. Syst.}, 14(1):253--285, 2015.

\bibitem{kamei2021reduced}
Hiroko Kamei and Haibo Ruan.
\newblock Reduced lattices of synchrony subspaces and their indices.
\newblock {\em SIAM J. Appl. Dyn. Syst.}, 20(2):636--670, 2021.

\bibitem{aldis2008polynomial}
John~W Aldis.
\newblock A polynomial time algorithm to determine maximal balanced equivalence
  relations.
\newblock {\em Int. J. Bifurc. Chaos Appl. Sci. Eng.}, 18(2):407--427, 2008.

\bibitem{sequeira2022decomposition}
Pedro~M Sequeira, Jo\~{a}o~P Hespanha, and A~Pedro~Aguiar.
\newblock Decomposition of admissible functions in weighted coupled cell
  networks.
\newblock {\em SIAM J. Appl. Dyn. Syst.}, 22(2):1114--1152, 2023.

\bibitem{enderton1977elements}
Herbert~B Enderton.
\newblock {\em Elements of Set Theory}.
\newblock Academic press, 1977.

\bibitem{davey_priestley_2002}
B~A Davey and H~A Priestley.
\newblock {\em Introduction to Lattices and Order}.
\newblock Cambridge University Press, 2 edition, 2002.

\bibitem{tarjan1972depth}
Robert Tarjan.
\newblock Depth-first search and linear graph algorithms.
\newblock {\em SIAM J. Comput.}, 1(2):146--160, 1972.

\bibitem{aguiar2017patterns}
Manuela A~D Aguiar, Ana Paula~S Dias, and Flora Ferreira.
\newblock {Patterns of synchrony for feed-forward and auto-regulation
  feed-forward neural networks}.
\newblock {\em Chaos}, 27(1):013103, 2017.

\end{thebibliography}

\end{document}